\documentclass[11pt, a4paper]{article}

\usepackage{geometry}
\usepackage{amsmath}
\usepackage{amsthm}
\usepackage{amssymb}
\usepackage{bbm} 
\usepackage{color}
\usepackage{xcolor}
\usepackage{graphicx}
\usepackage[toc,page]{appendix}
\usepackage{authblk}
\usepackage[unicode,colorlinks=true,linkcolor=black,citecolor=black,urlcolor=black,pagebackref=false]{hyperref} 

\newcommand{\R}{\mathbb{R}}
\newcommand{\1}{\mathbbm{1}}
\renewcommand{\P}{\mathbb{P}}
\newcommand{\F}{\mathcal{F}}
\newcommand{\E}{\mathbb{E}}

\DeclareMathOperator\supp{supp}
\DeclareMathOperator\id{id}

\newtheorem{theorem}{Theorem}
\newtheorem{proposition}[theorem]{Proposition}
\newtheorem{lemma}[theorem]{Lemma}
\newtheorem{corollary}[theorem]{Corollary}
\newtheorem{definition}[theorem]{Definition}

\theoremstyle{definition}
\newtheorem{example}[theorem]{Example}
\newtheorem{remark}[theorem]{Remark}

\title{Failure of the Markov property for stochastic Volterra equations}

\author[1]{Martin Friesen \thanks{Email: martin.friesen@dcu.ie}}
\author[2]{Stefan Gerhold \thanks{Email: sgerhold@fam.tuwien.ac.at}}
\author[2]{Kristof Wiedermann \thanks{Email: kristof.wiedermann@tuwien.ac.at}}
\affil[1]{\small School of Mathematical Sciences, Dublin City University}
\affil[2]{\small Institute of Statistics and Mathematical Methods in Economics, TU Wien}

\date{\today}
\numberwithin{equation}{section}
\numberwithin{theorem}{section}

\begin{document}

\maketitle

\begin{abstract}
\noindent Memory-driven stochastic dynamics arise naturally in many applications, and stochastic Volterra equations (SVEs) offer a flexible framework for modeling such systems. Their convolution structure with Volterra kernels endows the dynamics with a formal path-dependency, which suggests the failure of the Markov property. While this has previously been rigorously established only for Gaussian Volterra processes, by constructing nondegenerate admissible perturbations through Markovian lifts, we prove that also general SVEs with Hölder-continuous coefficients do not possess the Markov property for a broad class of Volterra kernels. Moreover, we show that the associated Markovian lift is, in general, necessarily infinite-dimensional. These observations reflect the intrinsic infinite-dimensionality of memory effects in SVEs and underscore the need for analytical and probabilistic tools beyond the classical Markovian framework.
\end{abstract}
\vspace{0.2cm}
{\small \textbf{Keywords:} stochastic Volterra equation; path-dependence; Markov property; Markovian lift; nondegenerate perturbation; admissible perturbation; regularity results.\vspace{0.2cm}\newline
}

\section{Introduction}

\subsection{Literature and Motivation}

Many systems that shape our world, from cells and ecosystems to financial markets and weather patterns, evolve in fundamentally random ways. Assuming that the future depends only on the present leads to the rich and well-developed theory of Markov processes, whose elegance and generality have established it as a cornerstone of modern stochastic modeling. In practice, however, the present does not always determine the future distributions, and memory arises across disciplines. For instance, it appears in biology through age-structured populations \cite{MR3024808}, in physics and chemistry through anomalous diffusions~\cite{MR3793188} and materials with stress retention \cite{MR1238939}, in computer science through recurrent architectures such as LSTMs, and in finance through rough volatility \cite{pricingroughvol, volisrough} and electricity spot prices \cite{BENNEDSEN2017301}. Memory effects also emerge naturally from model reduction (e.g.\ through the Mori–Zwanzig formalism \cite{MR4565387, MR4396389}) or are deliberately introduced to capture long- and short-range temporal dependencies in data.

Within this landscape, stochastic Volterra equations (SVEs) have assumed a central role. They form a highly active area of research, particularly driven by applications to rough volatility models and electricity spot-price dynamics, where roughness at small time scales plays a crucial role \cite{MR3079297, rvol, doi:10.1137/23M1617370}, but are also frequently used in mathematical physics for modeling fractional dynamics. SVEs encode memory through the Volterra kernel, which determines how past trajectories influence future dynamics on infinitesimal scales, while simultaneously allowing for flexible modeling of rough behavior. This flexibility, however, comes at the cost of significant mathematical complexity. For singular kernels, SVEs fail to be semimartingales, while for kernels not equal to the exponential function, they introduce a formal path-dependency into the dynamics.

The intrinsic path-dependence of SVEs gives rise to both conceptual and practical challenges, strongly suggesting the failure of the Markov property. For instance, simulation schemes must access the entire past trajectory of the driving noise at each step \cite{alfonsi2022}, a difficulty that motivates finite-dimensional Markovian approximations \cite{multifactorapproxroughvol, MR4521278}. Limit distributions and stability properties exhibit fundamentally new phenomena, such as power-law convergence rates for linear models and multiple equilibria, that extend well beyond the classical Markovian framework \cite{MR4503737, FrJi22, JPS22}. The corresponding distributions often display heavy tails and long-range dependence, giving rise to limit theorems outside the Gaussian domain of attraction \cite{MR3561100}. At the same time, analyzing distributional properties through path-dependent Kolmogorov equations remains highly nontrivial \cite{MR4047986}. Path-dependence also manifests explicitly after suitable transformations of the process \cite{BFK25}, where it plays a key role in establishing equivalence of laws and studying maximum likelihood estimation. In rough volatility models, it is further reflected in hedging strategies that depend on the past volatility or the forward variance curve \cite{roughhestonhedging, MR3968277}. Taken together, these observations provide compelling evidence for a breakdown of the Markov property. Indeed, it has become almost folklore that stochastic Volterra processes are inherently non-Markovian, although rigorous proofs remain scarce outside Gaussian Volterra processes \cite{BaGe20}, or specific one-dimensional settings \cite{FGW25}.

Recall that a stochastic process $X$ with state space $D$ defined on a filtered probability space $(\Omega, \mathcal{F}, (\mathcal{F}_t)_{t \in\R_+}, \P)$, has the Markov property with respect to the filtration $(\mathcal{F}_t)_{t \in\R_+}$ if for all $0 \leq t < T$ and each bounded and measurable function $f:D\longrightarrow\R$ on $D$ it holds
\begin{align}\label{eq: Markov property}
    \E[f(X_T) \, | \, \F_t] = \E[ f(X_T) \, | \, X_t] \ \ \text{ a.s.}
\end{align}
Clearly, the Markov property is preserved when passing to a coarser filtration with respect to which $X$ is adapted. We also emphasize that \eqref{eq: Markov property} formulates the Markov property only for a fixed (not necessarily deterministic) initial state $X_0 \sim \nu$. 

To illustrate the failure of the Markov property, let us consider the Volterra square-root process, which is a Volterra analogue of the classical square-root diffusion process. It is defined as the unique nonnegative weak solution of the SVE
\begin{equation}\label{eq:roughHeston}
X_t = x_0 + \int_0^t K(t-s)(b + \beta X_s)\, \mathrm{d}s + \sigma \int_0^t K(t-s)\sqrt{X_s}\, \mathrm{d}B_s,
\quad t \in \R_+,
\end{equation}
where $B$ is a standard Brownian motion, $\beta\in\R$, and $x_0, b, \sigma \in \R_+$, see \cite[Theorem 6.1]{AbiJaLaPu19}. In the special case of the fractional Riemann-Liouville kernel $K(t) = t^{H-1/2} / \Gamma(H+\tfrac{1}{2})$, where $H \in (0,1/2)$ denotes the Hurst parameter, one obtains the rough Cox–Ingersoll–Ross process central for rough volatility modeling \cite{roughhestonhedging, roughhestcharfct}. The path-dependence in \eqref{eq:roughHeston} becomes visible in the structure of its conditional expectations (see \cite[Theorem 4.5\hspace{0.03cm}(i)]{AbiJaLaPu19}), which for $0 \leq t < T$ take the form
\begin{align} \nonumber
\mathbb{E}[X_T\, | \, \mathcal{F}_t] &= b \int_0^{T-t}E_K(s)\, \mathrm{d}s - \Pi_{T-t}(t)\hspace{0.02cm}x_0
\\ &\quad \,  + (\Delta_{T-t}E_K\ast L)(0)\,X_t + \int_{[0,t]} X_{t-s}\,\Pi_{T-t}(\mathrm{d}s),\label{eq: 8}
\end{align}
where the operator $\Delta_h$ is for $h\in\R_+$ defined by $\Delta_h f = f(\cdot+h)$. Here $E_K$ is the unique solution of $E_K(t) = K(t) + \beta \int_0^t K(t-s)E_K(s)\,\mathrm{d}s$, $L$ denotes the resolvent of the first kind of $K$, and $\Pi_z(t) = (\Delta_z E_K\ast L)(t) - \Delta_z(E_K\ast L)(t)$ for $t,z \in\R_+$. Since the right-hand side of \eqref{eq: 8} formally depends on the entire past $(X_s)_{s\le t}$, it is natural to conclude that $X$ is not Markov. However, this intuition alone is insufficient. For example, if $K \equiv 1$, then $\Pi_z \equiv 0$, and the path-dependence vanishes. From this perspective, a rigorous proof of non-Markovianity requires an argument that $\Pi_z$ is genuinely non-constant and that the integral $\int_{[0,t]} X_{t-s}\,\Pi_{T-t}(\mathrm{d}s)$ cannot be  $X_t$-measurable. This problem is addressed in Subsection~\ref{section:affine case} based on the absolute continuity of the law for a class of Volterra Ito-processes with nondegenerate Volterra kernels.

\subsection{Discussion of the results}

We study the failure of the Markov property for the general class of convolution-type stochastic Volterra equations (SVEs) on $\R^d$ of the form
\begin{equation}\label{eq:generalSVIE}
  X_t = g(t)+\int_{0}^{t}K^b(t-s) \hspace{0.03cm}b(s,X_{s})\, \mathrm{d}s + \int_0^t K^{\sigma}(t-s)\hspace{0.03cm}\sigma(s,X_s)\, \mathrm{d}B_s,\quad t \in \R_{+},
\end{equation} 
where $B$ denotes an $m$-dimensional Brownian motion on some filtered probability space $(\Omega,\mathcal{F},(\mathcal{F}_t)_{t \in\R_+},\mathbb{P})$, $b: \R_+\times\R^d \longrightarrow \R^d$ the drift, and $\sigma: \R_+\times\R^d \longrightarrow \R^{d \times m}$ the diffusion coefficients. The Volterra kernels $K^b, K^{\sigma} \in L_{\mathrm{loc}}^2(\R_+; \R^{d \times d})$ encode the memory structure, while $g:\R_+\longrightarrow\R^d$ is an $\mathcal{F}_0$-measurable input that is naturally linked to the Volterra kernels. For classical SDEs with $K^b=K^{\sigma} \equiv 1$, the choice $g\equiv x_0\in\R^d$ is commonly used, while for general Volterra kernels, the class of admissible input curves $g$ will be discussed in
Section~\ref{section:markovianliftapproach}, see \eqref{eq: admissible input curves} therein. 

Concerning our main assumptions, we suppose that the drift and diffusion coefficients are H\"older continuous in space, uniformly in the time variable. Furthermore, the corresponding Volterra kernels $K^b, K^{\sigma}$ shall satisfy a standard $L^2$-increment condition (see~\eqref{eq: K increment} in condition (A)). Such an assumption guarantees that solutions of \eqref{eq:generalSVIE} have H\"older continuous sample paths. Finally, we suppose that the Gram matrix function associated with $K^{\sigma} (K^{\sigma})^{\intercal}$ is positive definite and its smallest eigenvalue has power-type lower asymptotics as $t \searrow 0$. Under these assumptions and a minor technical balance condition, we establish as a combination of Sections \ref{section:abs cont} -- \ref{section:failure} the failure of the Markov property in an abstract setting. More precisely, any weak solution $X$ of \eqref{eq:generalSVIE} does not fulfill the Markov property, provided that $\P[X_{t_0} \in \Gamma_{\sigma,t_0}] >0$ holds for some $t_0 > 0$ with
\begin{align}\label{eq: Gamma sigma}
            \Gamma_{\sigma,t} := \big\{ x \in \R^d \ : \ \mathrm{det}\big(\sigma(t,x)\hspace{0.02cm}\sigma(t,x)^{\intercal}\big) \neq 0 \big\}, \quad t>0.
\end{align}
The assumption $\P[X_{t_0} \in \Gamma_{\sigma,t_0}] >0$ merely states that the diffusion term is nondegenerate with positive probability, and hence is not a major restriction. For example, if $d = 1$ and $\mathbb{P}[X_t\in\Gamma_{\sigma,t}]=0$ holds for each $t > 0$, it would simply imply $\sigma(\cdot,X) = 0$ up to indistinguishability. The process $X$ would then be entirely $\F_0$-measurable, as its dynamics would be solely governed by $g$. 

In Section~\ref{section: examples} we verify our abstract conditions for the case of diagonal Volterra kernels with regularly varying entries. To illustrate our results, let us consider below the scalar-valued case under the additional condition that $g \equiv x_0\in\R^d$ is deterministic. Our first result covers the simpler case where $K^b = K^{\sigma}$.

\begin{theorem}[$K^b = K^\sigma$ scalar-valued]\label{thm: intro 1}
    Let $K^b = K^{\sigma} = k \hspace{0.03cm}\mathrm{id}_{\R^d}$ where $k$ is smooth on $\R_+^*$, $k'$ is eventually monotone as $t \searrow 0$, and $\int_{1}^{\infty} \big| k'(t)\big|^2\, \mathrm{d}t < \infty$. Assume that $k$ is regularly varying with index $\rho = H - \frac{1}{2}$, where $H \in (0,1)$, $k^{(m)}$ is locally square-integrable and completely monotone with Bernstein measure $\mu$ for some $m \in \mathbb{N}_0$, and there exists a sequence $(\lambda_n)_{n \geq 1} \subseteq \mathrm{supp}(\mu) \cap \R_+^*$ of pairwise distinct elements such that 
    \begin{align}\label{eq: 14}
        \sum_{n=1}^{\infty} \frac{1}{\lambda_n} = \infty.
    \end{align}
    Let $b, \sigma$ be $\chi_b, \chi_{\sigma} \in (0,1]$-H\"older continuous in $x$ uniformly in $t$, see \eqref{eq:uniformHölder}, and consider a deterministic $g \equiv x_0$. Then any weak solution $X$ of \eqref{eq:generalSVIE} is not a Markov process with respect to the filtration generated by $X$ and the associated Brownian motion $B$, provided that there exists $t_0 > 0$ with $\P[ \mathrm{det}(\sigma(t_0,X_{t_0})\hspace{0.02cm}\sigma(t_0,X_{t_0})^{\intercal}) \neq 0] > 0$. 
\end{theorem}

Theorem \ref{thm: intro 1} is a scalar version of Theorem \ref{thm: time shift deterministic} applied for constant $g \equiv x_0$. It covers a wide range of regularly varying kernels for which $\mu$ has either a nontrivial absolutely continuous component with respect to the Lebesgue measure (thus including regular kernels), or its support is sufficiently rich as $x \to \infty$ (covering singular or slowly varying kernels). In particular, our results contain the parametric class 
\begin{align}\label{eq: parametric}
k(t) = c\hspace{0.02cm} (t+\varepsilon_0)^{H - \frac{1}{2}}e^{- \lambda (t+\varepsilon_1)}\log(1 + (t + \varepsilon_2)^{\alpha}),
\end{align}
where $\lambda, \varepsilon_0, \varepsilon_1, \varepsilon_2 \geq 0$, $c > 0$, and $H \in (0,\frac{1}{2}]$, $\alpha \in [-1,0]$ in the completely monotone case. Moreover, for $\lambda=0$, it covers the case of Bernstein functions with $H \in (\frac{1}{2}, 1)$, $\varepsilon_0>0$ and $\alpha=0$ or $H=\frac{1}{2}$, $\varepsilon_2>0$ and $\alpha \in (0,1]$. Condition \eqref{eq: 14} requires that either $H \neq \frac{1}{2}$ or $\alpha \neq 0$. Note also that for shifted kernels with $\varepsilon_0 > 0$, hyperrough cases with $H < 0$ are included. As a concrete example, Theorem \ref{thm: intro 1} covers the Volterra square-root process~\eqref{eq:roughHeston} for regularly varying Volterra kernels.

Next, we formulate an analogue of Theorem~\ref{thm: intro 1} for scalar-valued Volterra kernels $K^b = k^b\hspace{0.02cm} \mathrm{id}_{\R^d}$ and $K^{\sigma} = k^{\sigma}\hspace{0.02cm}\mathrm{id}_{\R^d}$ that are not necessarily identical. In such a case, additional technical conditions that relate $k^b$ with $k^{\sigma}$ need to be imposed.

\begin{theorem}[general scalar-valued case]\label{thm: intro}
    Let $K^b = k^b\hspace{0.03cm} \mathrm{id}_{\R^d}$ and $K^{\sigma} = k^{\sigma}\hspace{0.03cm}\mathrm{id}_{\R^d}$, where $k^b, k^{\sigma}$ are continuously differentiable on $\R_+^*$, $(k^b)', (k^{\sigma})'$ are eventually monotone as $t \searrow 0$, and
    \[
        \int_{1}^{\infty} \left(\big| (k^b)'(t)\big|^2+\big| (k^{\sigma})'(t)\big|^2\right)\, \mathrm{d}t < \infty.
    \]
    Suppose that there exist $H^b, H^{\sigma} \in (0,1)$ with $|H^b - H^{\sigma}| < \frac{1}{2}$, $H^{\sigma} \neq \frac{1}{2}$ such that $k^a$ is regularly varying in $t = 0$ with index $\rho_a = H^a - \frac{1}{2}$ for $a \in \{b,\sigma\}$, and there exists $C > 0$ such that
    \begin{align}\label{eq: 15}
        \left|\frac{k^a(tu)}{k^a(t)}\right| \leq C u^{(H^a - \frac{1}{2})_+}, \qquad t \in (0,1), \ u \geq 1, \ a \in \{b,\sigma\}.
    \end{align}
    Let $b, \sigma$ be $\chi_b, \chi_{\sigma} \in (0,1]$-H\"older continuous in $x$ uniformly in $t$, see \eqref{eq:uniformHölder}, and let $g \equiv x_0$ be deterministic. Finally, suppose that the following balance conditions hold 
    \begin{align}\label{eq: intro}
        \max\{ H^b, H^{\sigma}\} < \frac{1}{2} + \chi \min\{H^b, H^{\sigma}\}, \ \text{ where }\, \chi = \begin{cases} \chi_{\sigma}, &  H^{\sigma} < H^b
        \\ \min\left\{ \chi_{\sigma}, \frac{1 + \chi_b}{2}\right\}, &  H^b \leq H^{\sigma}.
        \end{cases}
    \end{align}   
    Then any weak solution $X$ of \eqref{eq:generalSVIE} is not a Markov process with respect to the filtration generated by $X$ and the associated Brownian motion $B$, provided that there exists $t_0 > 0$ with $\P[ \mathrm{det}(\sigma(t_0,X_{t_0})\hspace{0.02cm}\sigma(t_0,X_{t_0})^{\intercal}) \neq 0] > 0$. 
\end{theorem}

Again, Theorem \ref{thm: intro} is the scalar version of Theorem \ref{theorem:NonMarkovSec5MostAbstract} applied for $g \equiv x_0$. In contrast to Theorem \ref{thm: intro 1}, here we impose the asymptotic scaling property \eqref{eq: 15}. Furthermore, condition \eqref{eq: intro} is a technical requirement ensuring the absolute continuity of the corresponding Volterra Ito-process arising from admissible perturbations. It is automatically fulfilled when $H^b \in (0,\frac{1}{2}]$ and $H^{\sigma} \in (0,\frac{1}{2})$. Furthermore, \eqref{eq: intro} is equivalent to
\[
    | H^b - H^{\sigma}| < \frac{1}{2} - \left( 1 - \chi \right) \min\{H^b, H^{\sigma}\}, 
\]
which shows that $H^b, H^{\sigma}$ are required to be sufficiently close to each other. As before, these conditions are satisfied for a broad class of Volterra kernels including kernels $k^b$, $k^{\sigma}$ of the form \eqref{eq: parametric} with parameters $c_{a}\neq 0$, $\lambda_a,\varepsilon_{0,a},\varepsilon_{1,a},\varepsilon_{2,a}\ge 0$, $H_a\in (0,\frac{1}{2}]$, $\alpha_a\in [-1,0]$ for $a\in\{b,\sigma\}$. Using the notation of Theorem \ref{thm: intro}, $H^{\sigma} \neq \frac{1}{2}$ is guaranteed for $H_{\sigma} \neq \frac{1}{2}$ and $\varepsilon_{0,\sigma}= 0$. Moreover, under \eqref{eq: intro}, Theorem \ref{thm: intro} also applies to general $H_a\in (0,1)$ provided that, in view of \eqref{eq: 15}, $\varepsilon_{0,a}=0$ holds for each $a\in\{b,\sigma\}$ with $H_a>1/2$.

Finally, let us remark that for general \textit{H\"older continuous} coefficients, our density method from Section 2 appears to be optimal, whence \eqref{eq: intro} cannot be substantially relaxed. For smooth coefficients $b,\sigma$, however, alternative techniques based on Malliavin calculus may yield further improvements on~\eqref{eq: intro}.

Our contribution is complementary to the accompanying work~\cite{FGW25}, which excludes the \textit{stronger} time-homogeneous Markov property for one-dimensional SVEs. In \cite{FGW25}, computational moment methods and a small-time CLT reduction show that, for a broad class of kernels and affine drifts, the time-homogeneous Markov property is essentially equivalent to 
\begin{align}\label{eq: exponential}
    \exists \lambda \in \R: \qquad K(t) = K(0)\hspace{0.02cm}e^{\lambda t}, \quad t \in \R_+.
\end{align}
In contrast, the present work addresses the more general (time-inhomogeneous) Markov property \eqref{eq: Markov property} for $d$-dimensional SVEs with potentially time-inhomogeneous coefficients. We provide sufficient conditions for its failure, but do not address the necessity of \eqref{eq: exponential}. Moreover, since our method relies on nondegenerate admissible perturbations, it cannot cover the case $K^{\sigma} \equiv 1$, which remains accessible to the computational method of \cite[Theorem 2.1]{FGW25} as it does not rely on the structure of the diffusion part. Conversely, our results do apply when $K^b \equiv 1$, which is outside the scope of \cite{FGW25}. Finally, \cite{FGW25} treats the full range $H > 0$, whereas here we restrict to $H \in (0,1)$ to avoid more sophisticated constructions of the Markovian lift. From this perspective, each approach covers cases that the other does not.

\subsection{Admissible nondegenerate perturbations and methodology} 

To prove the failure of the Markov property, let us suppose that $X$ would satisfy \eqref{eq: Markov property}. We show that this leads to a contradiction, whenever one can construct a one-dimensional process $(Z_t)_{t \in \R_+}$ such that~$Z_t$ is measurable with respect to the sigma algebra generated by $X_t$ for $t > 0$, and $(X_t, Z_t)$ is absolutely continuous to the Lebesgue measure. In the specific case of the rough Cox-Ingersoll-Ross process \eqref{eq:roughHeston}, the relation \eqref{eq: 8} offers the natural candidate $Z_t = \int_{[0,t]} X_{t-s}\,\Pi_{T-t}(\mathrm{d}s)$. Beyond this illustrative case, we construct $Z$ as a Volterra Ito-process of the form
\[
    Z_t = \widetilde{g}(t) + \int_0^t \widetilde{k}^b(t-s)\hspace{0.02cm}b(s,X_s)\, \mathrm{d}s + \int_0^t \widetilde{k}^{\sigma}(t-s)\hspace{0.02cm}\sigma(s,X_s)\, \mathrm{d}B_s,\quad t\in\R_+.
\]
Here $\widetilde{g}$ is $\F_0$-measurable, and $\widetilde{k}^b, \widetilde{k}^{\sigma}$ are perturbations of the Volterra kernels $K^b, K^{\sigma}$.

In Section~\ref{section:abs cont}, we prove that the laws of the joint random vectors $(X_t, Z_t)$ are absolutely continuous with respect to the Lebesgue measure restricted to $\Gamma_{\sigma, t} \times \R$. For this purpose, we extend the method from \cite[Section~3]{F24} to also accommodate anisotropic regularity (see Proposition \ref{proposition:VolterraItoanisotropdensityperturbations} in the appendix). In this way, we also cover non-scalar kernels with different orders of regularity. Our central assumption is based on a non-determinism condition as introduced, e.g., in \cite{MR4488556, MR4342752} for Volterra-L\'{e}vy and Gaussian Volterra processes, and in \cite{F24} for Volterra Ito-processes. Expressed via Gram matrices, this condition requires that $G(h) = \int_0^h \widetilde{K}^{\sigma}(r)\widetilde{K}^{\sigma}(r)^{\intercal}\, \mathrm{d}r$, where $\widetilde{K}^{\sigma} = (K^{\sigma}, \widetilde{k}^{\sigma})^{\intercal}$, is positive definite, and that its smallest eigenvalue satisfies the lower bound
\[
    \liminf_{h \searrow 0} h^{-2\gamma_*}\lambda_{\min}(G(h)) > 0
\]
for some $\gamma_* > 0$. This condition implies the linear independence of the rows of $\widetilde{K}^{\sigma}$. As expected, it typically fails for kernels of the form \eqref{eq: exponential} with the same exponent, while different exponents in \eqref{eq: exponential} yield larger values of $\gamma_*$, which rules out our technical balance conditions. Still, rich classes of non-exponential kernels fulfill this condition, with examples and sufficient conditions discussed in Subsection~\ref{subsection:nondegregvarkernel}. 

Our construction of the Volterra Ito-process $Z$ is crucially based on Markovian lifts that embed the dynamics into an enlarged state space that implicitly contains the past trajectory of the process. In this way, a Markovian structure is restored, albeit at the cost of infinite-dimensionality and unbounded coefficients. For stochastic Volterra equations, several variants of Markovian lifts based on Laplace transforms have been proposed \cite{MR3057145, MR2511555, MR1658690, MR3926553}. In \cite{MR4503737}, lifts based on time shifts were constructed for absolutely continuous, and hence regular, Volterra kernels. A general approach based on generalized Feller semigroups was developed in \cite{MR4181950}, while \cite{DG23} investigates Markovian lifts in UMD Banach spaces with applications to optimal control. Hilbert-space-valued lifts for completely monotone kernels were introduced in \cite{H23, hamaguchi2023weak} and further refined in \cite{FGW24}. In the present work, we adopt the recent framework of \cite{BBCF25}, which formulates an abstract Hilbert space setting that unifies the approaches of \cite{MR4503737, H23}. This framework offers the necessary flexibility to capture the essential features of memory and accommodates all examples of interest from the literature. Such an abstract formulation allows us to separate lift-specific constructions from general structures behind the path-dependency, and offers a flexible framework to cover also other types of Markovian lifts beyond the established ones.  

Requiring that $Z_t$ is $X_t$-measurable essentially restricts the class of possible kernels to \textit{admissible perturbations} $\widetilde{k}^b, \widetilde{k}^{\sigma}$ of the original kernels $K^b, K^{\sigma}$. Using the flexibility gained by Markovian lifts, we characterize \textit{admissible perturbations} as $L^2$-limits of linear combinations of time shifts of the original Volterra kernels, subject to the constraint that for each $t > 0$ the admissible perturbations $\widetilde{k}^b(t), \widetilde{k}^{\sigma}(t)$ depend solely on the values of $K^b(r), K^{\sigma}(r)$ for $r \geq t$. Beyond general characterizations, we demonstrate that an abstract version of the forward Weyl/Marchaud fractional derivative/integral yields a generic procedure for obtaining admissible perturbations. As a byproduct of this construction, we show that the corresponding Markovian lift is typically infinite-dimensional, and that the support of the regular conditional distribution of the Markovian lift given the Volterra process contains infinitely many elements. 

Finally, remark that our methodology extends to non-convolution kernels $K^b(t,s)$, $K^{\sigma}(t,s)$, albeit at the cost of more delicate admissibility conditions with respect to linear combinations of translations of $K^b(\cdot,s)$, $K^{\sigma}(\cdot,s)$ with an additional parameter~$s$. Moreover, the methods developed herein can be adapted to establish the failure of the \textit{almost everywhere Markov property} for stochastic Volterra processes with jumps. This weaker form of the Markov property is tailored to processes that lack sufficient time-regularity and arises naturally in the study of Markov selection theorems~\cite{MR2365480}.

\subsection{Notation} 
Here and below, for arguments requiring estimates merely modulo a multiplicative constant, we denote by~$\lesssim$ an inequality up to a constant factor that is not further specified. The precise quantities on which the constant is allowed to depend are, in any case, clear from the context. 

For $x \in \R^d$, we denote by $|x|$ the usual Euclidean norm. For a matrix $A \in \R^{d \times m}$, let us define the operator (spectral) norm $|A| = \sup_{|x|=1}|Ax| = \lambda_{\mathrm{max}}(A^{\intercal}A)^{1/2}$, where the Euclidean norms on $\R^m$ and $\R^d$ are, by slight abuse of notation, both denoted by~$|\cdot|$. In particular, as a consequence of Weyl's inequality for singular values, one gets for $A,B \in \R^{d\times m}$ the Lipschitz bound 
\begin{align}\label{eq: matrix bound}
    \big| \lambda_{\min}(A^{\intercal}A)^{1/2} - \lambda_{\min}(B^{\intercal}B)^{1/2} \big| \leq |A - B|.
\end{align}

For $n,l\in\mathbb{N}$ and $\chi \in (0,1]$ we denote by $B(\R_+; C^{\chi}(\R^n; \R^l))$ the space of measurable functions $f:\R_+\times\R^n\longrightarrow \R^l$ such that for each $T > 0$ there exists $c_T > 0$ with 
\begin{align}\label{eq:uniformHölder}
    |f(t,x)-f(t,y)| \le c_T\hspace{0.02cm} |x-y|^{\chi} \ \text{ and } \ |f(t,0)| \leq c_T.
\end{align}
where $t\in [0,T],\ x,y\in\R^n$. For $\chi = 0$, by slight abuse of notation, let $B(\R_+; C^{0}(\R^n; \R^l))$ be the space of measurable functions $f:\R_+\times\R^n\longrightarrow \R^l$ such that for each $T > 0$ there exists a constant $\overline{c}_T>0$ with
\begin{equation}\label{eq:uniformlineargrowth}
    |f(t,x)|\le \overline{c}_T\hspace{0.02cm} (1+|x|), \quad \forall t\in [0,T],\ x\in\R^n.
\end{equation}
Clearly, one has $B(\R_+; C^{\chi}(\R^n; \R^l)) \subseteq B(\R_+; C^{0}(\R^n; \R^l))$ for $\chi \in (0,1]$.

\subsection{Structure of the work}

Section~\ref{section:abs cont} provides the preliminary setup for our results. We first study the absolute continuity of the law for Volterra Ito-processes with nondegenerate kernels, discuss extensions to anisotropic settings, and finally illustrate our density method for the Volterra square-root process introduced in \eqref{eq:roughHeston}. 
In Section~\ref{section:markovianliftapproach}, we introduce the Markovian lift framework and proceed by showing that for nondegenerate Volterra kernels the range of the corresponding covariance operator is necessarily infinite-dimensional, and that the support of the regular conditional distribution of the Markovian lift given the Volterra process contains infinitely many elements. Section~\ref{section:failure} forms the core technical part of this work. Here, we introduce the notion of admissible perturbations, discuss their properties, sufficient conditions, and particular examples. The section concludes with a proof of the failure of the Markov property. Concrete realizations of our Markovian lift framework and illustrative examples are presented in Section~\ref{section: examples}, while further technical auxiliary results are collected in the appendix.

\section{Absolute continuity of Volterra Ito-processes}\label{section:abs cont}

\subsection{General case}\label{subsection:exampleperturbations}

Below, we study the absolute continuity for Volterra Ito-processes with drift $b_t = b(t,X_t)$ and diffusion $\sigma_t = \sigma(t,X_t)$ with $X$ satisfying \eqref{eq:generalSVIE}, i.e. 
\begin{align}\label{eq:perturbedVolterraprocessDef}
    \widetilde{Z}_t := \widetilde{g}(t) + \int_0^t \widetilde{K}^{b}(t-s)\hspace{0.02cm}b(s,X_s)\, \mathrm{d}s
    + \int_0^t \widetilde{K}^{\sigma}(t-s)\hspace{0.02cm}\sigma(s,X_s)\, \mathrm{d}B_s,\quad t\in\R_+.
\end{align}
Here, $\widetilde{K}^b, \widetilde{K}^{\sigma} \in L_{\mathrm{loc}}^2(\R_+; \R^{N \times d})$, $\widetilde{g}: \R_+ \longrightarrow \R^N$ is $\F_0$-measurable, and $N \in \mathbb{N}$. 

To deduce the existence of a density, it is required that $\widetilde{K}^{\sigma}$ is sufficiently nondegenerate. The latter is captured by non-determinism conditions formulated in terms of lower bounds on the eigenvalues of Gram matrices constructed from $\widetilde{K}^{\sigma}(\widetilde{K}^{\sigma})^{\intercal}$. Namely, define on $\R_+$ an $\R^{N\times N}$-valued matrix function $G$ as the sum of suitable Gram matrix functions $G^{(i)}$ in the following way: For each $i \in \{1,\dots, d\}$, and $\ell, \ell' \in \{1,\dots, N\}$ we set
    \begin{align}\label{eq: Gram}
        G^{(i)}_{\ell \ell'}(h) = \int_0^h \widetilde{K}^{\sigma}_{\ell i}(r) \widetilde{K}^{\sigma}_{\ell' i}(r)\, \mathrm{d}r \ \ \text{ and } \ \ 
        G_{\ell\ell'}(h) = \sum_{i=1}^d G_{\ell \ell'}^{(i)}(h).
    \end{align}

\begin{definition}\label{def: nondegeneracy general}
   A Volterra kernel $\widetilde{K}^{\sigma} \in L_{\mathrm{loc}}^2(\R_+; \R^{N \times d})$ is $\gamma_*$-nondegenerate, if  
    \begin{align}\label{eq: lower bound}
       \liminf_{h \searrow 0}h^{-2\gamma_*} \lambda_{\min}(G(h)) > 0,
    \end{align}
    where $\gamma_* > 0$ and $\lambda_{\min}(G(h))$ denotes the smallest eigenvalue of $G(h)$.
\end{definition}

In particular, $\gamma_*$-nondegeneracy implies that the rows of $\widetilde{K}^{\sigma}$ are linearly independent. It can be verified through a lower bound for   
\begin{align*}
    \int_0^h \big| \widetilde{K}^{\sigma}(r)^{\intercal}\xi\big|^2\, \mathrm{d}r 
    = \sum_{i=1}^d \int_0^h \bigg( \sum_{\ell = 1}^N \widetilde{K}^{\sigma}_{\ell i}(r) \xi_{\ell} \bigg)^2\, \mathrm{d}r
    = \sum_{\ell, \ell' = 1}^N \xi_{\ell}\xi_{\ell'} G_{\ell \ell'}(h),
\end{align*}
uniformly in $\xi \in \R^N$ with $|\xi| = 1$. In Subsection~\ref{subsection:nondegregvarkernel}, we will show that kernels with regularly varying components typically satisfy the nondegeneracy property \eqref{eq: lower bound}.

Recall that~$X$ appearing in \eqref{eq:perturbedVolterraprocessDef} is a solution of \eqref{eq:generalSVIE}. Let us suppose that the Volterra kernels $K^b, K^{\sigma}$ and $g$ in \eqref{eq:generalSVIE} satisfy the following minor regularity condition:
\begin{enumerate}
    \item[(A)] $K^b, K^{\sigma} \in L_{\mathrm{loc}}^2(\R_+;\R^{d\times d})$, and there exists $\gamma_K > 0$ such that for each $T > 0$ there exists a constant $C_T > 0$ satisfying
    \begin{align}\label{eq: K increment}
        \int_0^T |K^a(r + h) - K^a(r)|^2 \, \mathrm{d}r + \int_0^h |K^a(r)|^2\, \mathrm{d}r \leq C_T h^{2\gamma_K}, \qquad h \in (0,T],
    \end{align}
    for each $a \in \{b,\sigma\}$. Moreover, there exists $p > 2$ such that
    \begin{align}\label{eq: g Hoelder}
        \|g(t) - g(s)\|_{L^p(\Omega)} \leq C_T(t-s)^{\gamma_K}, \qquad 0 \leq s \leq t \leq T.
    \end{align}
\end{enumerate}

The constant $\gamma_K$ determines the sample path regularity of $X$. Sufficient conditions and characterizations for \eqref{eq: K increment} can be found in \cite[Example 2.3]{AbiJaLaPu19} and \cite[Section~3]{FGW24}. Recall that $\Gamma_{\sigma,t} = \{ x \in \R^d \, : \, \mathrm{det}(\sigma(t,x)\sigma(t,x)^{\intercal}) \neq 0\}$ was defined in \eqref{eq: Gamma sigma}. The following theorem establishes the absolute continuity of the law for Volterra kernels $\widetilde{K}^{\sigma}$ that satisfy the nondegeneracy condition.

\begin{theorem}\label{thm: density}
    Suppose that condition (A) holds, $b \in B(\R_+; C^{\chi_b}(\R^d; \R^d))$, and $\sigma \in B(\R_+; C^{\chi_{\sigma}}(\R^d; \R^{d\times m}))$ where $\chi_b \in [0,1]$ and $\chi_{\sigma} \in (0,1]$. Let $\widetilde{K}^b, \widetilde{K}^{\sigma} \in L_{\mathrm{loc}}^2(\R_+; \R^{N \times d})$ with $N \in \mathbb{N}$ be given such that
    \begin{align}\label{eq: upper bound}
         \int_0^h \big|\widetilde{K}^{a}(r)\big|^2\, \mathrm{d}r \le C\hspace{0.02cm}h^{2\gamma_{a}},\quad h\in [0,1],\ a\in\{b,\sigma\},
    \end{align}
    is satisfied for some $C,\gamma_b, \gamma_{\sigma} > 0$. Suppose that $\widetilde{K}^{\sigma}$ is $\gamma_*$-nondegenerate and that the constants satisfy the relation 
    \begin{align}\label{eq: H condition}
        \gamma_* < \min \left\{ \gamma_b + \frac{1}{2} + \chi_b\hspace{0.03cm}\gamma_K,\, \gamma_{\sigma} + \chi_{\sigma}\hspace{0.03cm}\gamma_K \right\}.
    \end{align}
    Let $X$ be a continuous weak solution of \eqref{eq:generalSVIE}, $\widetilde{g}: \R_+ \longrightarrow \R^N$ be an $\F_0$-measurable function, and denote by $\widetilde{Z}$ the Volterra Ito-process given by \eqref{eq:perturbedVolterraprocessDef}. Then  
    \[
        \nu_t(A) = \P\big[ \widetilde{Z}_t \in A, \ X_t \in \Gamma_{\sigma,t} \big],\quad A\in\mathcal{B}(\R^N),
    \]
    is for each $t > 0$ absolutely continuous with respect to the Lebesgue measure on $\R^N$. 
\end{theorem}
\begin{proof}
    We deduce the assertion by an application of \cite[Theorem 3.6\hspace{0.03cm}(b)]{F24}. For this purpose, fix an arbitrary $T > 0$. Below we verify the properties (A1) -- (A3) from \cite{F24} for any $s,t \in [0,T]$ with $s < t$ and constants independent of $(s,t)$:
    \begin{enumerate}
        \item[(i)] Combining \eqref{eq: upper bound} with Jensen's inequality and $(t-s)^{\gamma_b+1/2}, (t-s)^{2\gamma_{\sigma}}\ge 1$ when $(t-s)\ge 1$ proves $\int_s^t \big|\widetilde{K}^{b}(t-r)\big|\, \mathrm{d}r\lesssim (t-s)^{\gamma_b+1/2}$ and $\int_s^t \big|\widetilde{K}^{\sigma}(t-r)\big|^2\, \mathrm{d}r\lesssim (t-s)^{2\gamma_{\sigma}}$.
        \item[(ii)] Note that for $p>2$ from \eqref{eq: g Hoelder}, a straightforward  extension of \cite[Proposition~4.1, Remark~4.2]{F24} towards time-dependent $b,\sigma$, see also \cite[Lemma 3.4]{PROMEL2023291} for deterministic~$g$, yields the increment bound $\|X_t - X_s\|_{L^p(\Omega)} \lesssim (t-s)^{\gamma_K}$. Hence, it follows from the Hölder regularity of $b$, $\sigma$, which is uniform in the time argument on $[0,T]$, see~\eqref{eq:uniformHölder}, and Jensen's inequality that $\| b(t,X_t) - b(s,X_s)\|_{L^p(\Omega)} \lesssim (t-s)^{\alpha_b}$ and $\|\sigma(t,X_t) - \sigma(s,X_s)\|_{L^p(\Omega)} \lesssim (t-s)^{\alpha_{\sigma}}$ holds for $\alpha_b = \chi_b \gamma_K$, $\alpha_{\sigma} = \chi_{\sigma} \gamma_K$ and the above $p>2$. In particular, note that the previous bound remains valid for the case $\chi_b=0$, where $b$ is merely of linear growth in the space variable.
        \item[(iii)] Set $\rho_{t}:= \lambda_{\min}(\sigma(t,X_t)\hspace{0.02cm}\sigma(t,X_t)^{\intercal})^{1/2} \wedge 1$ for $t\in [0,T]$. By the assumed nondegeneracy condition~\eqref{eq: lower bound}, there exist $h_0 \in (0,1)$ and a constant $C > 0$ such that for each $\xi \in \R^N$ with $|\xi|=1$, and $0 \leq t-s \leq h_0$ we find
        \begin{align}\label{eq:2dimdensityB3estimation}
            \int_s^t \big| \sigma(s,X_s)^{\intercal} \widetilde{K}^{\sigma}(t-r)^{\intercal}\xi \big|^2 \, \mathrm{d}r
            \geq \rho_s^2 \int_s^t \big| \widetilde{K}^{\sigma}(t-r)^{\intercal}\,\xi \big|^2\, \mathrm{d}r
            \geq C\,\rho_{s}^2\, (t-s)^{2\gamma_*}.
        \end{align}
    By the nonnegativity of the integrand, this clearly extends to all $(t-s)\in[0,T]$ through an adjustment of the constant. Moreover, using \eqref{eq: matrix bound}, we obtain the bound $\|\rho_t - \rho_s \|_{L^p(\Omega)} \lesssim \| \sigma(t,X_t) - \sigma(s,X_s)\|_{L^p(\Omega)} \lesssim (t-s)^{\chi_{\sigma}\gamma_K}$.     
    \end{enumerate}
   The particular choice of $\alpha_b, \alpha_{\sigma}$ combined with \eqref{eq: H condition} proves $\gamma_* < \min\{ \alpha_b +\gamma_b+ \frac{1}{2}, \ \alpha_{\sigma} + \gamma_{\sigma} \}$. Therefore, an application of \cite[Theorem 3.6\hspace{0.03cm}(b)]{F24} implies that, for every $t\in (0,T]$, the measure defined by
    \[
        \mathcal{B}\big(\R^{N}\big)\ni A \longmapsto \E\Big[ \rho_t \cdot \hspace{0.03cm} \1_{A}\big(\widetilde{Z}_t\big) \Big]
    \]
    is absolutely continuous with respect to the Lebesgue measure~$\mathrm{Leb}_{N}$ on $\R^{N}$. For $n \geq 1$, define $\Gamma^{(n)}_{\sigma,t} := \big\{ x \in \R^d \, : \, \lambda_{\min}(\sigma(t,x)\sigma(t,x)^{\intercal})^{1/2} \wedge 1 \geq 1/n\big\}$. Then for each Borel set $A \subseteq \R^{N}$ with $\mathrm{Leb}_{N}(A)=0$, one has
    \begin{align*}
        \P\big[ \widetilde{Z}_t \in A, \ X_t \in \Gamma^{(n)}_{\sigma,t} \big]
        \leq n\, \E\big[ \big(\lambda_{\min}\left(\sigma(t,X_t)\sigma(t,X_t)^{\intercal}\right)^{1/2}\wedge 1\big)\, \1_{A}\big(\widetilde{Z}_t\big) \big] = 0.
    \end{align*}
    Thus, noting that $\Gamma_{\sigma,t} = \{x \in \R^d \, : \, \mathrm{det}(\sigma(t,x)\sigma(t,x)^{\intercal}) \neq 0\} = \cup_{n \geq 1}\Gamma_{\sigma,t}^{(n)}$, we obtain from the continuity from below of probability measures:
    \begin{align*}
        \nu_t(A) = \P\big[ \widetilde{Z}_t \in A, \ X_t \in \Gamma_{\sigma,t}\big] 
        = \lim_{n \to \infty} \P\big[ \widetilde{Z}_t \in A, \ X_t \in \Gamma_{\sigma,t}^{(n)} \big] = 0.
    \end{align*}
    Therefore, $\nu_t$ is absolutely continuous with respect to the Lebesgue measure on $\R^N$. 
\end{proof} 

For $\widetilde{K}^b = K^b$ and $\widetilde{K}^{\sigma} = K^{\sigma}$ we obtain $\nu_t(A) = \P[ X_t \in A \cap \Gamma_{\sigma,t}]$, and hence absolute continuity of the law of $X_t$ restricted onto $\Gamma_{\sigma,t}$. For the failure of the Markov property and estimates on the dimension of the Markovian lift, we need to study perturbations of~$X$ as outlined in the following example. 

\begin{example}\label{example: 1}
    Let $N \geq 1$ and set $\widetilde{K}^a = \big(K^a, \widetilde{k}^a\big)^{\intercal}$ where $\widetilde{k}^a \in L_{\mathrm{loc}}^2(\R_+; \R^{N \times d})$ for each $a \in \{b,\sigma\}$. In this case, we may obtain the absolute continuity on $\Gamma_{\sigma,t} \times \R^N$ of the $\R^d \times \R^N$-valued process $\widetilde{Z} = \big(X,\widetilde{X}\big)^{\intercal}$ with 
    \[
        \widetilde{X}_t = \widetilde{g}(t) + \int_0^t \widetilde{k}^b(t-s)\hspace{0.02cm}b(s,X_s)\, \mathrm{d}s + \int_0^t \widetilde{k}^{\sigma}(t-s)\hspace{0.02cm}\sigma(s,X_s)\, \mathrm{d}B_s,\quad t\in\R_+.
    \]
    We will see that $N = 1$ suffices to prove the failure of the Markov property, while general~$N$ will be used in 
    Section~\ref{section:markovianliftapproach} to study the dimension and the conditional distribution of the Markovian lift.
\end{example}

\subsection{Anisotropic case}

In many cases, the Volterra kernels have a more specific structure, and refinements of the $\gamma_*$-nondegeneracy condition may be used to sharpen condition \eqref{eq: H condition}. Below, we consider the most important case that reflects diagonal kernels up to permutations and also contains more general kernel matrices with $N\neq d$. The latter allows us to derive anisotropic regularity for kernels where in each coordinate different asymptotics are present.

\begin{definition}
    We call a Volterra kernel $\widetilde{K} \in L_{\mathrm{loc}}^2(\R_+; \R^{N \times d})$ diagonal-like, if for each $j \in \{1,\dots, N\}$ there exist $\widetilde{k}_j\in L_{\mathrm{loc}}^2(\R_+;\R)$ and a unique $a(j) \in \{1,\dots, d\}$ such that
    \[
        \widetilde{K}_{j\ell}(t) = \delta_{\ell a(j)}\widetilde{k}_j(t), \qquad t > 0, \ \ \ell \in \{1,\dots, d\}.
    \]
\end{definition}

This definition is motivated by a particular case of Example \ref{example: 1} where, e.g., $\widetilde{K}^a = (K^a, \widetilde{k}^a)$ for $a\in\{b,\sigma\}$ with diagonal $K^a$ and a row vector $\widetilde{k}^a$ with exactly one non-zero entry. More generally, we may also take $\widetilde{K}^a = (K^a, K_1^a,\dots, K_N^a)^{\intercal}$ where $K^a, K_1^a,\dots, K_N^a$ are diagonal $d\times d$-matrix functions. In any case, such a sparse structure leads to a specific form of the corresponding Gram matrix function $G$ defined in \eqref{eq: Gram}. 

Let $\widetilde{K}^{\sigma}$ be such a diagonal-like Volterra kernel. Define for $i \in\{1,\dots, d\}$ the sets 
\begin{align}\label{eq: S decomposition}
    S_i = \{ j \in \{1,\dots, N\} \ : \ a(j) = i \}.
\end{align}
Then $(S_i)_{i \in \{1, \dots, d\}}$ forms a partition of $\{1,\dots, N\}$. The associated Gram matrix $G$ defined in \eqref{eq: Gram} satisfies $G_{\ell \ell'} \equiv 0$ if $\ell \in S_i$ and $\ell' \in S_j$ for $i \neq j$, while for $\ell, \ell' \in S_i$ we obtain $G_{\ell \ell'}(h) = \int_0^h \widetilde{k}^{\sigma}_{\ell}(r)\hspace{0.02cm} \widetilde{k}^{\sigma}_{\ell'}(r)\, \mathrm{d}r$. In particular, it follows that
\begin{align}\label{eq: Gram2}
    \int_0^h \big| \widetilde{K}^{\sigma}(r)^{\intercal}\xi\big|^2\, \mathrm{d}r
    = \sum_{i=1}^d \bigg( \sum_{\ell, \ell' \in S_i} \xi_{\ell}\xi_{\ell'} G_{\ell \ell'}(h) \bigg).
\end{align}
This provides a natural decomposition of $G$ into blocks indexed by $S_1,\dots, S_d$. Hence, let us define for each $i \in \{1,\dots, d\}$ and $h>0$ the Gram matrices $\widetilde{G}_{\ell \ell'}^{(i)}(h) = \int_0^h \widetilde{k}^{\sigma}_{\ell}(r)\widetilde{k}^{\sigma}_{\ell'}(r)\, \mathrm{d}r$, $\ell, \ell' \in S_i$. On each $\R^{|S_i|\times|S_i|}$-valued block $\widetilde{G}^{(i)}$, we suppose that the nondegeneracy condition is satisfied with parameters $\gamma_*^1, \dots, \gamma_*^d$. 

\begin{definition}\label{def: nondegeneracy anisotrop}
    The diagonal-like Volterra kernel $\widetilde{K}^{\sigma} \in L_{\mathrm{loc}}^2(\R_+; \R^{N \times d})$ is called $\gamma_*$-nondegenerate, where $\gamma_* = (\gamma_*^1,\dots, \gamma_*^d)\in (\R_+^*)^d$, if for each $i \in \{1,\dots, d\}$, it holds that
    \[
        \liminf_{h \searrow 0} h^{-2\gamma_*^i}\lambda_{\min}\big(\widetilde{G}^{(i)}(h)\big) > 0.
    \]
    For $i \in \{1,\dots, d\}$ with $S_i=\emptyset$ we define $\gamma_*^{i}=0$.
\end{definition}
 
\begin{remark}\label{remark:nondegeneracyconnection}
    A diagonal-like kernel $\widetilde{K}^{\sigma}$ that is nondegenerate according to Definition~\ref{def: nondegeneracy anisotrop} is also $\max\{\gamma_*^1, \dots, \gamma_*^d\}$-nondegenerate in the sense of Definition \ref{def: nondegeneracy general}. Thus, Definition~\ref{def: nondegeneracy anisotrop} provides a refinement where we keep track of the asymptotics on each $S_1,\dots, S_d$ separately.
\end{remark}

For SDEs driven by cylindrical L\'evy processes, anisotropic regularity effects were studied in~\cite{FJR18} in terms of anisotropic Besov spaces. An extension of these results for Volterra Ito-processes is given in Proposition \ref{proposition:VolterraItoanisotropdensityperturbations} of the appendix. In particular, we obtain the following refinement of Theorem~\ref{thm: density}:

\begin{theorem}\label{thm: density diagonal} 
   Suppose that condition~(A) is satisfied, and the components of $b: \R_+\times \R^d \longrightarrow \R^d$ and $\sigma = \mathrm{diag}(\sigma_1,\dots, \sigma_d): \R_+\times\R^d \longrightarrow \R^{d \times d}$ fulfill $b_i \in B(\R_+; C^{\chi_b^{i}}(\R^d; \R))$ and $\sigma_i \in B(\R_+; C^{\chi_{\sigma}^{i}}(\R^d; \R))$ with $\chi_b^i \in [0,1]$ and $\chi_{\sigma}^i \in (0,1]$ for each $i\in \{1,\dots, d\}$. Let $\widetilde{K}^b, \widetilde{K}^{\sigma} \in L_{\mathrm{loc}}^2(\R_+; \R^{N \times d})$ with $N \geq 1$ be diagonal-like with the same partition $(S_i)_{i\in\{1,\dots,d\}}$ in \eqref{eq: S decomposition}. Assume that $\widetilde{K}^{\sigma}$ is $\gamma_*$-nondegenerate in the sense of Definition \ref{def: nondegeneracy anisotrop}, and for each $i \in \{1,\dots, d\}$ and $a\in\{b,\sigma\}$,
   \begin{align}\label{eq: ani upper bound 1}
    \int_0^h \big|\widetilde{k}^{a}_j(t)\big|^2\, \mathrm{d}t \le C\hspace{0.02cm}h^{2\gamma_{a}^i}, \qquad j \in S_i,\ h\in [0,1],
   \end{align}
   holds for some $C>0$, $\gamma_{b}^1,\dots, \gamma_{b}^d,\gamma_{\sigma}^1,\dots, \gamma_{\sigma}^d > 0$, with the convention $\gamma_{a}^i = \infty$ when $S_i = \emptyset$, and these parameters satisfy the relation
        \begin{equation}\label{eq: H condition diagonal}
            \gamma^i_* < \min\left\{\gamma_b^i + \frac{1}{2}+\chi^i_{b}\hspace{0.02cm}\gamma_K,\, \gamma_{\sigma}^i + \chi^i_{\sigma}\hspace{0.02cm}\gamma_K \right\}, \qquad i \in\{1,\dots, d\}.
    \end{equation}
    Let $X$ be a continuous weak solution of \eqref{eq:generalSVIE}, $\widetilde{g}: \R_+ \longrightarrow \R^N$ be a $\F_0$-measurable function, and denote by $\widetilde{Z}$ the Volterra Ito-process given by \eqref{eq:perturbedVolterraprocessDef}. Then  
    \[
        \nu_t(A) = \P\big[ \widetilde{Z}_t \in A, \ X_t \in \Gamma_{\sigma,t} \big],\quad A\in\mathcal{B}(\R^N),
    \]
    is for each $t > 0$ absolutely continuous with respect to the Lebesgue measure on $\R^N$. 
\end{theorem} 
\begin{proof}
    Since $\widetilde{Z}$ is a Volterra Ito-process of the form \eqref{eq:anisotropicVolterraIto} with $b_s = b(s,X_s)$ and $\sigma_s^{i}=\sigma_{i}(s,X_s)$, we may prove regularity via Proposition \ref{proposition:VolterraItoanisotropdensityperturbations}. By assumption, it is clear that conditions (D1) and (D2) therein are satisfied. As before, condition (A) combined with \cite[Proposition~4.1, Remark~4.2]{F24} gives $\|X_t - X_s\|_{L^p(\Omega)} \lesssim_{p,T} (t-s)^{\gamma_K}$ for all $s,t\in [0,T]$ with $s\le t$ and each $T>0$. Hence, condition (D3) holds for $\alpha_b^{i} = \chi_b^{i} \gamma_K$ and $\alpha_{\sigma}^{i} = \chi_{\sigma}^{i} \gamma_K$ since $\| b_i(t,X_t) - b_i(s,X_s)\|_{L^p(\Omega)} \lesssim_{p,T} (t-s)^{\chi_b^{i} \gamma_K}$ and $\|\sigma_i(t,X_t) - \sigma_i(s,X_s)\|_{L^p(\Omega)} \lesssim_{p,T} (t-s)^{\chi_{\sigma}^{i}\gamma_K}$ for every $0<s<t\le T$, which again remains valid for the case $\chi_b^{i}=0$. Finally, condition~(D4) is an immediate consequence of \eqref{eq: H condition diagonal}. Hence, Proposition~\ref{proposition:VolterraItoanisotropdensityperturbations} proves that for each $t>0$ the measure
    \begin{equation}\label{eq:anisotropSVIEmeasuredensity}
        \mathcal{B}\big(\R^{N}\big)\ni A \longmapsto \E\left[\left(1\wedge\min_{i\in\{1,\dots,d\}}|\sigma_i(t,X_t)|\right)\, \1_{A}\big(\widetilde{Z}_t\big) \right]
    \end{equation}
    is absolutely continuous with respect to the Lebesgue measure on $\R^{N}$. Consider the sets $\Gamma^{(n)}_{\sigma,t} := \big\{ x \in \R^d \, : \, 1\wedge \min_{i\in\{1,\dots,d\}}|\sigma_i(t,x)|  \geq 1/n\big\}$ for $n\ge 1$. Arguing exactly as in the proof of Theorem \ref{thm: density} yields the assertion.
\end{proof} 

Note that, by the diagonality of $\sigma$, the set $\Gamma_{\sigma,t}^{(n)}$ in Theorem \ref{thm: density diagonal} is consistent with its general definition in the proof of Theorem \ref{thm: density}. We proceed with a short remark on the role of the H\"older exponents:
\begin{remark}
    If the drift $b$ or diffusion coefficient $\sigma$ is constant, we may take $\chi_b = \infty$ or $\chi_{\sigma} = \infty$, respectively. This simplifies conditions \eqref{eq: H condition} and \eqref{eq: H condition diagonal}. In particular, for Gaussian Volterra processes $\int_0^t K(t-s)\sigma\, \mathrm{d}B_s$ no relation between $\gamma$ and $\gamma_*$ is required. In this case, stronger results can be directly obtained by classical Fourier methods.
\end{remark}

Finally, let us mention that our assumption on $\widetilde{K}^{b}$ being also diagonal-like is for mere convenience. For general $\widetilde{K}^{b}\in L_{\mathrm{loc}}^2(\R_+; \R^{N \times d})$, it can easily be seen from the proof of Proposition \ref{proposition:VolterraItoanisotropdensityperturbations} that the absolute continuity of the measures $\nu_t$ can be established as soon as there exists $\gamma_b > 0$ such that $\limsup_{h \searrow 0}h^{-2\gamma_b}\int_0^h |\widetilde{K}^b(r)|^2\, \mathrm{d}r < \infty$, and \eqref{eq: H condition diagonal} is strengthened to
\[
    \gamma_*^i < \min\left\{ \gamma_b + \frac{1}{2} + \chi_b \gamma_K,\ \gamma_{\sigma}^i + \chi_{\sigma}^i \gamma_K \right\}, \qquad i \in \{1,\dots, d\},
\]
where $\chi_b=\min_{i\in\{1,\dots,d\}}\chi_b^{i}$. Similarly, one could also study more complex structures of $\widetilde{K}^{\sigma}$, in between diagonal-like and fully general kernels, where certain coordinates with different orders of regularity interact with each other. However, such a generalization requires its own custom-tailored regularity result in the spirit of Proposition~\ref{proposition:VolterraItoanisotropdensityperturbations}.

\subsection{The Volterra square-root process}\label{section:affine case}

Let us suppose that $K: \R_+^* \longrightarrow \R$ is completely monotone, not identically zero, and there exists $\gamma \in (0,1]$ such that for each $T > 0$ we find a constant $C_T > 0$ with
\begin{align}\label{eq:kernelincrementconditions}
    \int_0^h K(t)^2\, \mathrm{d}t + \int_0^T |K(t+h) - K(t)|^2\, \mathrm{d}t \leq C_T h^{2\gamma}, \qquad h \in (0,1].
\end{align}
Given $x_0, b, \sigma \in\R_+$ and $\beta \in \R$, it follows from \cite[Theorem 6.1]{AbiJaLaPu19} that \eqref{eq:roughHeston} admits a unique continuous nonnegative weak solution $X$. In this section, we illustrate how the absolute continuity for corresponding Volterra Ito-processes can be used to prove the failure of the Markov property for the Volterra square-root process \eqref{eq:roughHeston}. To keep this section concise and informative, we focus only on the one-dimensional case. However, the method below could certainly be applied to general affine Volterra processes and even to models with general Hölder diffusion coefficients beyond the square-root.

First, we recall a representation of the conditional first moment of the process obtained in \cite[Theorem 4.5\hspace{0.03cm}(i)]{AbiJaLaPu19}. Since $K$ is completely monotone, it follows from \cite[Theorem 5.5.4]{grippenberg} that it admits a unique (measure) resolvent of the first kind $L$ defined by $L \ast K = K \ast L = 1$. This resolvent satisfies
\begin{align}\label{eq: L resolvent}
    L(\mathrm{d}s) = K(0)^{-1}\delta_0(\mathrm{d}s) + L_0(s)\,\mathrm{d}s,
\end{align}
where $K(0) \in (0,\infty]$ with the convention $\infty^{-1} := 0$, and $L_0 \in L_{\mathrm{loc}}^1(\R_+)$ is completely monotone. Define $\Delta_h f(t) = f(t+h)$, and by \cite[p.\ 3170]{AbiJaLaPu19} combined with the substitution $t+h - u = r$ and the particular form of $L$, we obtain $\Delta_h K \ast L (t) = 1 - \int_0^h K(r)\hspace{0.02cm} L_0(t+h-r)\, \mathrm{d}r$. In particular, $\Delta_h K \ast L \in [0,1]$ is nondecreasing, and hence of locally bounded variation. As the latter holds uniformly in $h\in [0,T]$, the technical condition \cite[Eq.\ (4.15)]{AbiJaLaPu19} is satisfied. 

Next, denote by $E_K \in L_{\mathrm{loc}}^2(\R_+) \cap C(\R_+^*)$ the unique solution of $E_K(t) = K(t) + \beta \int_0^t K(t-s)E_K(s)\, \mathrm{d}s$, and let $(\Pi_z)_{z\in\R_+}$ be given by
\begin{equation}\label{eq:Pihdefinition}
    \Pi_z(t)=( \Delta_z E_K\ast L)(t) - \Delta_z(E_K\ast L)(t),\quad t,z \in\R_+.
\end{equation}
By direct computation, we find that
\begin{align}\label{eq: Pi z formula}
    \Pi_z(t) = - \int_{(t,t+z]} E_K(t+z - r)\, L(\mathrm{d}r) = - \int_0^z E_K(r) L_0(t+z-r)\, \mathrm{d}r,
\end{align}
where we have used \eqref{eq: L resolvent} and the substitution $r \longmapsto t+z-r$. In particular, $\Pi_z$ is for each $z > 0$ continuous and nonpositive. For $\beta \leq 0$, $E_K$ is completely monotone (which follows from \cite[Theorem 5.3.1]{grippenberg} applied to $-\beta E_K$), while for $\beta > 0$, $E_K$ is also nonnegative. Consequently, $\Pi_z$ is monotone and thus of locally bounded variation. An application of \cite[Theorem 4.5\hspace{0.03cm}(i)]{AbiJaLaPu19} gives for every $t, T\in\R_+$ with $t<T$ the desired formula for the conditional expectation: 
\begin{align}\label{eq:fancyexpectationformula}
    \notag \mathbb{E}[X_T\, | \, \mathcal{F}_t] &= \int_0^{T-t}E_K(s)b\, \mathrm{d}s - \Pi_{T-t}(t)x_0 
    \\ &\qquad  + ( \Delta_{T-t}E_K\ast L)(0)\hspace{0.02cm}X_t + \int_{[0,t]} X_{t-s}\,\Pi_{T-t}(\mathrm{d}s),
\end{align}

Path-dependence is encoded in the integral $\int_{[0,t]} X_{t-s}\, \Pi_{T-t}(\mathrm{d}s)$. Its dynamics in $t$ can be derived from the original SVE, where $K$ is replaced with the auxiliary Volterra kernel 
\[
        K_z(t) =(K\ast\mathrm{d}\Pi_z)(t)= \int_{[0,t]} K(t-s) \,\Pi_{z}(\mathrm{d}s).
\]
The next lemma shows that $K_z$ is bounded above and below by the shifted kernel $K(\cdot + z)$.

\begin{lemma}\label{lemma:Kzkernelasymptotics}
    Assume that $L_0'\not\equiv 0$. Then $\mathrm{d}\Pi_z>0$ for each $z > 0$. Moreover, there exists an increasing function $f:\R_+^* \longrightarrow \R_+$ such that $f(0+) = 0$, and  
    \begin{equation}\label{eq:Kzkernelasymptoticsalternate}
        1-f(z) \leq \frac{K_z(t)}{K(t+z) - K(z)K(0)^{-1}K(t)} \leq 1 + f(z)
    \end{equation}
    holds for all $t,z > 0$ with the convention $\infty^{-1} = 0$, and the denominator is strictly positive.
\end{lemma}
\begin{proof}
    Since $L_0'\not\equiv 0$, the complete monotonicity of $L_0$ gives $L_0'<0$. Using \eqref{eq: Pi z formula}, we obtain $\Pi_z(\mathrm{d}t) = - \int_0^z E_K(r)L_0'(t+z-r)\, \mathrm{d}r\,\mathrm{d}t$, and hence $\mathrm{d}\Pi_z>0$. By complete monotonicity combined with $K \not \equiv 0$, we get $K(z) > 0$ for each $z > 0$. Hence, by the continuity of $K\ast E_K$ and the complete monotonicity of $K$, the function
    \[
        f(z) = \frac{|\beta| \sup_{u\in [0,z]}|(K\ast E_K)(u)|}{K(z)},\quad z>0,
    \]
    is well-defined and increasing with $f(0+) = 0$. Using $E_K = K + \beta (K \ast E_K)$, the continuity of $K\ast E_K$ with $(K \ast E_K)(0) = 0$, and $K(0) > 0$, we obtain for every $z>0$:
    \begin{equation}\label{eq:KEKEstimate}
        \left( 1 - f(z) \right) K(u) \leq E_K(u) \leq \left( 1 + f(z)\right)K(u), \qquad u \in (0,z].
    \end{equation}
    Hence, by \eqref{eq:KEKEstimate} we can estimate for all $t\in\R_+^*$:
    \begin{align*}
        0 < K_z(t) &= - \int_0^t \int_0^z K(t-r)L_0'(r+z-u)E_K(u)\, \mathrm{d}u\, \mathrm{d}r
        \\ &\le  (1 + f(z))\int_0^t \int_0^z K(t-r)(-L_0')(r+z-u)K(u)\, \mathrm{d}u \,\mathrm{d}r
        \\ &= (1 + f(z))\int_{[0,t]} K(t-r) \,\widetilde{\Pi}_z(\mathrm{d}r),
    \end{align*}
    and analogously obtain the lower bound $K_z(t)\ge (1 - f(z))\int_{[0,t]} K(t-r) \,\widetilde{\Pi}_z(\mathrm{d}r)$. Here we have set $\widetilde{\Pi}_z(t) = (\Delta_z K \ast L)(t) - \Delta_z (K\ast L)(t) = - \int_0^z K(r) L_0(t+z-r)\, \mathrm{d}r$, which corresponds to $\Pi_z$ with $\beta = 0$, i.e.\ $E_K = K$. By direct computation, we find that
    \begin{align*}
        \widetilde{\Pi}_z(t) &= K(0)^{-1}K(t+z) + \int_0^t K(t+z-r)L_0(r)\, \mathrm{d}r - 1,
    \end{align*}
    and, hence, it follows that $\widetilde{\Pi}_z(\mathrm{d}t) = (\Delta_z K' \ast L(t) + K(z)L_0(t))\,\mathrm{d}t$. From this, we deduce
    \begin{align*}
        \int_{[0,t]} K(t-r) \,\widetilde{\Pi}_z(\mathrm{d}r) &= K \ast (\Delta_z K' \ast L)(t) + K(z) (K \ast L_0)(t)
        \\ &= K(t+z) - K(z)K(0)^{-1}K(t),
    \end{align*}
    where we have used the associativity of the convolution and $K \ast L = 1$. Since $\widetilde{\Pi}_z$ is strictly increasing and $K>0$, we get $\int_{[0,t]} K(t-r) \,\widetilde{\Pi}_z(\mathrm{d}r) > 0$ and hence the denominator in~\eqref{eq:Kzkernelasymptoticsalternate} is strictly positive. 
\end{proof}

For singular kernels, i.e.\ $K(0)=\infty$, \eqref{eq:Kzkernelasymptoticsalternate} implies in combination with $f(0+)=0$ that for small $z>0$, $K_z$ becomes uniformly close to the shifted kernel $K(\cdot+z)$. Moreover, as the construction of the function~$f$ shows, when $\beta = 0$, we even have $K_z=K(\cdot+z)$ for all $z>0$, provided that $K(0) = \infty$. 

\begin{remark}
 By the definition of resolvents, the following properties are equivalent:
 \begin{enumerate}
     \item[(i)] There exist $\lambda, c\in\R_+$ such that $K(t) = c e^{-\lambda t}$ for $t \in\R_+$,
     \item[(ii)] The resolvent of the first kind is given by $L(\mathrm{d}t)=K(0)^{-1}\delta_0(\mathrm{d}t)+c\,\mathrm{d}t$ for some $c \in\R_+$.
     \item[(iii)] The resolvent of the first kind \eqref{eq: L resolvent} satisfies $L_0' \equiv 0$.
 \end{enumerate}
 It is a remarkable consequence of Lemma \ref{lemma:Kzkernelasymptotics} that for $L_0'\not \equiv 0$, it follows that $\mathrm{d}\Pi_z>0$ for all $z>0$, whence $X\ast\mathrm{d}\Pi_z$ in \eqref{eq:fancyexpectationformula} integrates over the full past trajectory of $X$. From this perspective, below we prove that this path-dependent integral is indeed not measurable with respect to the $\sigma$-algebra generated by $X_t$.
\end{remark}

Let $M_z:\R_+\longrightarrow \R^{2\times 2}$ be the matrix-valued function which maps $h\in\R_+$ to the positive semidefinite matrix
\[
        M_z(h) = \int_0^h \begin{pmatrix} K(r)^2 & K(r)K_z(r) \\ K(r) K_z(r) & K_z(r)^2 \end{pmatrix} \, \mathrm{d}r,
\]
and denote by $\lambda_{\min}(M_z(h))\ge 0$ its smallest eigenvalue. Based on formula \eqref{eq:fancyexpectationformula}, we provide a sufficient condition such that $X$ cannot be a Markov process.

\begin{theorem}\label{thm:affinedensitynonMarkov}
    Let $X$ be the Volterra square-root process on $\R_+$ with $\sigma, b > 0$ defined on a filtered probability space $(\Omega, \mathcal{F}, (\mathcal{F}_t)_{t \in \R_+}, \P)$. Assume that $L_0'\not\equiv 0$ and there exist $z, \gamma_* > 0$ such that
    \begin{align}\label{eq: 9}
        \gamma_* < \min \left\{ \gamma,\ \frac{1}{2} \right\} + \frac{\gamma}{2},
    \end{align}
    where $\gamma$ is given by \eqref{eq:kernelincrementconditions}, and $\gamma_*$ satisfies
    \begin{equation}\label{eq:eigenvalueasympAffine}
        \liminf_{h \searrow 0} h^{-2\gamma_*} \lambda_{\min}(M_z(h)) > 0.
    \end{equation} 
    Then $X$ is not a Markov process in the sense of \eqref{eq: Markov property} with respect to $(\mathcal{F}_t)_{t \in \R_+}$.
\end{theorem}
\begin{proof}
    Suppose that $X$ is a Markov process. Fix $0 < t < T$ such that $z=T-t$. Then by using $\E[X_T \, | \, \mathcal{F}_t ] = \E[X_T \, | \, X_t]$ combined with \eqref{eq:fancyexpectationformula}, we arrive at
    \begin{align}\label{eq: moment}
        \notag \int_{[0,t]} X_{t-s}\,\Pi_{T-t}(\mathrm{d}s) &= \E[X_T \, | \, X_t] - \int_0^{T-t}E_K(s)b\, \mathrm{d}s 
        \\ &\qquad  + \Pi_{T-t}(t)x_0 - ( \Delta_{T-t}E_K\ast L)(0)\hspace{0.02cm}X_t.
    \end{align}
    This shows that $\int_{[0,t]} X_{t-s}\,\Pi_{T-t}(\mathrm{d}s)$ is necessarily $\Sigma(X_t)$-measurable, where $\Sigma(X_t)$ denotes the $\sigma$-algebra generated by $X_t$. To obtain the desired contradiction, let us prove that the latter is impossible. 

    Define $Y^{z'}_t = \int_0^t X_{t-s}\,\Pi_{z'}(\mathrm{d}s) = (X \ast \mathrm{d}\Pi_{z'})_t$ for $z' > 0$. Then for $z'=z = T-t$ we have $Y_t^{T-t} = \int_{[0,t]} X_{t-s}\,\Pi_{T-t}(\mathrm{d}s)$ since the measure induced by $\Pi_{T-t}$ is atomless. By~\eqref{eq: moment} we find a measurable function $g_{t,T}$ such that $Y_t^{T-t} = g_{t,T}(X_t)$ holds a.s. Denote by $G_{t,T} = \{ (x, g_{t,T}(x)) \hspace{0.02cm} : \, x \in \R_+ \}$ the graph of $g_{t,T}$. From the Paley–Zygmund inequality combined with the formula for $\E[X_t]$ given in \cite[Lemma 4.2]{AbiJaLaPu19}, we obtain
    \[
        0 < \frac{(\E[X_t])^2}{\E[X_t^2]}\le \P[ X_t > 0 ]
        = \P\left[ (X_t, g_{t,T}(X_t)) \in G_{t,T} \cap ((0,\infty) \times \R) \right].
    \]
    Hence, to arrive at a contradiction, it suffices to prove that 
    \begin{equation}\label{eq:affinemeasure}
        \P\left[ (X_t, g_{t,T}(X_t)) \in G_{t,T} \cap ((0,\infty) \times \R) \right] = 0,
    \end{equation}
    which, in turn, follows once the law of the two-dimensional vector $(X_t, Y_t^{T-t}) = (X_t, g_{t,T}(X_t))$ is absolutely continuous with respect to the Lebesgue measure on $(0,\infty) \times \R$. 
    
    The desired absolute continuity follows from Theorem \ref{thm: density}. Indeed, by \cite[Lemma 2.1]{AbiJaLaPu19} we may write 
    \[
        Y^z_t = (\Pi_{z}(t)-\Pi_{z}(0))x_0 + \int_0^t K_z(t-s)(b + \beta X_s)\, \mathrm{d}s
        + \sigma \int_0^t K_z(t-s) \sqrt{X_s}\, \mathrm{d}B_s,
    \]
    and hence $Z = (X, Y^{T-t})^{\intercal}$ is a Volterra Ito-process given by 
    \[
        Z_s = \widetilde{g}(s) + \int_0^s \widetilde{K}(s-r)(b+\beta X_r)\, \mathrm{d}r
        + \sigma \int_0^s \widetilde{K}(s-r)\sqrt{X_r}\, \mathrm{d}B_r,\quad s\in\R_+,
    \]
    where $\widetilde{g}(s) = (\Pi_z(s)x_0 - \Pi_z(0)x_0, x_0)^{\intercal}$ and $\widetilde{K}^b(s)=\widetilde{K}^{\sigma}(s)=\widetilde{K}(s) := (K(s), K_z(s))^{\intercal}$. By Lemma \ref{lemma:Kzkernelasymptotics}, the Volterra kernel $K_z$ is bounded and hence satisfies $\int_0^h K_z(r)^2\, \mathrm{d}t \lesssim h$. Consequently, we obtain $\int_0^h |\widetilde{K}(r)|^2\, \mathrm{d}r \lesssim h + h^{2\gamma} \lesssim h^{2(\gamma \wedge 1/2)}$. For the lower bound, we obtain from assumption \eqref{eq:eigenvalueasympAffine} for every $\xi\in\R^2$ with $|\xi|=1$:
    \[
        \int_0^h |\widetilde{K}(r)^{\intercal}\xi|^2\, \mathrm{d}r = \int_0^h \left| K(r)\xi_1 + K_z(r)\xi_2\right|^2\, \mathrm{d}r = \xi^{\top}M_z(h)\hspace{0.02cm} \xi
    \geq C_* h^{2\gamma_*}.
    \]
    Hence, $\widetilde{K}$ is $\gamma_*$-nondegenerate, the regularity condition \eqref{eq: upper bound} holds with parameter $\gamma \wedge \frac{1}{2}$, and by \eqref{eq:kernelincrementconditions}, condition (A) is satisfied with $\gamma_K=\gamma$. Therefore, \eqref{eq: 9} guarantees \eqref{eq: H condition} and Theorem \ref{thm: density} combined with $\Gamma_{\sigma,t} = (0,\infty)$, shows that $\nu_t(A) = \P[ Z_t \in A, \ X_t > 0 ] = \P[ Z_t \in A\cap ((0,\infty) \times \R)]$ has a density. This proves the assertion. 
\end{proof} 

Recall that the Markov property is preserved when passing to a coarser filtration with respect to which $X$ is still adapted. Hence, under the above conditions, $X$ is also not a Markov process with respect to any finer filtration $\mathcal{G}_t \supseteq\mathcal{F}_t$.

Since the proof of \cite[Theorem 4.5\hspace{0.03cm}(i)]{AbiJaLaPu19} is independent of the particular form of $\sigma$, it also holds for general SVEs with affine linear drifts. In this case, $\sigma$ is H\"older continuous with parameter $\chi_{\sigma} \in (0,1]$, $\Gamma_{\sigma,t} = \{x \in \R \ : \ \sigma(x) \neq 0\}$, and we need to assume that $\P[ \sigma(X_{t_0}) \neq 0] > 0$ for some $t_0 > 0$. Then condition \eqref{eq: 9} becomes $\gamma_*<(\gamma\wedge\frac{1}{2})+\gamma\hspace{0.02cm}\chi_{\sigma}$. 

\begin{example}\label{example:affineargumentnotoptimal}
    Suppose that $\beta = 0$ and $K$ is, additionally to the assumptions above, also regularly varying in $t = 0$ with index $H - \frac{1}{2}$, $H \in (0, \frac{1}{2})$. It follows from Lemma~\ref{lemma:Kzkernelasymptotics} that $K_z(0) = K(z) \neq 0$, and hence $K_z$ is regularly varying at $t = 0$ of order $\widetilde{H} = 0$. Consequently, we obtain $\gamma\wedge\frac{1}{2} = H$, and it follows from Lemma~\ref{lemma: fractional like} below that $\gamma_*=\frac{1}{2}$ for all $z>0$. Therefore, Theorem~\ref{thm:affinedensitynonMarkov} is applicable provided that $H>1/3$, while for general Hölder continuous $\sigma$ the condition becomes $H>(2+2\chi_{\sigma})^{-1}$.
\end{example}

While the affine structure of the processes under consideration provides a useful intuition for their path-dependent nature, the practical applicability of Theorem \ref{thm:affinedensitynonMarkov} for establishing the failure of the Markov property is restricted to singular kernels that are ``not too rough''. Indeed, if $K$ is regular, i.e.\ $K(0) < \infty$, so that $K'$ is bounded, then $\gamma = 1/2$ and $\gamma_* \geq 1$ (cf.~\eqref{eq:lambdaminregularasymp}), and in this case condition \eqref{eq: 9} can never be satisfied. The cases of very small Hurst parameters and regular kernels together with general dynamics will be addressed in the following sections by means of Markovian lifts. In contrast to the arguments developed above, one advantage of the Markovian lift approach lies in the flexibility of replacing $K_z$ with other admissible Volterra kernels.

\section{Markovian lifts}\label{section:markovianliftapproach}

\subsection{Abstract formulation}\label{section:abstractLifts}

Below, we introduce a framework for Hilbert space-valued Markovian lifts of \eqref{eq:generalSVIE} developed in \cite{BBCF25} for infinite-dimensional SVEs, but now adjusted to our finite-dimensional setting. Let us suppose that the following condition is satisfied, cf. \cite[Assumption~A]{BBCF25}:
\begin{enumerate}
    \item[(B)] There exist separable Hilbert spaces $\mathcal{H}, \mathcal{V}$ such that $\mathcal{H}\cap\mathcal{V}$ is dense in $\mathcal{V}$, and there exists a $C_0$-semigroup $(S(t))_{t\geq 0}$ on $\mathcal{H}$ that leaves $\mathcal{H} \cap \mathcal{V}$ invariant, is bounded on~$\mathcal{V}$, and its unique extension onto $\mathcal{V}$, which we denote also by $(S(t))_{t \geq 0}$, is again a $C_0$-semigroup. This semigroup satisfies $S(t) \in L(\mathcal{H}, \mathcal{V})$ for $t > 0$, and there exists $\rho \in[0,1/2)$ such that for every $T>0$ we have a constant $C_T(\mathcal{H},\mathcal{V})>0$ with
    \begin{equation}\label{eq:abstractoperatornormestimate}
    \| S(t)\|_{L(\mathcal{H},\mathcal{V})}\leq C_{T}(\mathcal{H},\mathcal{V})(1+t^{-\rho}), \qquad t\in (0,T].
    \end{equation}
    Finally, there exist bounded linear operators $\Xi: \mathcal{V} \longrightarrow \R^d$ and $\xi_b, \xi_{\sigma}: \R^d \longrightarrow \mathcal{H}$.
\end{enumerate}

When studying Markovian lifts, we shall always suppose that assumption (B) is satisfied. Below, we study the following abstract stochastic equation on $\mathcal{V}$:
\begin{align}\label{eq: abstract markovian lift}
    \mathcal{X}_t = S(t)\xi + \int_0^t S(t-s)\hspace{0.02cm}\xi_b\hspace{0.02cm} b(s,\Xi \mathcal{X}_s)\, \mathrm{d}s + \int_0^t S(t-s)\hspace{0.02cm}\xi_{\sigma} \hspace{0.02cm}\sigma(s,\Xi \mathcal{X}_s)\, \mathrm{d}B_s,\quad t\in\R_+,
\end{align}
where $B$ is an $m$-dimensional Brownian motion, $b: \R_+\times\R^d \longrightarrow \R^d$ denotes the drift, $\sigma: \R_+\times\R^d \longrightarrow \R^{d \times m}$ the diffusion coefficient and $\xi \in \mathcal{V}$ describes the initial condition. Equation \eqref{eq: abstract markovian lift} can be seen as a mild formulation of a (time-inhomogeneous) stochastic evolution equation on the Hilbert space $\mathcal{V}$ with the generator $(\mathcal{A}, D(\mathcal{A}))$ determined by the semigroup $(S(t))_{t \geq 0}$. Its existence, uniqueness and stability properties have been studied in \cite{BBCF25} for infinite-dimensional settings which contain \eqref{eq: abstract markovian lift} as a particular case. 

In this framework, stochastic Volterra processes appear as finite-dimensional projections $\Xi \mathcal{X}$ of \eqref{eq: abstract markovian lift}, where $\xi_b,\xi_{\sigma} \in L(\R^d, \mathcal{H})$ play the role of abstract lifts for the Volterra kernels $K^b,K^{\sigma}$. Namely, given a continuous $\mathcal{V}$-valued solution $\mathcal{X}$ of \eqref{eq: abstract markovian lift}, it is easy to see that $X_t := \Xi \mathcal{X}_t$, $t\in\R_+$, solves \eqref{eq:generalSVIE} with the kernel matrices and the initial curve given~by
\begin{align}\label{eq: lift condition}
    K^a(t) := \Xi S(t) \xi_a \in L(\R^d, \R^d) \ \text{ and } \ g(t) := \Xi S(t)\xi \in \R^d, \qquad t > 0,
\end{align}
for $a \in \{b,\sigma\}$, and $g(0)=\Xi\xi$. To state the converse direction, let us define the class of admissible initial conditions by
\begin{equation}\label{eq: admissible input curves}
    \mathcal{G}_p = \{ g = \Xi S(\cdot) \xi \ : \ \xi \in L^p(\Omega, \mathcal{F}_0, \P; \mathcal{V}) \}, \quad p \in [1,\infty].
\end{equation}
The following proposition illustrates how the Markovian lift $\mathcal{X}$, i.e.\ a solution to \eqref{eq: abstract markovian lift}, can be constructed from the associated Volterra process $X$.

\begin{proposition}\label{prop: Markov lift}
    Suppose that condition (B) is satisfied, consider $b \in B(\R_+; C^0(\R^d; \R^d))$ and $\sigma \in B(\R_+; C^0(\R^d; \R^{d \times m}))$, and let $X$ be a continuous weak solution of \eqref{eq:generalSVIE} with kernels $K^a = \Xi S(\cdot)\xi_a$, $a \in \{b,\sigma\}$, and $g \in \mathcal{G}_p$ for some $p > 2$ satisfying
    \begin{equation}\label{eq:prestrictioninitialcond}
        \frac{1}{p} + \rho < \frac{1}{2},
    \end{equation}
    defined on a filtered probability space $(\Omega, \mathcal{F}, (\mathcal{F}_t)_{t \in \R_+}, \P)$. Then for each $T > 0$ there exists a continuous weak solution $\mathcal{X} \in L^p(\Omega, \P; C([0,T];\mathcal{V}))$ of~\eqref{eq: abstract markovian lift} that is defined on the same filtered probability space w.r.t.\ the identical $(\mathcal{F}_t)_{t \in \R_+}$-Brownian motion. This solution satisfies $X = \Xi \mathcal{X}$ on $[0,T]$ and 
    \begin{equation}\label{eq:abstractLiftmomentbounds}
        \E\left[ \sup_{t \in [0,T]}|X_t|^p + \sup_{t \in [0,T]} \| \mathcal{X}_t\|_{\mathcal{V}}^p \right] < \infty.
    \end{equation}
\end{proposition}
\begin{proof}
    First note that $|g(t)| \leq \| \Xi \|_{L(\mathcal{V}, \R^d)}\sup_{t \in [0,T]}\|S(t)\|_{L(\mathcal{V})}\|\xi\|_{\mathcal{V}}$, whence $g$ is bounded in $L^p(\Omega, \mathcal{F}_0, \P; \R^d)$ by the boundedness of $(S(t))_{t\ge 0}$ on $\mathcal{V}$ and the definition of $\mathcal{G}_p$. Moreover, using assumption (B) combined with~\eqref{eq: lift condition}, we obtain from \eqref{eq:abstractoperatornormestimate} for every $t,T>0$ with $t\le T$ and $a \in \{b,\sigma\}$:
    \begin{align}\label{eq: K pointwise bound}
        |K^a(t)| \leq \| \Xi \|_{L(\mathcal{V}, \R^d)}\hspace{0.02cm}C_T(\mathcal{H}, \mathcal{V})\hspace{0.02cm}\|\xi_a\|_{L(\R^d,\mathcal{H})}\hspace{0.02cm}(1 + t^{-\rho}),
    \end{align}
    and hence $K^a \in L_{\mathrm{loc}}^2(\R_+; \R^{d \times d})$ as $\rho\in [0,1/2)$. Therefore, \cite[Proposition 4.1, Remark~4.2]{F24}, which straightforwardly generalizes to our time-dependent coefficients as the linear growth condition holds uniformly on $[0,T]$, see also \cite[Lemma 3.4]{PROMEL2023291} for deterministic~$g$, allows us to verify 
    \begin{align}\label{eq: moment bound X}
        \sup_{t \in [0,T]}\E[ |X_t|^p ] < \infty.
    \end{align}
    Next, let us define 
    \begin{equation}\label{eq:MarkovianLiftconstruction}
        \mathcal{X}_t = S(t)\xi + \int_0^t S(t-s)\hspace{0.02cm}\xi_b\hspace{0.02cm} b(s,X_s)\, \mathrm{d}s + \int_0^t S(t-s)\hspace{0.02cm}\xi_{\sigma} \hspace{0.02cm}\sigma(s,X_s)\, \mathrm{d}B_s, \quad t\in\R_+.
    \end{equation}
    Observe that $\mathbb{E}\big[\sup_{t\in [0,T]}\|S(t)\xi\|_{\mathcal{V}}^p\big]\le\sup_{t\in [0,T]}\|S(t)\|_{L(\mathcal{V})}^p\mathbb{E}\big[\|\xi\|_{\mathcal{V}}^p\big]<\infty$, whence the initial curve fulfills the required integrability. Moreover, $S(\cdot)\xi$ has continuous paths as $(S(t))_{t\ge 0}$ is a $C_0$-semigroup and $\xi\in\mathcal{V}$. Combining the linear growth condition of the coefficients $b$, $\sigma$ with $\xi_a \in L(\R^d,\mathcal{H})\cap L_2(\R^d,\mathcal{H})$, $a\in\{b,\sigma\}$, where $L_2(\R^d,\mathcal{H})$ denotes the space of Hilbert-Schmidt operators from $\R^d$ to $\mathcal{H}$, and the moment bound \eqref{eq: moment bound X} shows $\xi_b \hspace{0.02cm}b(\cdot,X)\in L^{\infty}([0,T];L^p(\Omega,\P;\mathcal{H}))$ and $\xi_{\sigma} \hspace{0.02cm}\sigma(\cdot,X)\in L^{\infty}([0,T];L^p(\Omega,\P; L_2(\R^m,\mathcal{H})))$. Thus, an application of \cite[Proposition 2.2]{BBCF25} yields $\mathcal{X} \in L^p(\Omega, \P; C([0,T];\mathcal{V}))$ for all $T>0$. By an application of $\Xi \in L(\mathcal{V}, \R^d)$ and taking into account \eqref{eq: lift condition}, we observe that 
    \begin{align}
        \Xi \mathcal{X}_t = g(t) + \int_0^t K^b(t-s)\hspace{0.02cm} b(s,X_s)\, \mathrm{d}s + \int_0^t K^{\sigma}(t-s)\hspace{0.02cm}\sigma(s,X_s)\, \mathrm{d}B_s=X_t.
    \end{align}    
    Hence $X = \Xi \mathcal{X}$ holds up to indistinguishability and $\mathcal{X}$ is a continuous solution to \eqref{eq: abstract markovian lift}. In particular, $\mathcal{X}$ satisfies the desired moment bound, and thus we can conclude
    \[
        \E \left[\sup_{t \in [0,T]}|X_t|^p\right] \leq \| \Xi \|_{L(\mathcal{V}, \R^d)}^p \hspace{0.02cm}\E \left[\sup_{t \in [0,T]}\|\mathcal{X}_t\|_{\mathcal{V}}^p\right] < \infty,
    \]
    which proves the claimed moment bound also for $X$.
\end{proof}

The required integrability of $g$ with $p$ satisfying \eqref{eq:prestrictioninitialcond} is only relevant when~$g$ is non-deterministic. Moreover, notice that \eqref{eq:prestrictioninitialcond} reduces to $p>2$ when \eqref{eq:abstractoperatornormestimate} is satisfied with $\rho=0$. Such an assumption stems from \cite[Proposition~2.2]{BBCF25} and hence from an application of the factorization method \cite[Proposition 5.9 and Theorem 5.10]{DaPrato_Zabczyk_2014}, which provides a continuous modification for the stochastic convolution in~\eqref{eq: abstract markovian lift}.

If uniqueness in law holds for \eqref{eq: abstract markovian lift}, then, under the above conditions, it defines a Markov process, justifying the notion of a Markovian lift. For completely monotone scalar kernels and time-homogeneous coefficients, a general proof under weak uniqueness is given in \cite[Lemma 4.3]{hamaguchi2023weak}, the $C_b$-Feller property for Lipschitz coefficients is established in \cite[Theorem 2.19]{H23}, and \cite[Corollary~2.6]{BBCF25} studies the case of operator-valued kernels with an analogue for time-dependent coefficients provided in \cite[Corollary~2.5]{BBCF25}.

\subsection{Dimension of the Markovian lift}

Let $\mathcal{X}$ be the Markovian lift constructed in Proposition \ref{prop: Markov lift}. 
We estimate its dimension through the range of the associated covariance operator. Namely, for each fixed $t>0$, the bilinear form $\phi_t$ on $\mathcal{V}$ defined by $(u,v) \longmapsto \E\left[ \langle u, \mathcal{X}_t\rangle_{\mathcal{V}}\langle \mathcal{X}_t, v \rangle_{\mathcal{V}} \right]$ is positive semidefinite and symmetric. Applying the Cauchy-Schwarz inequality twice yields $\phi_t(u,v)\le C\hspace{0.02cm}\|u\|_{\mathcal{V}}\|v\|_{\mathcal{V}}$ with $C:=\E[\|\mathcal{X}_t\|_{\mathcal{V}}^2]<\infty$ due to \eqref{eq:abstractLiftmomentbounds}. By the Riesz representation theorem, there exists a unique \textit{covariance operator} $Q_t \in L(\mathcal{V})$ such that
\begin{align}\label{eq: covariance operator}
    \E\left[ \langle u, \mathcal{X}_t\rangle_{\mathcal{V}}\langle \mathcal{X}_t, v\rangle_{\mathcal{V}} \right] = \langle u, Q_t v \rangle_{\mathcal{V}}, \qquad u,v \in \mathcal{V}.
\end{align}
Clearly, $Q_t$ is self-adjoint. Since $\sum_{n=1}^{\infty}\phi_t(e_n, e_n) \leq \E[\|\mathcal{X}_t\|_{\mathcal{V}}^2] < \infty$ for any orthonormal basis $(e_n)_{n \geq 1} \subseteq \mathcal{V}$ by Parseval's identity, the operator $Q_t$ is also trace-class. 

Under this notation, we obtain the orthogonal decomposition $\mathcal{V} = \overline{\mathrm{im}(Q_t)} \oplus \mathrm{ker}(Q_t)$, where $\mathrm{ker}(Q_t) = \{ y \in \mathcal{V} \ | \ Q_ty = 0 \}$ denotes the null-space of $Q_t$, and $\mathrm{im}(Q_t)$ its range. The dimension of $\mathcal{X}_t$ is defined as the dimension of the linear space $\overline{\mathrm{im}(Q_t)}$, which is justified by the following observation: Evaluating \eqref{eq: covariance operator} for $u = v$ shows that $v \in \mathrm{ker}(Q_t)$ if and only if $\langle v, \mathcal{X}_t\rangle_{\mathcal{V}} = 0$ a.s., which is equivalent to $v \in \mathrm{supp}(\P \circ \mathcal{X}_t^{-1})^{\perp}$. Hence,
\begin{displaymath}
    \overline{\mathrm{span}\ \mathrm{supp}(\P \circ \mathcal{X}_{t}^{-1})}^{\perp} = \mathrm{supp}(\P \circ \mathcal{X}_t^{-1})^{\perp} = \mathrm{ker}(Q_t) = \overline{\mathrm{im}(Q_t)}^{\perp},
\end{displaymath}
and thus we obtain
\[
    \overline{\mathrm{im}(Q_t)} = \overline{\mathrm{span}\ \mathrm{supp}(\P \circ \mathcal{X}_{t}^{-1})}, \qquad t > 0.
\]
Recall that $\Gamma_{\sigma,t}$ is defined in \eqref{eq: Gamma sigma}. The following establishes a sufficient condition for the infinite-dimensionality of $\overline{\mathrm{im}(Q_t)}$. 

\begin{corollary}\label{cor: dimension}
    Suppose that conditions (A) and (B) hold, $b \in B(\R_+; C^{\chi_b}(\R^d; \R^d))$, and $\sigma \in B(\R_+; C^{\chi_{\sigma}}(\R^d; \R^{d\times m}))$ with $\chi_b \in [0,1]$ and $\chi_{\sigma} \in (0,1]$. Let $X$ be a continuous weak solution of~\eqref{eq:generalSVIE} with $K^a=\Xi S(\cdot)\xi_a$, $a \in \{b,\sigma\}$, and $g=\Xi S(\cdot)\xi \in \mathcal{G}_p$ for some $p > 2$ satisfying~\eqref{eq:prestrictioninitialcond}, and let $\mathcal{X}$ be the corresponding Markovian lift of $X$. For $N \geq 1$ and functionals $\Xi_1, \dots, \Xi_N \in \mathcal{V}^*$, define for $a \in \{b,\sigma\}$:
    \[
        \widetilde{K}^a = \big(\Xi_1 S(\cdot)\xi_a, \dots, \Xi_N S(\cdot)\xi_a\big)^{\intercal} \in L_{\mathrm{loc}}^2(\R_+,\R^{N\times d}).
    \] 
    Suppose that $\widetilde{K}^b, \widetilde{K}^{\sigma}$ satisfy \eqref{eq: upper bound}, $\widetilde{K}^{\sigma}$ is $\gamma_*$-nondegenerate and \eqref{eq: H condition} holds. Then for each $t_0 > 0$ with $\P[X_{t_0} \in \Gamma_{\sigma,t_0}] > 0$, one has  
    \[
        \mathrm{dim}\, \overline{\mathrm{im}(Q_{t_0})} \geq N.
    \]
    In particular, if these conditions hold for each $N \geq 1$, then even $\mathrm{dim}\, \overline{\mathrm{im}(Q_{t_0})} = \infty$, i.e.\ $\mathcal{X}_{t_0}$ is necessarily infinite-dimensional.
\end{corollary}
\begin{proof}
    Let us first show that the operators $\Xi_1, \dots, \Xi_N$ are linearly independent. Take scalars $a_1,\dots, a_N$ such that $\sum_{j=1}^N a_j \Xi_j = 0$. Evaluating this on $S(t)\xi_{\sigma} \in L(\R^d,\mathcal{V})$ with $t > 0$ arbitrary but fixed, gives
    \[
        \big(\widetilde{K}^{\sigma}(t)^{\intercal}a\big)^{\intercal} = a_1 \Xi_1 S(t)\xi_{\sigma}  + \dots + a_N \Xi_N S(t)\xi_{\sigma} = 0, \qquad t > 0,
    \]
    where $a = (a_1,\dots, a_N)^{\intercal}$. By assumption, $\widetilde{K}^{\sigma}$ is $\gamma_*$-nondegenerate according to Definition~\ref{def: nondegeneracy general}, and hence $0 = \int_0^h |\widetilde{K}^{\sigma}(t)^{\intercal}a|^2\, \mathrm{d}t \geq C_* h^{2\gamma_*}|a|^2$ for sufficiently small $h>0$. This yields $a = 0$ and shows that also $\Xi_1,\dots, \Xi_N$ are linearly independent.

    By assumption, we can apply Theorem \ref{thm: density} and thus the law of $(\Xi_1 \mathcal{X}_{t_0}, \dots, \Xi_N \mathcal{X}_{t_0})$ is absolutely continuous with respect to the Lebesgue measure on $\{X_{t_0} \in \Gamma_{\sigma,t_0}\}$. In particular, we obtain
    \begin{align}\label{eq: dimension}
        \P[ \Xi_j \mathcal{X}_{t_0} = 0] 
        = \P[ \Xi_j \mathcal{X}_{t_0} = 0, \ X_{t_0} \not \in \Gamma_{\sigma,t_0}] 
        \leq \P[ X_{t_0} \not \in \Gamma_{\sigma,t_0}] < 1
    \end{align}
    and hence $\P[ \Xi_j \mathcal{X}_{t_0} \neq 0] > 0$ for each $j \in\{1,\dots, N\}$. Since $\Xi_j \in \mathcal{V}^*$, there exists $v_j \in \mathcal{V}$ such that $\Xi_j = \langle v_j, \cdot\rangle_{\mathcal{V}}$. Hence, we obtain
    \[
     0 < \E[ (\Xi_j \mathcal{X}_{t_0})^2 ] = \E[ \langle v_j, \mathcal{X}_{t_0}\rangle^2 ] = \langle v_j, Q_{t_0}v_j \rangle_{\mathcal{V}}
    \]
    and thus $v_1,\dots, v_N \not \in \mathrm{ker}(Q_{t_0})$. Let us denote by $P_{\mathrm{im}}$ and $P_{\mathrm{ker}}$ the orthogonal projection operators onto $\overline{\mathrm{im}(Q_{t_0})}$ and $\mathrm{ker}(Q_{t_0})$, respectively. Since $v_j \not \in \mathrm{ker}(Q_{t_0})$, and $v_j = P_{\mathrm{im}}v_j + P_{\mathrm{ker}}v_j$, we obtain $P_{\mathrm{im}}v_j \neq 0$.

    Suppose that $\mathrm{dim}\,\overline{\mathrm{im}(Q_{t_0})} = n < N$. Since $P_{\mathrm{im}}v_1,\dots, P_{\mathrm{im}}v_N \neq 0$, there exist scalars $c_1,\dots, c_N \in \R$, not all equal to zero, such that
    $\sum_{j=1}^N c_j P_{\mathrm{im}}v_j = 0$. Now define the new vector $v = \sum_{j=1}^N c_j v_j$. Since the functionals $\Xi_1,\dots, \Xi_N$ are linearly independent, also $v_1,\dots, v_N$ are linearly independent, and hence $v \neq 0$. But $P_{\mathrm{im}}v = \sum_{j=1}^N c_j P_{\mathrm{im}}v_j = 0$ by construction. Since $\mathcal{V} = \overline{\mathrm{im}(Q_{t_0})} \oplus \mathrm{ker}(Q_{t_0})$, we obtain $v \in \mathrm{ker}(Q_{t_0})$. The latter leads to a contradiction as shown below.

    Note that $Z := \langle v, \mathcal{X}\rangle_{\mathcal{V}}$ is a Volterra Ito-process given by
    \[
        Z_t = g_v(t) + \int_0^t K^b_v(t-s)\hspace{0.02cm}b(s,X_s)\, \mathrm{d}s + \int_0^t K^{\sigma}_v(t-s)\hspace{0.02cm}\sigma(s,X_s)\, \mathrm{d}B_s,\quad t\in\R_+,
    \]
    where $g_v(t) = \langle v, S(t)\xi \rangle_{\mathcal{V}}$ and $K^a_v(t) = \langle v, S(t)\xi_a \rangle_{\mathcal{V}} = \sum_{j=1}^N c_j \Xi_j S(t)\xi_a\in\R^{1\times d}$ for $a \in \{b,\sigma\}$. Let $c = (c_1,\dots, c_N)^{\intercal} \in \R^N$. Then $\int_0^h |K^{\sigma}_v(t)|^2\, \mathrm{d}t = \int_0^h \big| \widetilde{K}^{\sigma}(t)^{\intercal}c \big|^2\, \mathrm{d}t \geq C_* h^{2\gamma_*} |c|^2 $ for sufficiently small $h>0$ and hence $K^{\sigma}_v$ is $\gamma_*$-nondegenerate by $c\neq 0$. Thus, again by Theorem \ref{thm: density}, the law of $Z_{t_0}$ is absolutely continuous with respect to the Lebesgue measure on the event $\{X_{t_0} \in \Gamma_{\sigma,t_0}\}$. Arguing as in~\eqref{eq: dimension}, we obtain $\P[ Z_{t_0} \neq 0] > 0$, and hence
    \[
        0 < \E\left[ Z_{t_0}^2 \right] = \E\left[ \langle v, \mathcal{X}_{t_0}\rangle_{\mathcal{V}}^2 \right] = \langle v, Q_{t_0}v \rangle.
    \]
    This is a contradiction to $v \in \mathrm{ker}(Q_{t_0})$. Therefore, we may conclude $\mathrm{dim}\,\overline{\mathrm{im}(Q_{t_0})} \geq N$, and the assertion is proved.
\end{proof}

Remark that $K_v^b$, $K_v^{\sigma}$ in the above proof are in general not diagonal-like, except when $\widetilde{K}^b, \widetilde{K}^{\sigma}$ are diagonal-like and there exists $i_0\in\{1,\dots,d\}$ such that $S_{i_0}=\{1,\dots,N\}$. In the subsequent sections, we construct a large class of functionals $\Xi_j$ that satisfy these assumptions. In particular, for fractional kernels, we will show that the Markovian lift is necessarily infinite-dimensional.

\subsection{Conditional distribution given the Volterra process}

It is natural to expect that the failure of the Markov property for the stochastic Volterra process is encoded in a nontrivial behavior of its Markovian lift $\mathcal{X}$, i.e.\ $\mathcal{X}_t$ should not be a function of $X_t=\Xi\mathcal{X}_t$ for some $t>0$. In mathematical terms, this heuristic can be described by the properties of the conditional distribution of $\mathcal{X}_t$ given $X_t = x$, formally determined by the relation
\[
    \P[ \mathcal{X}_t \in \mathrm{d}y, \, X_t \in \mathrm{d}x ]
    = \P[ \mathcal{X}_t \in \mathrm{d}y \, | \, X_t = x]\hspace{0.03cm}\P[X_t \in \mathrm{d}x].
\]
Intuitively, if the Volterra process is a Markov process, then it already contains all the necessary information, and the Markovian lift does not provide any further insights. Thus, for such cases one expects that $\P[ \mathcal{X}_t \in \mathrm{d}y \, | \, X_t = x ] = \delta_{h_t(x)}$ should hold for some measurable function $h_t:\R^d\longrightarrow\mathcal{V}$. Our next result, therefore, provides another perspective on the failure of the Markov property and complements our findings of the previous section.

\begin{corollary}
    Suppose that conditions (A) and (B) hold, $b \in B(\R_+; C^{\chi_b}(\R^d; \R^d))$, and $\sigma \in B(\R_+; C^{\chi_{\sigma}}(\R^d; \R^{d\times m}))$ with $\chi_b \in [0,1]$ and $\chi_{\sigma} \in (0,1]$. Let $X$ be a continuous weak solution of~\eqref{eq:generalSVIE} with $K^a=\Xi S(\cdot)\xi_a$, $a \in \{b,\sigma\}$, and $g=\Xi S(\cdot)\xi \in \mathcal{G}_p$ for some $p > 2$ satisfying~\eqref{eq:prestrictioninitialcond}, and let $\mathcal{X}$ be the corresponding Markovian lift of $X$. If there exist $N \geq 1$, and functionals $\Xi_1, \dots, \Xi_N \in \mathcal{V}^*$ such that 
    \[
        \widetilde{K}^a = \big(\Xi S(\cdot)\xi_a, \Xi_1 S(\cdot)\xi_a, \dots, \Xi_N S(\cdot)\xi_a\big)^{\intercal} \in L_{\mathrm{loc}}^2\big(\R_+,\R^{(d+N)\times d}\big),\quad a\in\{b,\sigma\},
    \] 
    $b$ and $\sigma$ satisfy the conditions of Theorem \ref{thm: density} or Theorem \ref{thm: density diagonal}, respectively, then for any $t_0 > 0$ such that $\P[ X_{t_0} \in \Gamma_{\sigma, t_0}] > 0$ and each version of $\P[ \mathcal{X}_{t_0} \in \mathrm{d}y \, | \, X_{t_0} = x]$: 
    \[
        \mathrm{Leb}_d\big( x \in \Gamma_{\sigma,t_0} \, \big| \, \# \mathrm{supp}\big(\P[ \mathcal{X}_{t_0} \in \mathrm{d}y \ | \ X_{t_0} = x]\big) = \infty \big) > 0.
    \]
\end{corollary}
\begin{proof}
    Fix any $t_0 > 0$ such that $\P[ X_{t_0} \in \Gamma_{\sigma,t_0}] > 0$. Either by Theorem \ref{thm: density} or Theorem~\ref{thm: density diagonal}, we observe that 
    \[
        A \times B \longmapsto \nu^N(A \times B) := \P\big[ (\Xi_1 \mathcal{X}_{t_0}, \dots, \Xi_N \mathcal{X}_{t_0}) \in A, \, X_{t_0} \in B \big]
    \]
    is absolutely continuous on $\mathcal{B}(\R^N) \otimes \mathcal{B}(\Gamma_{\sigma,t_0})$ with density denoted by $\overline{\nu}^N(y,x)$. Hence, also its regular conditional distribution 
    \begin{align}\label{eq: reg conditional}
        \nu^N( A \, | \, x) = \P[(\Xi_1 \mathcal{X}_{t_0}, \dots, \Xi_N \mathcal{X}_{t_0}) \in A \, | \, X_{t_0} = x]
    \end{align}
    has a version that is absolutely continuous on $\mathcal{B}(\R^N)$ with density $\overline{\nu}^N(y\,|\,x)$ given by
    \[
        \overline{\nu}^N(y \, | \, x) = \1_{\left\{\int_{\R^N} \overline{\nu}^N(y,x)\, \mathrm{d}y > 0 \right\}}(x) \frac{ \overline{\nu}^N(y,x)}{\int_{\R^N} \overline{\nu}^N(y,x)\, \mathrm{d}y}, \qquad x\in\Gamma_{\sigma,t_0}.
    \] 
    In particular, the support of $\nu^N(\cdot\, |\, x)$ necessarily contains infinitely many elements for all $x\in\Gamma_{\sigma,t_0}$ with $p_{t_0}(x):=\int_{\R^N} \overline{\nu}^N(y,x)\, \mathrm{d}y>0$. Hence,
    \begin{align}\label{eq: bla bla}
        0<\P[X_{t_0}\in C]=\int_{\Gamma_{\sigma,t_0}} \1_{\{ x \in C \, | \, \# \mathrm{supp}(\nu^N(\cdot\,|\,x)) = \infty \}}\, p_{t_0}(x)\,\mathrm{d}x,
    \end{align}
    where we used $\P[X_{t_0} \in \mathrm{d}x] = p_{t_0}(x)\,\mathrm{d}x$ on $\Gamma_{\sigma,t_0}$, and $C:= \{x \in \Gamma_{\sigma, t_0} \, | \, p_{t_0}(x) > 0 \}$ necessarily has positive probability by $\P[ X_{t_0} \in \Gamma_{\sigma,t_0}] > 0$. Let $\nu(\mathrm{d}y \, | \, x) = \P[ \mathcal{X}_{t_0} \in \mathrm{d}y \, | \, X_{t_0} = x]$ be any version of the regular conditional distribution. Then
    \begin{align*}
        \int_{\Gamma_{\sigma,t_0}} \1_{\{ x \in C \, | \, \# \mathrm{supp}(\nu(\cdot\,|\, x)) = \infty\}} \hspace{0.02cm}p_{t_0}(x)\, \mathrm{d}x
        &= \P \circ X_{t_0}^{-1}\big( x \in C \, | \, \# \mathrm{supp}(\nu(\cdot\,|\,x)) = \infty \big)
        \\ &\geq  \P \circ X_{t_0}^{-1}\big( x \in C \, | \, \# \mathrm{supp}(\nu^N(\cdot\,|\,x)) = \infty \big)
        \\ &= \int_{\Gamma_{\sigma,t_0}} \1_{\{ x \in C \, | \, \# \mathrm{supp}(\nu^N(\cdot\,|\,x)) = \infty \}} \hspace{0.02cm}p_{t_0}(x)\, \mathrm{d}x > 0,
    \end{align*}
    where the last inequality follows from \eqref{eq: bla bla}. Hence, 
    \begin{displaymath}
        \big\{ x\in C \, | \, \1_{\{  \# \mathrm{supp}(\nu( \cdot \, | \, x )) = \infty\}} = 1 \big\} = \big\{ x\in C \, | \, \# \mathrm{supp}\big(\nu( \cdot \, | \, x )\big) = \infty \big\}
    \end{displaymath}
    also has positive Lebesgue measure. Since $C\subseteq \Gamma_{\sigma,t_0}$, this proves the assertion.
\end{proof}

Let us stress that, while this result hints at the failure of the Markov property, it does not provide a rigorous proof thereof since in the above we do not address if the Markov property for $X$ implies $\P[ \mathcal{X}_t \in \mathrm{d}y \, | \, X_t = x ] = \delta_{h_t(x)}$ for some measurable function $h_t$.

\section{Failure of the Markov property through admissible perturbations}\label{section:failure}

\subsection{Admissible perturbations}\label{subsection:abstractMarkovproperty}

In the present section, we suppose that assumption~(B) from the previous section is satisfied. To extract a nondegenerate behavior from the Markovian lift, we study the distributional properties of the $\R^d \times \R$-valued random vector $\big(\Xi \mathcal{X}_t, \widetilde{\Xi}\mathcal{X}_t \big)$. In contrast to the previous section, here $\widetilde{\Xi}$ is not an arbitrary functional, but needs to be constructed from \textit{admissible perturbations} $\widetilde{k}^a$ of the original Volterra kernels $K^a$, $a \in \{b,\sigma\}$, and $\widetilde{g}$ has to be constructed from an admissible perturbation of the initial condition $g$. 

Such admissible perturbations are given in terms of a family of linear operators $\widetilde{\Xi}(t): \mathcal{D} \longrightarrow\R$ on a linear subspace $\mathcal{D}\subseteq\mathcal{V}$ parameterized by $t > 0$ that are not necessarily bounded on the entire space~$\mathcal{V}$. For $i \in \{1,\dots, d\}$ and a linear mapping $F: \mathcal{D} \longrightarrow \R^d$, let us denote by
\[
    F^i: \mathcal{D} \longrightarrow \R, \qquad y \longmapsto F^iy := (Fy)_i
\]
its $i$-th coordinate functional. Likewise, for a matrix $M \in \R^{d \times d}$, we let $M_i := (M_{ij})_{j \in \{1,\dots, d\}}$ be its $i$-th row. 

\begin{definition}\label{def:admissibleoperator}
    Let $\xi_a \in L(\R^d, \mathcal{H})$ and set $K^a = \Xi S(\cdot) \xi_a \in L_{\mathrm{loc}}^2(\R_+; \R^{d\times d})$ with $a \in \{b,\sigma\}$. A family $(\widetilde{\Xi}(t))_{t > 0}$ of linear operators $\widetilde{\Xi}(t): \mathcal{D} \longrightarrow \R$ on a linear subspace $\mathcal{D} \subseteq \mathcal{V}$ is called admissible, if there exists $(\widetilde{\Xi}_{\lambda})_{\lambda \in (0,1)} \subseteq L(\mathcal{V},\R)$ and some $\widetilde{k}^b,\widetilde{k}^{\sigma} \in L_{\mathrm{loc}}^2(\R_+; \R^{1\times d})$ such that for each $y \in \mathcal{D}$, $T,t>0$, and $a \in \{b,\sigma\}$:
    \begin{equation}\label{eq:admissiblelamdalim}
        \widetilde{\Xi}(t)y = \lim_{\lambda \to 0}\widetilde{\Xi}_{\lambda}S(t)y \ \text{ and } \ 
        \lim_{\lambda \to 0}\int_0^T \big| \widetilde{\Xi}_{\lambda}S(t)\xi_a - \widetilde{k}^a(t)\big|^2\, \mathrm{d}t = 0.
    \end{equation}
    Moreover, for each $\lambda \in (0,1)$ there exist $(N_n^{\lambda})_{n \geq 1}\subseteq \mathbb{N}$, $(a_{ij}^{(n,\lambda)})_{i \in \{1,\dots, d\}, j \in\{ 1,\dots, N_n^{\lambda}\}}\subseteq \R$ and $(z_{ij}^{(n,\lambda)})_{i \in \{1,\dots, d\}, j \in\{1,\dots, N_n^{\lambda}\}}\subseteq \R_+$ such that
    \begin{equation}\label{eq:admissibleNlimy}
        \widetilde{\Xi}_{\lambda}S(t)y = \lim_{n \to \infty}\sum_{i=1}^d\sum_{j=1}^{N_n^{\lambda}} a_{ij}^{(n,\lambda)}\hspace{0.02cm}\Xi^i S\big(z_{ij}^{(n,\lambda)}+t\big)y, \qquad \forall y \in \mathcal{D},
    \end{equation}
    holds for each $t > 0$, and for every $T>0$ we have for $a \in \{b,\sigma\}$:
    \begin{equation}\label{eq:admissibleNlimKerneloperatorcase}
            \lim_{n \to \infty}\int_0^T \bigg| \sum_{i=1}^d\sum_{j=1}^{N_n^{\lambda}} a_{ij}^{(n,\lambda)}K^a_{i}\big(z_{ij}^{(n,\lambda)}+t\big) - \widetilde{\Xi}_{\lambda} S(t)\xi_a\bigg|^2\, \mathrm{d}t = 0.
    \end{equation}
\end{definition} 

Note that in \eqref{eq:admissibleNlimKerneloperatorcase}, the approximation sequence does not depend on $T > 0$. The latter is for convenience only, allowing us to simplify the notation below. The Volterra kernels $\widetilde{k}^b, \widetilde{k}^{\sigma}$ appearing in \eqref{eq:admissiblelamdalim} are uniquely determined by the family $\widetilde{\Xi}(t)$, $t>0$, and the properties~\eqref{eq:admissiblelamdalim} and~\eqref{eq:admissibleNlimKerneloperatorcase}. To indicate this, let us formally write $\widetilde{\Xi}(t) \xi_b := \widetilde{k}^b(t)$ and $\widetilde{\Xi}(t) \xi_{\sigma} := \widetilde{k}^{\sigma}(t)$ as elements in $L_{\mathrm{loc}}^2(\R_+; \R^{1\times d})$, where the notation becomes rigorous once $\xi_b, \xi_{\sigma}\in L(\R^d,\mathcal{D})$. 

The additional approximation via the family of operators $(\widetilde{\Xi}_{\lambda})_{\lambda\in (0,1)}$ is required for singular Volterra kernels, where the family of operators $\widetilde{\Xi}(t)$ is evaluated only in $t > 0$. However, for regular kernels, the family $\widetilde{\Xi}(t)$ may also be well-defined in $t = 0$ as shown in the examples below. 

\begin{example}\label{example:easyadmissible}
 Fix any $i \in \{1, \dots, d\}$ and $a \in \{b,\sigma\}$. The following operators are admissible:
\begin{enumerate}
    \item[(i)] (translations) For $z \geq 0$ the family of bounded linear operators $\widetilde{\Xi}(t) = \Xi^i S(z+t)$, $t\ge 0$, with $\mathcal{D} = \mathcal{V}$, $\widetilde{k}^a(t) = K^a_i(t+z)$ is admissible.

    \item[(ii)] (constants) Suppose that there exists $S_{\infty} \in L(\mathcal{V})$ such that $S(t) \longrightarrow S_{\infty}$ strongly on~$\mathcal{V}$ as $t\to\infty$. Then $S(t)S_{\infty} = S_{\infty}S(t) = S_{\infty} = S_{\infty}^2$, and it is easy to see that $\widetilde{\Xi}(t) = \Xi^i S_{\infty}$ is admissible with $\mathcal{D} = \mathcal{V}$. To compute $\widetilde{k}^a$, note that $S_{\infty}\xi_a := S_{\infty}S(t)\xi_a$ is well-defined since it is independent of $t$. Hence, $\widetilde{k}^a(t) = \Xi^i S_{\infty} S(t)\xi_a = \Xi^i S_{\infty}\xi_a = K^a_i(\infty)$ is the constant Volterra kernel. 

    \item[(iii)] (derivatives) Let $(\mathcal{A}, D(\mathcal{A}))$ be the generator of $(S(t))_{t \geq 0}$, and suppose that $\xi_a \in L(\R^d,D(\mathcal{A}))$. Then $\widetilde{\Xi}(t) = \Xi^i \mathcal{A} S(t)$ is admissible with $\mathcal{D} = D(\mathcal{A})$ and we obtain $\widetilde{k}^a(t) = \Xi^i \mathcal{A} S(t)\xi_a = (K_i^a)'(t)$.
\end{enumerate}
\end{example}

The convergence $S(t) \longrightarrow S_{\infty}$ is satisfied for the commonly used examples of Markovian lifts and naturally appears for the study of invariant measures for stochastic Volterra processes, see \cite{BBCF25}.

Remark that admissible operators are forward-looking in the sense that $\widetilde{\Xi}(t)\xi_b, \widetilde{\Xi}(t)\xi_{\sigma}$ are built from $K^b(r), K^{\sigma}(r)$, $r\ge t$. An important class of examples is based on the subordination of the semigroup. Below, we provide two examples which correspond to abstract forward fractional differentiation and integration, respectively. Our first example provides an abstract infinite-dimensional analogue of the forward (fractional) derivative dating back to Weyl~\cite{MR3618577}~and Marchaud \cite{MR3532941}, see also~\cite{Fe18} for a non-technical exposition on the topic.

\begin{example}\label{admissible operator fractional differentiation}
    Assume that $\|S(z)\|_{L(\mathcal{V})} \lesssim 1 + \sqrt{z}$, $z\in\R_+$, and let $\nu_1,\dots, \nu_d$ be Borel measures on $\R_+$ such that $\nu_i(\{0\}) = 0$ and $\int_{\R_+} (1 \wedge z)\hspace{0.03cm} e^{-\lambda z}\, \nu_i(\mathrm{d}z) < \infty$ for each $\lambda > 0$ and $i \in\{1,\dots, d\}$. Suppose that for each $T > 0$, $i \in \{1,\dots, d\}$, and $a \in \{b,\sigma\}$:
    \begin{equation}\label{eq:Kintegrabilityadmissibledifferentiation}
            \int_0^T \left( \int_{\R_+} |K^a_i(z+t) - K^a_i(t)|\, \nu_i(\mathrm{d}z)\right)^2\, \mathrm{d}t < \infty.
    \end{equation}
    Then $(\widetilde{\Xi}(t))_{t > 0}$ defined by
    \[
        \widetilde{\Xi}(t)y = \sum_{i=1}^dc_i \int_{\R_+} \Xi^i S(t)(S(z)y - y)\, \nu_i(\mathrm{d}z),
    \]
    with $c_1,\dots, c_d \in \R$, is admissible on the domain
    \begin{equation}
            \mathcal{D} = \left\{ y \in \mathcal{V} \ : \ \int_{\R_+} |\Xi^i S(t)(S(z)y - y)|\, \nu_i(\mathrm{d}z) < \infty, \ \forall t > 0,\ \forall i\in\{1,\dots,d\} \right\}
    \end{equation}
    Moreover, $\widetilde{\Xi}(\cdot)\xi_a = \widetilde{k}^a \in L_{\mathrm{loc}}^{2}(\R_+; \R^{1 \times d})$ is given by 
    \begin{align}\label{eq: K admissible differentiation}
        \widetilde{k}^a(t) = \sum_{i=1}^d c_i \int_{\R_+} \left(K^a_i(z+t) - K^a_i(t) \right)\, \nu_i(\mathrm{d}z).
    \end{align}
\end{example}

The proof of the above claims is given in Appendix \ref{subsection:fracdifferentiationadmiss}. For the particular case $\nu_i(\mathrm{d}z) = e^{-\lambda_i z} z^{-1 - \gamma_i}\,\mathrm{d}z$ with $\gamma_i \in (0,1)$ and $\lambda_i \geq 0$, we \textit{formally} obtain $\widetilde{\Xi} = \sum_{i=1}^d c_i \Xi^i (\lambda_i-\mathcal{A})^{\gamma_i}$, which reflects the fractional nature of this operation. Below, we complement the latter with an abstract infinite-dimensional fractional integration.

\begin{example}\label{admissible operator Sinfty}
    Suppose that $S(t) \longrightarrow S_{\infty}$ strongly in $L(\mathcal{V})$ as $t \to \infty$. Let $\nu_1,\dots, \nu_d$ be Borel measures on $\R_+$ such that $\int_{\R_+} e^{-\lambda z}\, \nu_i(\mathrm{d}z) < \infty$ for each $\lambda > 0$ and $i \in\{1,\dots, d\}$. Suppose that for each $T > 0$, $i \in\{1,\dots, d\}$, and $a \in \{b,\sigma\}$:
    \begin{equation}\label{eq:Kintegrabilityadmissible}
            \int_0^T \left( \int_{\R_+} |K^a_i(z+t) - K^a_i(\infty)|\, \nu_i(\mathrm{d}z)\right)^2\, \mathrm{d}t < \infty,
    \end{equation}
    where $K^a(\infty) = \Xi S_{\infty}S(t)\xi_a$. Then $(\widetilde{\Xi}(t))_{t > 0}$ defined by
    \[
        \widetilde{\Xi}(t)y = \sum_{i=1}^d c_i \int_{\R_+} \Xi^i S(z+t)(y - S_{\infty}y)\, \nu_i(\mathrm{d}z),
    \]
    with $c_1,\dots, c_d \in \R$ is admissible on the domain
    \begin{equation}
            \mathcal{D} = \left\{ y \in \mathcal{V} \ : \ \int_{\R_+} |\Xi^i S(z+t)(y - S_{\infty}y)|\, \nu_i(\mathrm{d}z) < \infty, \ \forall t > 0,\ \forall i\in\{1,\dots,d\} \right\}
    \end{equation}
    Moreover, $\widetilde{\Xi}(t)\xi_a = \widetilde{k}^a \in L_{\mathrm{loc}}^{2}(\R_+; \R^{1 \times d})$ is given by 
    \begin{align}\label{eq: K admissible}
        \widetilde{k}^a(t) = \sum_{i=1}^d c_i \int_{\R_+} \left(K^a_i(z+t) - K^a_i(\infty) \right)\, \nu_i(\mathrm{d}z).
    \end{align}
\end{example}

Remark that, since $S(t) \longrightarrow S_{\infty}$ strongly on $\mathcal{V}$, the uniform boundedness principle yields $\sup_{z \geq 0}\|S(z)\|_{L(\mathcal{V})}<\infty$. Hence, the proof of this result is very similar to the argument for Example \ref{admissible operator fractional differentiation} and therefore omitted. The particular choice $\nu_i(\mathrm{d}z) = e^{-\lambda_i z}\frac{z^{\gamma_i-1}}{\Gamma(\gamma_i)}\mathrm{d}z$ with $\gamma_i \in (0,1)$ and $\lambda_i \geq 0$ yields the admissible operator \textit{formally} given by subordination of the semigroup with $\widetilde{\Xi} = \sum_{i=1}^d c_i \Xi^i (\lambda_i - \mathcal{A})^{-\gamma_i}\left(\mathrm{id}_{\mathcal{V}} - S_{\infty}\right)$. For related results on subordination of semigroups, we refer to~\cite{MR4493597}.

\subsection{Failure of the Markov property}

We now have all the tools to prove the failure of the Markov property \eqref{eq: Markov property} for SVEs. Let us denote by $\mathcal{F}_t^{X,B} = \Sigma( X_s \, | \, s \leq t) \vee \Sigma( B_s \, | \, s \leq t)$, $t\in\R_+$, the natural filtration generated by $X$ and $B$. Since $X$ is $(\mathcal{F}_t)_{t\in\R_+}$-adapted and $B$ is an $(\mathcal{F}_t)_{t \in \R_+}$-Brownian motion by the notion of a weak solution, we always have $\mathcal{F}_t^{X,B}\subseteq \mathcal{F}_t$ for all $t\in\R_+$. The following is our main result on the failure of the Markov property.

\begin{theorem}\label{thm: abstract nonmarkov}
    Suppose that condition (B) holds, and consider $b \in B(\R_+; C^0(\R^d; \R^d))$, and $\sigma \in B(\R_+; C^0( \R^d; \R^{d \times m}))$. Let $X$ be a continuous weak solution of \eqref{eq:generalSVIE} defined on some complete filtered probability space $(\Omega, \mathcal{F}, (\mathcal{F}_t)_{t \in \R_+}, \P)$ with $K^{a}=\Xi S(\cdot)\xi_a$, $a\in\{b,\sigma\}$, and $\Xi S(\cdot)\xi=g \in \mathcal{G}_p$ for some $p > 2$ with $\frac{1}{p} + \rho < \frac{1}{2}$. Assume that there exists an admissible operator family $(\widetilde{\Xi}(t))_{t > 0}$ on a linear subspace $\mathcal{D}\subseteq\mathcal{V}$ such that $\P[ \xi \in \mathcal{D} ] = 1$. Denote by $\widetilde{k}^a(t) = \widetilde{\Xi}(t)\xi_a$ with $t > 0$ and $a \in \{b,\sigma\}$ the corresponding admissible Volterra kernels and~define
    \begin{align}\label{eq: Xi X perturbation}
        \widetilde{\Xi}\mathcal{X}_t := \widetilde{\Xi}(t)\xi + \int_0^t \widetilde{k}^b(t-s)\hspace{0.02cm}b(s,X_s)\, \mathrm{d}s + \int_0^t \widetilde{k}^{\sigma}(t-s)\hspace{0.02cm}\sigma(s,X_s)\, \mathrm{d}B_s, \quad t\in\R_+^*.
    \end{align}
    Assume that there exists $t_0 > 0$ such that $\P[X_{t_0} \in \Gamma_{\sigma,t_0}] > 0$ and $(X_{t_0}, \widetilde{\Xi}\mathcal{X}_{t_0})$ is absolutely continuous to the Lebesgue measure on $\Gamma_{\sigma, t_0} \times \R$, where $\Gamma_{\sigma,t_0}$ is defined in~\eqref{eq: Gamma sigma}. Then $X$ is not a Markov process with respect to $(\mathcal{F}^{X,B}_t)_{t \in \R_+}$. 
\end{theorem}
\begin{proof}
    Denote by $\mathcal{X}$ the corresponding $\mathcal{V}$-valued Markovian lift as constructed in Proposition~\ref{prop: Markov lift} which satisfies $\mathcal{X} \in L^p(\Omega, \P; C([0,T];\mathcal{V}))$ and $X = \Xi \mathcal{X}$ on $[0,T]$ for every $T>0$.

    \textit{Step 1.} Suppose that $X$ forms a Markov process in the sense of \eqref{eq: Markov property} with respect to $(\mathcal{F}^{X,B}_t)_{t \in \R_+}$. Then $\E[f(X_T) \, | \, \F^{X,B}_t] = \E[ f(X_T)\, | \, X_t]$ holds for all $T > t$ and any bounded measurable function $f:\R^d\longrightarrow \R$. By approximation and componentwise evaluation, the identity also extends to functions $f:\R^d\longrightarrow\R^d$ that are measurable and satisfy $\E[|f(X_T)|] < \infty$. Using the semigroup property and $\Xi\mathcal{X}=X$, we obtain
    \begin{align}\label{eq:X S}
        \mathcal{X}_T &= S(T-t)\mathcal{X}_t + \int_t^T S(T-r)\hspace{0.02cm}\xi_b\hspace{0.02cm} b(r, X_r)\, \mathrm{d}r + \int_t^T S(T-r)\hspace{0.02cm}\xi_{\sigma} \hspace{0.02cm}\sigma(r, X_r)\, \mathrm{d}B_r.
    \end{align}
    The bounds \eqref{eq:abstractoperatornormestimate} and \eqref{eq:abstractLiftmomentbounds} imply in connection with $b$, $\sigma$ being of linear growth that $\E[ \int_t^T S(T-r)\hspace{0.02cm}\xi_{\sigma}\hspace{0.02cm} \sigma(r,X_r)\, \mathrm{d}B_r \, | \, \mathcal{F}^{X,B}_t] = 0$ since $B$ is an $(\mathcal{F}_t)_{t\in\R_+}$-Brownian motion and $\mathcal{F}_t^{X,B}\subseteq \mathcal{F}_t$. Hence, by taking conditional expectations in \eqref{eq:X S}, and noting that $\mathcal{X}_t$ is $\mathcal{F}_t^{X,B}$-measurable due to \eqref{eq:MarkovianLiftconstruction}, we find 
    \begin{align}\label{eq:liftconditionalexpectation}
        \E\big[\mathcal{X}_T\, | \, \F^{X,B}_t\big] &= S(T-t)\mathcal{X}_t + \int_t^T S(T-r)\hspace{0.02cm}\xi_b \hspace{0.02cm}\E\big[b(r,X_r)\, | \, \F^{X,B}_t\big]\, \mathrm{d}r,
    \end{align}
    where we applied Fubini's theorem for Bochner integrals, which is allowed by the moment bounds shown in Proposition \ref{prop: Markov lift}. Using the Markov property for $f = \id_{\R^d}$ and $f = b(r,\cdot)$, $r>t$, and an application of $\Xi$ to $\E[\mathcal{X}_T\, | \, \F^{X,B}_t]$, gives 
    \begin{align*}
        \E[X_T\, | \, X_t] = \Xi\, \E[ \mathcal{X}_T \, | \, \F^{X,B}_t ]
        = \Xi S(T-t)\mathcal{X}_t + \int_t^T K^b(T-r) \hspace{0.02cm}\E[b(r,X_r)\, | \, X_t]\, \mathrm{d}r.
    \end{align*}
   Hence, $\Xi S(T-t)\mathcal{X}_t$ is measurable w.r.t.~the $\sigma$-algebra $\Sigma(X_t)$ generated by $X_t$. Letting $T = z + t$ for $z\in\R_+$ and taking projections on the $i$-th coordinate, we observe that also $\Xi^i S(z)\mathcal{X}_t$ is $\Sigma(X_t)$-measurable for each $z \in \R_+$ and every $i \in \{1,\dots, d\}$.

    \textit{Step 2\hspace{0.02cm}.} By the definition of admissibility in Definition \ref{def:admissibleoperator}, we find bounded linear operators $(\widetilde{\Xi}_{\lambda})_{\lambda \in (0,1)}\subseteq L(\mathcal{V},\R)$ such that $\widetilde{\Xi}_{\lambda}S(z)\xi \longrightarrow \widetilde{\Xi}(z)\xi$ for each $z > 0$ in probability by $\xi\in\mathcal{D}$ a.s., and $\widetilde{\Xi}_{\lambda} S(\cdot)\xi_a \longrightarrow \widetilde{k}^a$ in $L^2([0,T]; \R^{1\times d})$ as $\lambda \to 0$ for every $T>0$ and $a \in \{b, \sigma\}$. Fix $\lambda \in (0,1)$. Then we find sequences $(N^{\lambda}_n)_{n\ge 1}$, $(a_{ij}^{(n,\lambda)})_{i \in \{1,\dots, d\}, j \in\{ 1,\dots, N^{\lambda}_n\}}$ and $(z_{ij}^{(n, \lambda)})_{i \in \{1,\dots, d\}, j \in\{ 1,\dots, N^{\lambda}_n\}}$ such that $\widetilde{\Xi}^{(n)}_{\lambda}S(z) y \longrightarrow \widetilde{\Xi}_{\lambda}S(z)y$ holds for each $y \in \mathcal{D}$ and $z > 0$ as $n\to \infty$, where $\widetilde{\Xi}_{\lambda}^{(n)}\in L(\mathcal{V},\R)$ is defined by 
   \begin{align}\label{eq: Xin approximation}
    \widetilde{\Xi}_{\lambda}^{(n)} = \sum_{i=1}^d \sum_{j=1}^{N^{\lambda}_n} a_{ij}^{(n, \lambda)}\hspace{0.02cm} \Xi^i S(z_{ij}^{(n, \lambda)}).
   \end{align}   
   Moreover, for every $T>0$ we obtain $\widetilde{\Xi}_{\lambda}^{(n)}S(\cdot)\xi_a \longrightarrow \widetilde{\Xi}_{\lambda}S(\cdot)\xi_a$ in $L^2([0,T]; \R^{1 \times d})$ as $n\to\infty$ for $a \in \{b, \sigma\}$. The particular form of $\widetilde{\Xi}_{\lambda}^{(n)}$ combined with step 1 shows that the random variable $\widetilde{\Xi}^{(n)}_{\lambda} \mathcal{X}_t$ is measurable w.r.t.~the sigma-algebra generated by $X_t$, i.e.\ $\Sigma(X_t)$. Passing for $t>0$ to the limit $n \to \infty$ in 
   \begin{align*}
       \widetilde{\Xi}^{(n)}_{\lambda} \mathcal{X}_t &= \widetilde{\Xi}^{(n)}_{\lambda} S(t)\xi + \int_0^t \widetilde{\Xi}^{(n)}_{\lambda} S(t-s)\hspace{0.02cm}\xi_b\hspace{0.02cm} b(s,X_s)\, \mathrm{d}s + \int_0^t \widetilde{\Xi}^{(n)}_{\lambda} S(t-s)\hspace{0.02cm}\xi_{\sigma}\hspace{0.02cm} \sigma(s,X_s)\, \mathrm{d}B_s,
   \end{align*}
   we conclude that $\widetilde{\Xi}_{\lambda}^{(n)} \mathcal{X}_t \longrightarrow \widetilde{\Xi}_{\lambda}\mathcal{X}_t$ holds in probability with the limit defined by the right-hand side of
    \begin{align}\label{eq:projectionXitildestochFubini}
        \widetilde{\Xi}_{\lambda}\mathcal{X}_t = \widetilde{\Xi}_{\lambda}S(t)\xi + \int_0^t \widetilde{\Xi}_{\lambda}S(t-s)\hspace{0.02cm}\xi_b \hspace{0.02cm}b(s,X_s)\, \mathrm{d}s + \int_0^t \widetilde{\Xi}_{\lambda}S(t-s)\hspace{0.02cm}\xi_{\sigma} \hspace{0.02cm}\sigma(s,X_s)\, \mathrm{d}B_s.
    \end{align}
   Indeed, for the first term we obtain $\widetilde{\Xi}_{\lambda}^{(n)}S(t)\xi \longrightarrow \widetilde{\Xi}_{\lambda}S(t)\xi$ a.s.\ by $\xi\in\mathcal{D}$ a.s.\ combined with \eqref{eq:admissibleNlimy}, while for the integrals we may prove convergence in $L^2(\Omega,\P;\R)$ from which convergence in probability follows from Vitali's convergence theorem. For the stochastic integral we obtain from the BDG inequality, $\sigma$ being of linear growth, locally uniformly in the time argument (see \eqref{eq:uniformlineargrowth}), and the moment estimate \eqref{eq:abstractLiftmomentbounds} from Proposition~\ref{prop: Markov lift}: 
   \begin{equation}\label{eq:diffusionconvergence}
    \begin{aligned}
   \mathbb{E}&\bigg[\bigg|\int_0^t \big( \widetilde{\Xi}^{(n)}_{\lambda} -\widetilde{\Xi}_{\lambda}\big)S(t-s)\hspace{0.02cm}\xi_{\sigma}\hspace{0.02cm}\sigma(s,X_s)\, \mathrm{d}B_s\bigg|^2\bigg]\\
   &\lesssim  \int_0^t \left| \big(\widetilde{\Xi}^{(n)}_{\lambda} -\widetilde{\Xi}_{\lambda}\big)S(t-s)\hspace{0.02cm}\xi_{\sigma}\right|^2\hspace{0.02cm}\mathbb{E}\big[|\sigma(s,X_s)|^2\big]\, \mathrm{d}s\\
   &\lesssim \sup_{s\in [0,t]}\big( 1+\E[|X_s|^2]\big)\int_0^t \left| \widetilde{\Xi}_{\lambda}^{(n)}S(r)\xi_{\sigma} - \widetilde{\Xi}_{\lambda} S(r) \xi_{\sigma}\right|^2 \, \mathrm{d}r,
   \end{aligned} 
   \end{equation}
   which converges to zero by \eqref{eq:admissibleNlimKerneloperatorcase}. Similarly, we obtain for the drift by an additional application of Jensen's inequality: 
   \begin{equation}\label{eq:driftconvergence}
   \begin{aligned}
   \mathbb{E}&\bigg[\bigg|\int_0^t \big( \widetilde{\Xi}^{(n)}_{\lambda} -\widetilde{\Xi}_{\lambda}\big)S(t-s)\hspace{0.02cm}\xi_b\hspace{0.02cm}b(s,X_s)\, \mathrm{d}s\bigg|^2\bigg]
   \\ &\lesssim  t\int_0^t \left| \big(\widetilde{\Xi}^{(n)}_{\lambda} -\widetilde{\Xi}_{\lambda}\big)S(t-s)\hspace{0.02cm}\xi_b\right|^2\hspace{0.02cm}\mathbb{E}\big[|b(s,X_s)|^2\big]\, \mathrm{d}s\\
   &\lesssim \sup_{s\in [0,t]}\big( 1+\E[|X_s|^2]\big)\int_0^t \left| \widetilde{\Xi}_{\lambda}^{(n)}S(r)\xi_b - \widetilde{\Xi}_{\lambda} S(r) \xi_b\right|^2 \, \mathrm{d}r,
   \end{aligned}  
   \end{equation}
   where convergence to zero as $n\to\infty$ follows again from \eqref{eq:admissibleNlimKerneloperatorcase}. Denote by $\Sigma_0(X_t)$ the canonical completion of $\Sigma(X_t)$ with respect to~$\P$. Since $(\Omega, \mathcal{F}, \P)$ is complete, we obtain $\Sigma_0(X_t) \subseteq \mathcal{F}$. As convergence in probability preserves measurability, $\widetilde{\Xi}_{\lambda} \mathcal{X}_t$ is also $\Sigma_0(X_t)$-measurable. 
   
   Finally, we pass to the limit $\lambda \to 0$, which is justified by the definition of admissibility and similar arguments to above. For the initial curve we obtain from $\xi\in\mathcal{D}$ a.s., $t>0$ and \eqref{eq:admissiblelamdalim} that $\lim_{\lambda\to 0}\widetilde{\Xi}_{\lambda}S(t)\xi=\widetilde{\Xi}(t)\xi$ holds a.s.\ and thus also in probability. Moreover, for the stochastic integral, we find the estimate 
   \begin{align*}
   \mathbb{E}&\bigg[\bigg|\int_0^t \big( \widetilde{\Xi}_{\lambda}S(t-s)\xi_{\sigma}-\widetilde{k}^{\sigma}(t-s)\big)\hspace{0.02cm}\sigma(s,X_s)\, \mathrm{d}B_s\bigg|^2\bigg]\\
   &\lesssim \sup_{s\in [0,t]}\big( 1+\E[|X_s|^2]\big)\int_0^t \left| \widetilde{\Xi}_{\lambda}S(r)\xi_{\sigma} - \widetilde{k}^{\sigma}(r)\right|^2 \, \mathrm{d}r,
   \end{align*}  
   which converges to zero as $\lambda\to 0$ by \eqref{eq:admissiblelamdalim}. Analogously, one verifies the convergence of the pathwise Lebesgue integrals for the drift in $L^2(\Omega,\P;\R)$. Therefore, we observe that $\widetilde{\Xi}\mathcal{X}_t := \lim_{\lambda \to 0}\widetilde{\Xi}_{\lambda}\mathcal{X}_t$ exists in probability, and is given by \eqref{eq: Xi X perturbation}. By the same argument as above, $\widetilde{\Xi}\mathcal{X}_t$ is again $\Sigma_0(X_t)$-measurable. Therefore, by the Doob-Dynkin lemma there exists a measurable function $\ell(t, \cdot): \R^d \longrightarrow \R$ such that $\widetilde{\Xi}\mathcal{X}_t = \ell(t, X_t)$ holds a.s.

    \textit{Step 3.} To derive the desired contradiction, observe that by assumption $(X_{t_0}, \widetilde{\Xi}\mathcal{X}_{t_0})$ has a density with respect to the Lebesgue measure on $\Gamma_{\sigma,t_0} \times \R$. Denote the graph of the measurable function $\ell(t_0,\cdot)$ found in step 2 for $t=t_0$ by $\widetilde{G}_{t_0}:= \big\{\big(x, \ell(t_0,x)\big):\,x\in\R^d\big\}$. Then $\mathrm{Leb}_{d+1}(\widetilde{G}_{t_0}) = 0$ and by the assumed regularity result combined with the relation $\widetilde{\Xi}\mathcal{X}_{t_0} = \ell (t_0,X_{t_0})$ a.s., by step~2, we also get $\P\big[\big(X_{t_0},\ell(t_0,X_{t_0})\big) \in \widetilde{G}_{t_0} \cap (\Gamma_{\sigma,t_0} \times \R)\big] = 0$. Since we have $\P[ X_{t_0} \in \Gamma_{\sigma,t_0}]>0$ by assumption, we observe
    \begin{align}\label{eq:nonMarkovgeneralfinalcontradiction}
        0 < \P[ X_{t_0} \in \Gamma_{\sigma,t_0}]
        = \P[ ( X_{t_0}, \ell(t_0, X_{t_0})) \in \widetilde{G}_{t_0} \cap (\Gamma_{\sigma,t_0} \times \R) ]
        = 0,
    \end{align}
    providing a contradiction. Therefore, $X$ is not a Markov process.
\end{proof}

The assumption that $(X_t, \widetilde{\Xi}\mathcal{X}_t)$ has a density with respect to the Lebesgue measure on $\Gamma_{\sigma,t} \times \R$ can be verified either by Theorem~\ref{thm: density}, or for diagonal $\sigma$ by Theorem \ref{thm: density diagonal}. In any case, these theorems essentially restrict the class of admissible perturbations to nondegenerate ones, where $\widetilde{K}^{\sigma} = (K^{\sigma}, \widetilde{k}^{\sigma})^{\intercal}$ is $\gamma_*$-nondegenerate. For instance, this excludes kernels of the form $c e^{\lambda t}$ with $c,\lambda \in \mathbb{R}$, cf.\ Example \ref{example: exponentials} below, for which the Markov property is indeed satisfied under suitable uniqueness conditions.

\begin{remark}\label{remark:filtration}
Recall that, by the tower property, the Markov property is preserved when passing to a coarser filtration with respect to which $X$ is still adapted. Hence, under the above conditions, $X$ is also not a Markov process with respect to any finer filtration $\mathcal{G}_t\supseteq\mathcal{F}^{X,B}_t$, $t\in\R_+$. In particular, it is not a Markov process with respect to the filtration $(\mathcal{F}_t)_{t\in\R_+}$ that stems from the definition of a weak solution.
\end{remark}

\begin{remark}
The proof concept allows us not only to rule out the Markov property for~$X$, but also for any process of the form $(X, P\mathcal{X})$ where $P \in L(\mathcal{V}, \R^N)$ for some $N \in \mathbb{N}$. Indeed, suppose that $(X, P\mathcal{X})$ is a Markov process on $\R^d \times \R^N$, then 
\[
    \E\big[f(X_T, P\mathcal{X}_T) \, | \, \mathcal{F}_t^{X,B}\big] = \E[f(X_T,P\mathcal{X}_T) \, | \, (X_t, P\mathcal{X}_t) ].
\]
Arguing as in step 1, we arrive at
\[
    \E[(X_T, P\mathcal{X}_T) \, | \, (X_t,P\mathcal{X}_t) ] = (\Xi, P)S(T-t)\mathcal{X}_t + \int_t^T (\Xi, P)S(T-r)\E[b(r,X_r) \, | \, (X_t, P\mathcal{X}_t) ]\, \mathrm{d}r.
\]
Hence $(\Xi, P)S(T-t)\mathcal{X}_t$ is $(X_t, P\mathcal{X}_t)$-measurable. Arguing as in step 2 based on an analogous notion of admissibility, we can conclude that also the Volterra-Ito-type random variable $\widetilde{\Xi} \mathcal{X}_t$ obtained as a probability limit is $(X_t, P\mathcal{X}_t)$-measurable. Hence, we may obtain a contradiction by showing that for some $t=t_0>0$ with $\P[ X_{t_0} \in \Gamma_{\sigma,t_0}]>0$, $(X_{t_0}, P\mathcal{X}_{t_0}, \widetilde{\Xi}\mathcal{X}_{t_0})$ has a density on $\Gamma_{\sigma,t_0} \times \R^N \times \R$. The precise conditions and details are left to the interested reader.
\end{remark}

\subsection{Weak admissibility}\label{subsection:specialcasenonMarkovperturb}

In this subsection, we consider a weaker form of admissibility that is based on the important observation that \eqref{eq: Xin approximation} needs to converge simultaneously for $S(\cdot)\xi_b$ and $S(\cdot)\xi_{\sigma}$ towards $\widetilde{k}^b$ and $\widetilde{k}^{\sigma}$, respectively. Clearly, this is automatically satisfied when $K:= K^b = K^{\sigma}$. Here, the admissibility condition can be weakened to solely requiring an admissible perturbation $\widetilde{k}$ of suitable given functions $f_1,\dots, f_M$ in the following sense:

\begin{definition}\label{def:admissiblekernel}
    Let $M \in \mathbb{N}$ and $f_1,\dots, f_M \in L_{\mathrm{loc}}^2(\R_+; \R^{1\times d})$. For $T > 0$, the collection of admissible perturbations of $\mathbb{F} = \{f_1,\dots, f_M\}$ on $[0,T]$ is defined by  
    \[
        \mathcal{M}_{\mathrm{ad}}( \mathbb{F},T) = \overline{\left\{ \sum_{i=1}^M \sum_{j=1}^N a_{ij} f_i(\cdot + z_{ij}) \ | \ N \geq 1, \ z_{ij} \geq 0, \ a_{ij} \in \R \right\} }^{L^2([0,T]; \R^{1\times d})}
    \] 
    Let $\mathcal{M}_{\mathrm{ad}}(\mathbb{F}) = \cap_{T > 0}\mathcal{M}_{\mathrm{ad}}(\mathbb{F}, T)$ be the collection of functions that are admissible on $[0,T]$ for any $T > 0$. To simplify the notation, we also write $\mathcal{M}_{\mathrm{ad}}(f,T) = \mathcal{M}_{\mathrm{ad}}(\{f\},T)$ and $\mathcal{M}_{\mathrm{ad}}(f) = \mathcal{M}_{\mathrm{ad}}(\{f\})$.
\end{definition} 

By definition, admissible perturbations $\widetilde{k}\in \mathcal{M}_{\mathrm{ad}}(\mathbb{F},T)$ are row-vector-valued functions that can be obtained as $L^2([0,T])$-limits from linear combinations of translations of $f_1,\dots, f_M$. In our framework, admissible perturbations will be constructed using the collection of rows $\mathbb{F} = \{K_1,\dots, K_d\}$. However, in most cases it is already sufficient to consider admissible perturbations based on $\mathbb{F} = \{K_{i_0}\}$ for \textit{one} choice of $i_0 \in \{1,\dots, d\}$. As a consistency check, let us consider the case where $K$ is a linear combination of exponential functions as commonly used for finite-dimensional Markovian approximations, see \cite{markovian_structure, MR4521278}.
\begin{example}\label{example: exponentials}
    Suppose there exist $N \geq 1$, $a_1,\dots, a_N \in \R^*$ and $\lambda_1, \dots, \lambda_N \geq 0$ such that
    \begin{align*}
        f_{\mathrm{exp}}(t) = \sum_{j=1}^N a_j\hspace{0.03cm} e^{- \lambda_j t}, \qquad t\in\R_+.
    \end{align*}
    Then $\mathcal{M}_{\mathrm{ad}}(f_{\mathrm{exp}}) = \{ \sum_{j=1}^{N} c_j e^{-\lambda_j \cdot} \, : \, c_1,\dots, c_N \in \R\}$. Hence, there is no freedom to construct admissible perturbations outside this specific form. This highlights the special role of exponentials for the Markov property. 
\end{example}

\begin{example}
    Let $f(t) = \sin(t)$. Then it follows from the addition theorems for trigonometric functions that $\mathcal{M}_{\mathrm{ad}}(f) = \{ c_1 \sin + c_2 \cos \, : \, c_1, c_2 \in \R \}$.
\end{example}

This notion of admissibility is conceptually similar to the density in $L^2(\R)$ of the linear span of translations $f(\cdot+z)$ with $z \in \R$. Such a problem dates back to Wiener's Tauberian theorems and Fourier methods, see \cite{MR1503035}. However, the techniques therein do not directly apply to our framework as we only consider nonnegative translations $z \geq 0$. Still, by the use of Laplace transforms and hence implicitly M\"untz's theorem for the density of monomials in $L^p$-spaces, a variant of such a result is given below for the multi-dimensional case. For this purpose, we require a mild support condition:

\begin{definition}
    Let $\mu$ be a Borel measure on $\R_+$. We say that its support is nondegenerate if it contains a sequence of pairwise distinct elements $(\lambda_n)_{n \in \mathbb{N}} \subseteq \mathrm{supp}(\mu)\cap\R_+^*$ such that 
    \[
        \sum_{n=1}^{\infty} \frac{1}{\lambda_n} = \infty.
    \] 
\end{definition}

The next lemma shows that admissibility is satisfied for a remarkably large class of functions, whenever $f_1,\dots, f_M$ have at least one component $f_{ij}$ whose derivative is eventually completely monotone with its Bernstein measure having nondegenerate support. In particular, for every such component, each element in $L^2([0,T];\R)$ is admissible.

\begin{lemma}\label{lemma:compmonadmissiblekernel}
    Let $M \in \mathbb{N}$ and $f_1,\dots, f_M \in L_{\mathrm{loc}}^2(\R_+; \R^{1 \times d})$. Denote by $J \subseteq \{1,\dots, d\}$ the collection of indices $j \in J$ such that there exists $i \in \{1,\dots, M\}$ and $m_j \in \mathbb{N}_0$ with $f_{ij} \in C^{\infty}(\R_+^*; \R)$ and its $m_j$-th derivative $f_{ij}^{(m_j)} \in L_{\mathrm{loc}}^2(\R_+; \R)$ is completely monotone with its Bernstein measure $\mu_{ij}$ on $\R_+$ having nondegenerate support. Then for each $T > 0$:
    \[
        \left\{ (\widetilde{k}_1, \dots, \widetilde{k}_d) \in L^2([0,T]; \R^{1\times d}) \ \big| \ \widetilde{k}_j = 0, \ j \not \in J \right\} \subseteq \mathcal{M}_{\mathrm{ad}}(\mathbb{F}, T).
    \]
\end{lemma}
\begin{proof}
    Set $\mathbb{F} = \{f_1,\dots, f_M\}$. Since $\mathcal{M}_{\mathrm{ad}}(\mathbb{F}, T)$ is linear, it suffices to show that $g e_j^{\intercal} \in \mathcal{M}_{\mathrm{ad}}(\mathbb{F}, T)$ holds for each $j \in J$ and $g \in L^2([0,T]; \R)$, where $e_j \in \R^d$ denotes the $j$-th canonical basis vector. The latter is equivalent to $\mathcal{M}_{\mathrm{ad}}(\mathbb{F}, T)^{\perp} \subseteq \{ g e_j^{\intercal} \}^{\perp}$. Since $\mathcal{M}_{\mathrm{ad}}(\mathbb{F},T)^{\perp} = \overline{A}^{\perp} = A^{\perp}$, where 
    \[
        A := \left\{ \sum_{i=1}^M \sum_{j=1}^N a_{ij} f_i(\cdot + z_{ij}) \ | \ N \geq 1, \ z_{ij} \geq 0, \ a_{ij} \in \R \right\},
    \]
    it suffices to show that for any $g \in L^2([0,T]; \R)$ with $g e_j^{\intercal} \in A^{\perp}$, one necessarily has $g = 0$. Indeed, then for any choice of $g \in L^2([0,T]; \R)$, the orthogonal projection of $g e_j^{\intercal}$ onto $\mathcal{M}_{\mathrm{ad}}(\mathbb{F},T)^{\perp}$ vanishes and hence $g e_j^{\intercal}\in \mathcal{M}_{\mathrm{ad}}(\mathbb{F},T)$.

    Let us fix any $j \in J$ and $g e_j^{\intercal} \in A^{\perp}$. Select an index $i \in\{1,\dots, M\}$ associated with $j \in J$ according to the definition of $J$. Then it follows that
    \begin{align}\label{eq: 13}
        0 = \int_0^T g(t)e_j^{\intercal} f_i^{\intercal}(t+z)\, \mathrm{d}t = \int_0^T g(t) f_{ij}(t+z)\, \mathrm{d}t, \qquad \forall z \geq 0.
    \end{align}
    Differentiating \eqref{eq: 13} $m_j$-times in $z > 0$, which is justified by dominated convergence, and using the representation $f_{ij}^{(m_j)}(t+z) = \int_{\R_+}e^{-(t+z)x}\, \mu_{ij}(\mathrm{d}x)$ by complete monotonicity, we obtain from Fubini's theorem
    \begin{equation}\label{eq:Laplaceintermedstep}
        \int_{\R_+} e^{-zx} \left( \int_0^T e^{- x t}g(t) \, \mathrm{d}t \right)\, \mu_{ij}(\mathrm{d}x) = 0.
    \end{equation}  
    Note that all integrals are well-defined since by our assumptions:
    \begin{align*}
        \int_{\R_+}\int_0^T e^{- xt} |g(t)|\, \mathrm{d}t \,\mu_{ij}(\mathrm{d}x) 
        &\leq \int_0^T f_{ij}^{(m_j)}(t)\hspace{0.02cm}|g(t)|\, \mathrm{d}t 
        \\ &\leq \big\| f_{ij}^{(m_j)}\big\|_{L^2([0,T])} \hspace{0.02cm}\| g\|_{L^2([0,T])} < \infty.
    \end{align*}
    Hence, as $z\in\R_+$ was arbitrary, uniqueness of Laplace transforms gives $\int_0^T e^{-xt}g(t)\, \mathrm{d}t = 0$ for $\mu_{ij}$-a.a.\ $x\in\R_+$. 
    
    Define $N = \left\{ x \in \R_+ \ : \ \int_0^T e^{-xt}g(t) \, \mathrm{d}t = 0 \right\}$. Then $\mu_{ij}(N^c) = 0$, and since $N$ is closed, it follows from the definition of the support that $\mathrm{supp}(\mu_{ij}) \subseteq N$. Hence, we obtain
    \begin{equation}\label{eq:uniquenesslaplaceadmissibility}
        \int_0^T e^{-xt}g(t)\, \mathrm{d}t = 0, \qquad \forall x \in \mathrm{supp}(\mu_{ij}).
    \end{equation}
    By the nondegeneracy of the support, let $(\lambda_n)_{n \geq 1} \subseteq \mathrm{supp}(\mu_{ij})$ be a sequence of pairwise distinct elements such that $\sum_{n=1}^{\infty}\lambda_n^{-1} = \infty$. Then there exists a subsequence, again denoted by $(\lambda_n)_{n \geq 1}$, such that either $\lambda_n \to \lambda \in \R_+$, or $(\lambda_n)_{n \geq 1}$ is increasing to $\infty$. In the first case, define the entire function $F(x) = \int_0^T e^{-xt}g(t)\,\mathrm{d}t$ and note that $F(\lambda_n) = 0$ for all $n \geq 1$. By the identity theorem for analytic functions, we obtain $F \equiv 0$, and hence uniqueness of Laplace transforms proves $g = 0$. For the second case, we obtain
    \begin{displaymath}
        \frac{1}{\lambda_n}\le \frac{1}{\lambda_n}\frac{2\lambda_n}{\lambda_n +1}= \frac{2+\lambda_n -\lambda_n}{\lambda_n +1}=1-\left|\frac{\lambda_n -1}{\lambda_n +1} \right|,
    \end{displaymath}
    for every $n\ge\overline{n}$ with $\lambda_{\overline{n}}\ge 1$. Thus, $(\lambda_n)_{n \geq 1}$ is a uniqueness sequence for Laplace transforms by \cite[Theorem 1.11.1]{ArBaHiNe11} and evaluating \eqref{eq:uniquenesslaplaceadmissibility} at $x = \lambda_n$ yields again $g = 0$. As $j \in J$ was in both cases arbitrary, the assertion is proved.
\end{proof}

When $m = 0$, this lemma covers completely monotone kernels, while $m = 1$ corresponds to Bernstein functions. Remark that any Borel measure $\mu$ whose support is not nondegenerate necessarily satisfies either $\# \supp(\mu)<\infty$ or $\mathrm{supp}(\mu) = \{ \lambda_n \ : \ n \geq 1\}$ for some sequence $(\lambda_n)_{n \geq 1}$ strictly increasing to infinity with $\sum_{n=1}^{\infty}\lambda_n^{-1} < \infty$. In particular, $f(t) = \sum_{n=1}^{\infty} c_n\hspace{0.02cm}e^{-\lambda_n t}$ and it is clear that for such kernels $\mathcal{M}_{\mathrm{ad}}(f,T)$ is a proper subspace of $L^2([0,T]; \R)$. If $\# \supp(\mu) < \infty$, this follows from Example \ref{example: exponentials}, whereas in the second case it can be shown via M\"untz's theorem.  

\begin{remark}\label{remark:kerneladmissibleexpterm}
The above arguments can be modified straightforwardly to cover also the case $f_{ij}(t) = e^{-\alpha_{ij} t}\hspace{0.02cm} \widetilde{f}_{ij}(t)$ with $\alpha_{ij} \in \R$, provided that $\widetilde{f}=(\widetilde{f}_{ij})_{i \in \{1,\dots, M\}, j\in\{1,\dots,d\}}$ satisfies the conditions of the previous lemma. Indeed, \eqref{eq:Laplaceintermedstep} becomes 
\begin{equation*}
    \int_{\R_+} e^{-zx} \left( \int_0^T e^{- (\alpha_{ij}+x) t}g(t) \, \mathrm{d}t \right)\, \mu_{ij}(\mathrm{d}x) = 0.
\end{equation*}
Then the same arguments as above applied to $e^{-\alpha_{ij}\cdot}g\in L^2([0,T];\R)$ yield $e^{-\alpha_{ij}\cdot}g=0$ and thus again $g=0$. 
\end{remark}

Below we prove the failure of the Markov property under the weaker notion of admissibility introduced in Definition \ref{def:admissiblekernel}.

\begin{theorem}\label{thm: special case nonmarkov}
    Suppose that condition (B) is satisfied with $\xi_K := \xi_b = \xi_{\sigma}$, and consider $b \in B(\R_+; C^{0}(\R^d; \R^d))$, and $\sigma \in B(\R_+; C^{0}(\R^d; \R^{d\times m}))$. Let $X$ be a continuous weak solution of \eqref{eq:generalSVIE} defined on some complete filtered probability space $(\Omega, \mathcal{F}, (\mathcal{F}_t)_{t \in \R_+}, \P)$ with $K^b = K^{\sigma}=K := \Xi S(\cdot)\xi_K$ and $g=\Xi S(\cdot)\xi$ for some deterministic $\xi \in\mathcal{V}$. Suppose that there exists an admissible kernel $\widetilde{k} \in \mathcal{M}_{\mathrm{ad}}(\{K_1,\dots, K_d\})$ according to Definition \ref{def:admissiblekernel} and define 
    \begin{equation}\label{eq:Zperturbationdefinition}
        Z_t = \int_0^t \widetilde{k}(t-s)\hspace{0.02cm}b(s,X_s)\, \mathrm{d}s + \int_0^t \widetilde{k}(t-s)\hspace{0.02cm}\sigma(s,X_s)\, \mathrm{d}B_s,\quad t\in\R_+.
    \end{equation}
    If there exists $t_0 > 0$ such that $\P[X_{t_0} \in \Gamma_{\sigma,t_0}] > 0$ and $(X_{t_0}, Z_{t_0})$ is absolutely continuous to the Lebesgue measure on $\Gamma_{\sigma, t_0} \times \R$, then $X$ is not a Markov process with respect to $(\mathcal{F}_t^{X,B})_{t \in \R_+}$.
\end{theorem}
\begin{proof}
    First, observe that when $g=\Xi S(\cdot)\xi$ is deterministic, $g \in \mathcal{G}_p$ holds for any $p > 2$ and hence $\frac{1}{p} + \rho < \frac{1}{2}$ is satisfied. In particular, we may apply Proposition \ref{prop: Markov lift} and denote by $\mathcal{X}$ the Markovian lift constructed there, see \eqref{eq:MarkovianLiftconstruction}. The proof can be carried out similarly to the general case of Theorem \ref{thm: abstract nonmarkov}. Indeed, arguing exactly as in step 1 therein, we observe that $\Xi^i S(z)\mathcal{X}_t$ is $\Sigma(X_t)$-measurable for each $z \in \R_+$ and every $i \in \{1,\dots, d\}$. Performing analogous arguments to step 2, but with merely one approximation step, we construct a perturbation that plays the same role as $\widetilde{\Xi}\mathcal{X}_t$ in the proof of Theorem \ref{thm: abstract nonmarkov}.
    
    By assumption, $\widetilde{k}$ corresponds to an admissible kernel according to Definition~\ref{def:admissiblekernel}. For each fixed $t = T > 0$, we find sequences $(N_n)_{n\ge 1}$, $(a_{ij}^{(n)})_{i \in \{1,\dots, d\}, j \in\{ 1,\dots, N_n\}}$ and $(z_{ij}^{(n)})_{i \in \{1,\dots, d\}, j \in\{ 1,\dots, N_n\}}$ such that
    \begin{equation}\label{eq:admissibleNlimKernel}
        \lim_{n \to 0}\int_0^T \bigg| \sum_{i=1}^d \sum_{j=1}^{N_n} a_{ij}^{(n)}K_{i}\big(z_{ij}^{(n)}+t\big) - \widetilde{k}(t)\bigg|^2\, \mathrm{d}t = 0.
    \end{equation}
   Similarly to \eqref{eq: Xin approximation}, we define for every $n\ge 1$ the bounded linear operator
   \begin{align}\label{eq:linearcomboperatorgconstant}
        \widetilde{\Xi}^{(n)} = \sum_{i=1}^d \sum_{j=1}^{N_n} a_{ij}^{(n)}\hspace{0.02cm} \Xi^i S\big(z_{ij}^{(n)}\big).
   \end{align} 
   Noting that we have $\xi_K = \xi_b = \xi_{\sigma}$ and passing to the limit $n \to \infty$ in 
   \begin{align*}
       Z_t^{(n)} :=&\ \widetilde{\Xi}^{(n)} \mathcal{X}_t - \widetilde{\Xi}^{(n)} S(t)\xi \\
       =&\  \int_0^t \widetilde{\Xi}^{(n)} S(t-s)\hspace{0.02cm}\xi_K \hspace{0.02cm}b(s,X_s)\, \mathrm{d}s + \int_0^t \widetilde{\Xi}^{(n)} S(t-s)\hspace{0.02cm}\xi_{K}\hspace{0.02cm} \sigma(s,X_s)\, \mathrm{d}B_s
   \end{align*}
   shows that $Z_t^{(n)} \longrightarrow Z_t$ holds in probability for each $t>0$ with the limit given by
   \begin{align}\label{eq: Z}
    Z_t =  \int_0^t \widetilde{k}(t-s)\hspace{0.02cm} b(s,X_s)\, \mathrm{d}s + \int_0^t \widetilde{k}(t-s)\hspace{0.02cm} \sigma(s,X_s)\, \mathrm{d}B_s.
   \end{align}
   Indeed, combining $\widetilde{\Xi}^{(n)} S(\cdot)\xi_K\longrightarrow \widetilde{k}$ in $L^2([0,t]; \R^{1 \times d})$ according to \eqref{eq:admissibleNlimKernel} with the same estimates as in \eqref{eq:diffusionconvergence} and \eqref{eq:driftconvergence} proves the convergence of both integrals even in $L^2(\Omega,\P;\R)$. 
   
   Then, since $\Xi^i S(z)\mathcal{X}_t$ is $\Sigma(X_t)$-measurable for each $z \in \R_+$ and every $i \in \{1,\dots, d\}$, combined with the particular form of \eqref{eq:linearcomboperatorgconstant}, it follows that also $\widetilde{\Xi}^{(n)} \mathcal{X}_t$ is for every $n\ge 1$ measurable w.r.t.\ $\Sigma(X_t)$, and thus also w.r.t.\ its canonical completion $\Sigma_0(X_t)$. Consequently, since $\xi$ is deterministic, also $Z_t^{(n)}$ is $\Sigma_0(X_t)$-measurable. As the above convergence holds in probability, $Z_t$ inherits the $\Sigma_0(X_t)$-measurability. Thus, we find again for every $t>0$ a measurable function $\ell(t, \cdot): \R^d \longrightarrow \R$ such that $Z_t = \ell(t,X_t)$ a.s. 
    
    Finally, it is straightforward to adjust step 3 to the present case for the $t_0>0$ with $\P[X_{t_0} \in \Gamma_{\sigma,t_0}] > 0$ from our assumptions. Indeed, since $Z_{t_0} = \ell(t_0,X_{t_0})$ a.s.\ and $(X_{t_0}, Z_{t_0})$ is absolutely continuous with respect to the Lebesgue measure on $\Gamma_{\sigma, t_0} \times \R$, the same argument as in \eqref{eq:nonMarkovgeneralfinalcontradiction} proves that $X$ is not a Markov process.
\end{proof}

\section{Application to regularly varying kernels}\label{section: examples}

In this section, we introduce an explicit Markovian lift and apply our abstract results to a concrete class of stochastic Volterra equations with regularly varying Volterra kernels. For clarity of exposition and to allow for a simple verification of the nondegeneracy conditions introduced in Section~\ref{section:abs cont}, we restrict attention to diagonal kernels.  

\subsection{Nondegeneracy for regularly varying kernels}\label{subsection:nondegregvarkernel}

As a preliminary step, we address the nondegeneracy as introduced in Definitions \ref{def: nondegeneracy general} and~\ref{def: nondegeneracy anisotrop}. In both cases, the property was formulated as an asymptotic lower bound for the smallest eigenvalue of the associated Gram matrix functions, see \eqref{eq: Gram} and \eqref{eq: Gram2}. Recall that a measurable function $\ell: \R_+^* \longrightarrow \R_+^*$ is called \textit{slowly varying}, if
\[
    \lim_{t \to \infty}\frac{\ell(\lambda t)}{\ell(t)} = 1,\qquad \forall \lambda > 0.
\]
For properties of slowly varying functions, we refer to \cite{BiGoTe87}. In particular, by \cite[Proposition 1.3.6]{BiGoTe87}, we find for each $\varepsilon>0$ and $c>0$ some $h_0\in (0,1)$ such that 
\begin{align}\label{eq: slowly varying bound}
        c^{-1} \hspace{0.02cm}t^{\varepsilon} \leq \ell(1/t) \leq c \hspace{0.02cm}t^{-\varepsilon}, \qquad t \in (0,h_0).
\end{align}
The next lemma establishes the desired asymptotics.

\begin{lemma}\label{lemma: fractional like}
    Let $k_1,\dots, k_N: \R_+^* \longrightarrow \R$ be locally square-integrable, and regularly varying in $t= 0$, i.e.\ there exist $H_1,\dots, H_N > 0$ and slowly varying functions $\ell_1,\dots, \ell_N: \R_+^* \longrightarrow \R_+^*$ such that 
    \begin{align}\label{eq: remainder fractional}
        k_j(t) \sim c_j\hspace{0.02cm} t^{H_j - 1/2}\hspace{0.02cm}\ell_j(1/t), \qquad t \searrow 0,
    \end{align} 
    holds for all $j \in\{1,\dots, N\}$ with $c_j \neq 0$. Define the Gram matrix function 
    \[
        G_{ij}(h) = \int_0^h k_{i}(r)k_{j}(r)\, \mathrm{d}r, \qquad h>0,\ \, i,j \in\{1,\dots, N\}.
    \]
    If $H_1,\dots, H_N$ are pairwise distinct, then
    \[
        \liminf_{h \searrow 0} \frac{\lambda_{\min}(G(h))}{h^{2H_{\max}} \ell_{\min}(1/h)^{2}} > 0,
    \]
    where $H_{\max} = \max_{j\in\{1,\dots, N\}}H_j$ and $\ell_{\min} = \min_{j\in\{1,\dots, N\}}\ell_j$. 
\end{lemma}
\begin{proof}
    
Define the matrix functions $D(h) = \mathrm{diag}(c_1 h^{H_1}\ell_1(1/h), \dots, c_N h^{H_N}\ell_N(1/h))$ and $B(h) = D(h)^{-1}G(h)D(h)^{-1}$. Then, we obtain for $\xi \in \R^{N}$ with $|\xi|=1$ and $h\in (0,1)$:
    \begin{equation}\label{eq:eigenvalueboundregvarproof}
    \begin{aligned}
        \xi^{\intercal}G(h)\xi &= (D(h)\xi)^{\intercal} B(h) (D(h)\xi)
        \\ &\geq \lambda_{\min}(B(h))| D(h)\xi|^2
        \geq \left( \min_{j\in\{1,\dots, N\}}|c_j|^2\right) h^{2H_{\max}}\ell_{\min}(1/h)^2 \lambda_{\min}(B(h)).
    \end{aligned}
    \end{equation}
    Hence, it suffices to bound the smallest eigenvalue for $B(h)$. Take any $\varepsilon \in (0, H_{\min})$ and $h_0 \in (0,1)$ as in \eqref{eq: slowly varying bound}. Then $t^{H_j - 1/2}\ell_j(1/t)\in L^2([0,h_0])$ for each $j \in\{1,\dots, N\}$, which guarantees that all integrals below are well-defined. Define $r_j(t) = k_j(t) - c_j t^{H_j-1/2}\ell_j(1/t)$, then we obtain for $h\in (0,h_0)$ the decomposition $\frac{G_{ij}(h)}{c_ic_j} = G_{ij}^0(h) + G_{ij}^1(h)$ with
    \begin{align*}
        G^0_{ij}(h) &= \int_0^h t^{H_j + H_i - 1}\ell_j(1/t)\ell_i(1/t)\, \mathrm{d}t,
        \\ G_{ij}^1(h) &= \int_0^h \left( t^{H_j - 1/2}\ell_j(1/t) \frac{r_i(t)}{c_i} + t^{H_i - 1/2}\ell_i(1/t) \frac{r_j(t)}{c_j} + \frac{r_i(t) r_j(t)}{c_ic_j} \right)\, \mathrm{d}t.
    \end{align*}
    To bound the second term, let us define $\delta_j(t) = \frac{r_j(t)}{c_j t^{H_j-1/2}\ell_j(1/t)}$ for $j \in\{1,\dots, N\}$ and $t > 0$. It then follows from \eqref{eq: remainder fractional} that $\delta_j(t) \to 0$ as $t \searrow 0$. Hence, for $\varepsilon \in (0,1/3)$ there exists $h_1(\varepsilon) \in (0,h_0)$ such that $|\delta_j(t)| \leq \varepsilon$ for $t \in (0,h_1(\varepsilon))$ and each $j\in\{1,\dots,N\}$. Therefore, we obtain for $h \in (0,h_1(\varepsilon))$:
    \begin{align*}
        |G_{ij}^1(h)| &\leq \int_0^h \left( t^{H_j - 1/2}\ell_j(1/t) \frac{|r_i(t)|}{|c_i|} + t^{H_i - 1/2}\ell_i(1/t)\frac{|r_j(t)|}{|c_j|} + \frac{|r_i(t) r_j(t)|}{|c_ic_j|} \right)\, \mathrm{d}t
        \\ &= \int_0^h t^{H_j + H_i - 1}\ell_j(1/t) \ell_i(1/t) \left( |\delta_i(t)| + |\delta_j(t)| + |\delta_i(t)\delta_j(t)|\right) \, \mathrm{d}t
        \\ &\leq 3 \varepsilon \hspace{0.02cm}G_{ij}^0(h).
    \end{align*}
    Thus, we have shown that for $h \in (0, h_1(\varepsilon))$:
    \begin{align}\label{eq: GG0 bound}
        \left( 1 - 3 \varepsilon\right)G_{ij}^0(h) \leq \frac{G_{ij}(h)}{c_ic_j} \leq \left( 1 + 3 \varepsilon\right)G_{ij}^0(h). 
    \end{align}
    For the first term, we use the substitution $r = 1/t$ and then Karamata's theorem on regularly varying functions (see \cite[Theorem 1.5.11]{BiGoTe87}) to find for $h \searrow 0$:
    \begin{align*}
        G_{ij}^0(h) &= \int_{1/h}^{\infty}r^{-(1 + H_i + H_j)} \ell_i(r)\ell_j(r)\, \mathrm{d}r \sim \frac{h^{H_i + H_j}}{H_i + H_j}\ell_i(1/h)\ell_j(1/h) =: M_{ij}(h).
    \end{align*}
    In particular, we find $h_2(\varepsilon) \in (0,h_1(\varepsilon))$ such that $(1-\varepsilon)M_{ij}(h) \leq G_{ij}^0(h) \leq (1+\varepsilon)M_{ij}(h)$ for $h \in (0,h_2(\varepsilon))$. Combining this with \eqref{eq: GG0 bound}, we arrive at
    \[
        (1-3\varepsilon)(1-\varepsilon)M_{ij}(h) \leq \frac{G_{ij}(h)}{c_ic_j} \leq (1+3\varepsilon)(1+\varepsilon) M_{ij}(h).
    \]
    Next, let us define the Cauchy matrix $J = (J_{ij})_{i,j \in\{1,\dots, N\}}$ with entries $J_{ij} = (H_i + H_j)^{-1}$. Then, using the explicit form $B_{ij}(h) = G_{ij}(h)\hspace{0.02cm} (c_i c_j h^{H_i + H_j}\ell_i(1/h)\ell_j(1/h))^{-1}$, we obtain for $h \in (0, h_2(\varepsilon))$:
    \[
        (1-3\varepsilon)(1-\varepsilon) J_{ij} \leq B_{ij}(h) \leq (1+3\varepsilon)(1+\varepsilon) J_{ij},
    \]
    which yields the following bound on the operator norm:
    \[
     |B(h) - J| \leq \sum_{i,j=1}^N | B_{ij}(h) - J_{ij}|
     \leq (4\varepsilon + 3\varepsilon^2)\sum_{i,j=1}^N J_{ij}.
    \]
    Finally, as $B(h)$, $J$ are symmetric, an application of Weyl's inequality gives for sufficiently small $\varepsilon>0$:
    \begin{align*}
        \lambda_{\min}(B(h)) \geq \lambda_{\min}(J) - |B(h) - J|
        \geq \lambda_{\min}(J) - (4\varepsilon + 3\varepsilon^2)\sum_{i,j=1}^N J_{ij} > 0,
    \end{align*}
    since $\lambda_{\min}(J) > 0$ as the $(H_i)_{i\in\{1,\dots,N\}}$ are positive and pairwise distinct. Thus, in combination with \eqref{eq:eigenvalueboundregvarproof}, we have shown that there exist constants $C_* > 0$ and $\delta \in (0,1)$ such that $\lambda_{\min}(G(h)) \geq C_*\hspace{0.02cm}h^{2H_{\max}}\hspace{0.02cm} \ell_{\min}(1/h)^2$ for $h \in (0,\delta)$, which proves the assertion since $\ell_{\min}(1/h)>0$ by \eqref{eq: slowly varying bound}.
\end{proof}

Below, we collect important examples of regularly varying kernels for which the desired nondegeneracy can be obtained.

\begin{example}
    Let $H_1, \dots, H_N > 0$ be pairwise distinct, and consider $\lambda_1, \dots, \lambda_N \geq 0$, $c_1,\dots, c_N\neq 0$. The following frequently used kernels satisfy the assumptions of the previous lemma:
    \begin{enumerate}
        \item[(a)] Fractional kernels and gamma kernels given by $k_j(t) = c_j t^{H_j - 1/2} e^{-\lambda_j t}$.
        
        \item[(b)] Abel-type Mittag-Leffler kernels $k_j(t) = c_j t^{H_j-1/2} E_{H_j+1/2, H_j+1/2}(- \lambda_j\hspace{0.02cm} t^{H_j+1/2})$, where $E_{H_j+1/2, H_j+1/2}$ denotes the two-parameter Mittag-Leffler function.
    \end{enumerate}
    In particular, each combination of such kernels given by $\widetilde{K}^{\sigma}:=(k_1,\dots,k_N)^{\intercal}$ is $\gamma_*$-nondegenerate according to Definition \ref{def: nondegeneracy general} with $\gamma_* = H_{\max}$ and the regularity condition~\eqref{eq: upper bound} is satisfied with $\gamma_{\sigma} = H_{\min}$. Moreover, all of the above kernels can be also modified by slowly varying functions, which, in view of \eqref{eq: slowly varying bound}, implies the $(H_{\max}+\varepsilon)$-nondegeneracy of the resulting kernel while \eqref{eq: upper bound} still holds with $\gamma_{\sigma}=H_{\min}-\varepsilon$. As $\varepsilon>0$ can be selected arbitrarily small, this generalization has no effect on the regularity condition~\eqref{eq: H condition}. A natural class of such functions is of the form $\log(1+t^{-\beta})^p$ with parameters $p \in \R$ and $\beta > 0$.
\end{example}

Lemma \ref{lemma: fractional like} is naturally suited to verify the nondegeneracy of each column of a general kernel $\widetilde{K}^{\sigma} \in L_{\mathrm{loc}}^2(\R_+; \R^{N \times d})$, characterized by the Gram matrices $G^{(i)}(h)$, see \eqref{eq: Gram}. Nondegeneracy of the entire
kernel~$\widetilde{K}^{\sigma}$ then follows for the choice $\gamma_*=\max\{\gamma_*^1,\dots,\gamma_*^d\}$. Moreover, in the diagonal-like case, it can be directly applied to each $\widetilde{G}^{(i)}(h)$. Therefore, Lemma~\ref{lemma: fractional like} allows us to verify nondegeneracy for a large class of regularly varying kernels. This also includes the case $H_1 = 1/2$, where the first kernel component is regular at zero, provided that the remaining components do not exhibit the same asymptotics. However, when all kernels are regular (i.e., $H_1 = \dots = H_N = 1/2$), the situation may drastically change, as only higher-order asymptotics can be obtained. 
\begin{example}\label{example: regular}
    In dimension $d = 1$, consider $k, \widetilde{k} \in C^1(\R_+)$ with $k(0), \widetilde{k}(0) \neq 0$. Define 
\[
    M(h) = \int_0^h \begin{pmatrix} k(r)^2 & k(r) \widetilde{k}(r) \\ k(r) \widetilde{k}(r) & \widetilde{k}(r)^2 \end{pmatrix}\, \mathrm{d}r.
\]
Using the asymptotics $\int_{0}^{h}k(r)^2\,\mathrm{d}r\sim k(0)^2\hspace{0.02cm}h$, $\int_{0}^{h}\widetilde{k}(r)^2\,\mathrm{d}r\sim \widetilde{k}(0)^2\hspace{0.02cm}h$, and $\int_{0}^{h}k(r)\hspace{0.02cm}\widetilde{k}(r)\,\mathrm{d}r\sim k(0)\hspace{0.02cm}\widetilde{k}(0)\hspace{0.02cm}h$, where in all three cases the remainder is $\mathcal{O}(h^2)$ due to Taylor's theorem, we obtain for the first-order expansion for the smallest eigenvalue of $M(h)$:
\begin{align}\label{eq:lambdaminregularasymp}
   \notag\lambda_{\min}(M(h)) &=\frac{k(0)^2+\widetilde{k}(0)^2}{2}h+\mathcal{O}(h^2)-\frac{h}{2}\sqrt{\big(k(0)^2+\widetilde{k}(0)^2\big)^2+\mathcal{O}(h)}\\
        &=\frac{k(0)^2+\widetilde{k}(0)^2}{2}h+\mathcal{O}(h^2)-\frac{k(0)^2+\widetilde{k}(0)^2}{2}\hspace{0.02cm}h\hspace{0.02cm}\sqrt{1+\mathcal{O}(h)}=\mathcal{O}(h^2),
\end{align}
where we applied the expansion of the binomial series in the final step. Hence, concerning the nondegeneracy of the kernel $(k,\widetilde{k})^{\intercal}$, we necessarily have $\gamma_* \geq 1$ while $\gamma = 1/2$, which is highly restrictive with regard to \eqref{eq: H condition} for general Hölder continuous $b$, $\sigma$.
\end{example}

 However, as the considerations of Subsection~\ref{subsection: abscontkernelnonMarkov} will demonstrate, the admissible perturbations constructed in Subsections \ref{subsection:abstractMarkovproperty} and \ref{subsection:specialcasenonMarkovperturb} yield also kernels $\widetilde{k}$ with asymptotics different from the original kernel, thereby allowing us to utilize the full strength of Lemma~\ref{lemma: fractional like}.

\subsection{Markovian lift based on the shift semigroup}\label{subsection: abscontkernelnonMarkov}

Below, we study a Markovian lift based on \textit{time shifts} (TS). Such lifts have first been established, e.g.\ in~\cite{MR4503737}, for regular Volterra kernels in a weighted Sobolev space $W_{\mathrm{loc}}^{1,2}(\R_+; \R^{d \times d})$ with exponential tails. A modification that allows singular kernels was subsequently discussed in~\cite{MR4181950}, and has the dual of a Banach space equipped with the weak topology as its state-space. We follow the presentation and notation of \cite[Section 6]{BBCF25}, which provides a flexible Hilbert space framework, but now adjusted to diagonal regularly varying kernels with a particular choice for the weight function. 

\begin{enumerate}
    \item[(TS1)] (Small-time asymptotics and regularity) For $a \in \{b,\sigma\}$, let $k^a_1,\dots, k^a_d: \R_+^* \longrightarrow \R$ be continuously differentiable such that $(k_1^a)',\dots, (k_d^a)'$ are monotone in a right-neighborhood of zero. There exist $H^a_1, \dots, H^a_d \in (0,1)$ and slowly varying functions $\ell^a_1,\dots, \ell^a_d: \R_+^* \longrightarrow \R_+^*$ such that for each $j \in\{1,\dots, d\}$:
    \[
        k^a_j(t) \sim t^{H^a_j - 1/2}\ell^a_j(1/t) \ \text{ as } \ t \to 0, \ \ \text{ and } \ \ \int_{1}^{\infty} \big| (k^a_j)'(t)\big|^2\, \mathrm{d}t < \infty.
    \]
    Here and below we shall write $H^a_{\min} = \min\{H_1^a,\dots, H_d^a\}$ and $H_{\min} = \min\{H_{\min}^b, H_{\min}^{\sigma}\}$ with $H^a_{\max}$ and $H_{\max}$ denoting the corresponding maxima.
    
  \item[(TS2)] (Initial driving force) The admissible initial condition is an $\mathcal{F}_0$-measurable function $g: \R_+ \longrightarrow \R^d$ such that $g$ is a.s.\ absolutely continuous on $\R_+^*$, and there exist $\eta_g \in [0,1)$ and $p > 2$ with 
  \[
    \frac{2}{p} + 1 - 2H_{\min} < \eta_g
  \]
  and
    \[
        \E\big[|g(1)|^{p}\big] + \E\left[ \left(\int_{0}^{\infty} |g'(t)|^2\hspace{0.02cm} (1 \wedge t^{ \eta_g})\, \mathrm{d}t \right)^{p/2} \right] < \infty.
    \]
\end{enumerate}

Note that (TS2) is always satisfied for deterministic $g$ that are also constant in time. For given $k^a_1,\dots, k^a_d$, let us define the corresponding Volterra kernels by
\[
    K^a(t) = \mathrm{diag}\big(k^a_1(t), \dots, k^a_d(t)\big),\qquad t > 0, \ a \in \{b,\sigma\}.
\]
The next lemma verifies the kernel increment condition in assumption (A).

\begin{lemma}\label{lemma: AC B condition}
    Suppose that condition (TS1) is satisfied and fix $T > 0$. Then for each sufficiently small $\varepsilon>0$, there exists $C_{T,\varepsilon} > 0$ such that for $a \in \{b, \sigma\}$:
    \[
        \int_0^h |K^a(t)|^2\, \mathrm{d}t + \int_0^T |K^a(t+h) - K^a(t)|^2\, \mathrm{d}t \leq C_{T,\varepsilon}\hspace{0.02cm} h^{2H_{\min} - 2\varepsilon}, \qquad h \in (0,T].
    \]
    Furthermore, for each $\eta_K > 2 - 2H_{\min}$ we have
    \begin{equation}\label{eq:Kprimebound}
        \int_0^1 |(K^a)'(t)|^2 t^{\eta_K}\, \mathrm{d}t < \infty,
    \end{equation}
\end{lemma}
\begin{proof}
    Fix $a \in \{b,\sigma\}$. Using assumption (TS1) combined with standard bounds for slowly varying functions (see \eqref{eq: slowly varying bound}), we readily obtain $\int_0^h k^a_i(t)^2\, \mathrm{d}t \lesssim h^{2H^a_i - 2\varepsilon}$ for $i\in\{1,\dots, d\}$ and each small $\varepsilon > 0$. This shows that 
    \begin{align}\label{eq: 12}
        \int_0^h |K^a(t)|^2\, \mathrm{d}t \lesssim h^{2 H_{\min} - 2\varepsilon}, \qquad h \in (0,h_0],
    \end{align}
    for some $h_0>0$, and can be extended to $(0,T]$ by $K^a\in L^2([h_0,T];\R^{d\times d})$ due to continuity. Likewise, since $(k^a_i)'$ is monotone in a right-neighborhood of zero, we can apply the monotone density theorem (see \cite[Theorem 2]{MR94863}), to conclude that $\frac{t(k^a_i)'(t)}{k^a_i(t)} \longrightarrow H^a_i - 1/2$, and hence for $i\in\{1,\dots, d\}$:
    \begin{align}\label{eq: ki derivative 1}
        (k^a_i)'(t) \sim (H^a_i-1/2)\hspace{0.02cm}\frac{k^a_i(t)}{t} \sim (H^a_i-1/2)\hspace{0.02cm}t^{H^a_i - 3/2}\hspace{0.02cm}\ell^a_i(1/t), \qquad t \to 0,
    \end{align}
    provided that $H^a_i \neq 1/2$. For $H^a_i = 1/2$, one has
    \begin{align}\label{eq: ki derivative 2}
        \frac{t(k^a_i)'(t)}{\ell^a_i(1/t)} \longrightarrow 0, \qquad t \to 0.
    \end{align}
    In both cases, we obtain $|(k^a_i)'(t)| \lesssim t^{H^a_i - 3/2 - \varepsilon}$ for $\varepsilon > 0$ and $t\in (0,h_0]$ by \eqref{eq: slowly varying bound}, which can again be extended onto $[h_0,2T]$ by the boundedness of $(k^a_i)'$. 
    
    Hence, for $t \in (0,T]$ and $\varepsilon > 0$ when $H_i^{a} \in (0,\frac{1}{2}]$, and $\varepsilon \in (0, H_i^{a} - \frac{1}{2})$ when $H_i^{a} \in (\frac{1}{2},1 )$, we obtain
    \begin{align*}
        |k^a_i(t+h) - k^a_i(t)| &\leq \int_t^{t+h} |(k_i^a)'(r)|\, \mathrm{d}r 
        \lesssim \int_t^{t+h} r^{H_i^{a} - \frac{3}{2} - \varepsilon}\, \mathrm{d}r
        \\ &= \frac{(t+h)^{H_i^{a} - \frac{1}{2} - \varepsilon} - t^{H_i^{a} - \frac{1}{2} - \varepsilon} }{H_i^{a} - \frac{1}{2} - \varepsilon}
        = h^{H_i^{a} - \frac{1}{2} - \varepsilon} \frac{(1+\frac{t}{h})^{H_i^{a} - \frac{1}{2} - \varepsilon} - \left(\frac{t}{h}\right)^{H_i^{a} - \frac{1}{2} - \varepsilon} }{H_i^{a} - \frac{1}{2} - \varepsilon}.
    \end{align*}
    Thus, we obtain from the substitutions $u=\frac{t}{h}$ and $x=u^{-1}$:
    \begin{align*}
        \int_0^T |k^a_i(t+h) - k^a_i(t)|^2\, \mathrm{d}t 
        &\lesssim h^{2H_i^{a} - 1 - 2\varepsilon} \int_0^T  \left( \left(1+\frac{t}{h} \right)^{H_i^{a} - \frac{1}{2} - \varepsilon} - \left( \frac{t}{h}\right)^{H_i^{a} - \frac{1}{2} - \varepsilon} \right)^2\, \mathrm{d}t
        \\ &\leq h^{2H_i^{a}  - 2\varepsilon} \int_0^{\infty} \left( (1+u)^{H_i^{a} - \frac{1}{2} - \varepsilon} - u^{H_i^{a} - \frac{1}{2} - \varepsilon} \right)^2\, \mathrm{d}u
        \\ &= h^{2H_i^{a} - 2\varepsilon} \int_0^{\infty} u^{2H_i^{a}-1  - 2\varepsilon} \left( \left(1 + u^{-1}\right)^{H_i^{a} - \frac{1}{2} - \varepsilon} - 1 \right)^2\, \mathrm{d}u
        \\ &= h^{2H_i^{a} - 2\varepsilon} \int_0^{\infty} x^{-1 - 2H_i^{a} + 2\varepsilon} \left( \left(1 + x\right)^{H_i^{a} - \frac{1}{2} - \varepsilon} - 1 \right)^2\, \mathrm{d}x.
    \end{align*}
    This proves the first assertion, since the integral on the right-hand side is finite as its integrand is $\lesssim x^{1-2H_i^{a}+2\varepsilon}$ for $x\in (0,1]$. For the second assertion, note that
    \[
        \int_0^1 |(k^a_i)'(t)|^2 t^{\eta_K}\, \mathrm{d}t
        \lesssim \int_0^1 t^{2H^a_i - 3 - 2\varepsilon + \eta_K}\, \mathrm{d}t < \infty
    \]
    since $\eta_K > 2 - 2H^a_i + 2\varepsilon$ when $\varepsilon>0$ is small enough. This completes the proof.
\end{proof}

Next, we construct the corresponding Markovian lift. For given $\eta \geq 0$, let $\mathcal{H}_{\eta}$ be the weighted Hilbert space of equivalence classes of absolutely continuous functions $y: \R_+^* \longrightarrow \R^d$ with finite norm
\[
  \| y\|_{\eta}^2 = |y(1)|^2 + \int_{0}^{\infty} |y'(x)|^2\hspace{0.02cm} (1 \wedge x^{\eta})\, \mathrm{d}x < \infty.
\]
Here, $\eta$ captures the small-time regularity of the derivative. Note that $\mathcal{H}_{\eta} \subseteq \mathcal{H}_{\eta'}$ when $\eta' > \eta$. The next lemma summarizes the properties of the lift and, in particular, shows that condition (B) is satisfied.

\begin{lemma}\label{lemma: AC A condition}
    For each $\eta \geq 0$, $S(t)y(x) = y(t+x)$ with $t,x \geq 0$ defines a $C_0$-semigroup $(S(t))_{t \geq 0}$ on $\mathcal{H}_{\eta}$. If $\eta \in [0,1)$, then $\Xi \in L(\mathcal{H}_{\eta}, \R^d)$ where
    \[
        \Xi y = y(0) = y(1) - \int_0^1 y'(x)\, \mathrm{d}x.
    \]
    In particular, under assumption (TS1), let $\eta_K > 2 - 2H_{\min}$. Then for each $\eta \in [0,1)$ with $\eta \leq \eta_K < \eta + 1$, condition (B) is satisfied for 
    \[
        \mathcal{H} = \mathcal{H}_{\eta_K}, \ \ \mathcal{V} = \mathcal{H}_{\eta}, \ \ \xi_b = K^b, \ \ \xi_{\sigma} = K^{\sigma}, \ \ \text{and} \ \ \rho = \frac{\eta_K - \eta}{2},
    \]
    and $\|S(t)\|_{L(\mathcal{V})}\lesssim 1+\sqrt{t}$, $t\in\R_+$. Furthermore, under assumption (TS2), for each $T > 0$ and $0 \leq s < t \leq T$ it holds that
    \[
        \E[|g(t) - g(s)|^p] \lesssim_T (t-s)^{\frac{1-\eta_g}{2} p}.
    \]
\end{lemma}
\begin{proof}
    The first part follows from \cite[Lemma 6.1]{BBCF25} with $\delta = 0$ therein. In particular, \eqref{eq:Kprimebound} proves in combination with condition (TS1) that $K^b,K^{\sigma}\in L(\R^d,\mathcal{H}_{\eta_K})$ for every $\eta_K > 2 - 2H_{\min}$. Concerning the bound on $\|S(t)\|_{L(\mathcal{V})}$, the proof of \cite[Lemma 6.1]{BBCF25} establishes the existence of $C>0$ such that for all $t\in\R_+$, $y\in\mathcal{V}$:
    \begin{displaymath}
        \|S(t)y\|_{\mathcal{V}}^2\le C\hspace{0.02cm}\bigg(1+\int_1^{1+t}\frac{1}{1\wedge x^{\eta}}\,\mathrm{d}x \bigg)\hspace{0.02cm}\|y\|_{\mathcal{V}}^2=C\hspace{0.02cm}\big(1+t \big)\hspace{0.02cm}\|y\|_{\mathcal{V}}^2,
    \end{displaymath}
    where the latter is in our case justified by $\eta\ge 0$. Finally, let us prove the desired bound for the increments of $g$. First, we obtain 
    \begin{align*}
        |g(t) - g(s)| &\leq \int_s^{t}|g'(x)|\, \mathrm{d}x 
        \\ &\leq \left(\int_0^{\infty}|g'(x)|^2 (1 \wedge x^{\eta_g})\, \mathrm{d}x \right)^{1/2} \left( \int_s^{t} \frac{1}{(1 \wedge x^{\eta_g})}\, \mathrm{d}x \right)^{1/2}.
    \end{align*}
    When $s \geq 1$, the last integral is bounded by $\int_s^{t} \frac{1}{(1 \wedge x^{\eta_g})}\, \mathrm{d}x = t-s\lesssim_T (t-s)^{1-\eta_g}$ as $\eta_g\ge 0$. For $s \in [0,1)$, consider first $0 \le s < t < 1$. Then 
    \[
        \int_s^t \frac{1}{(1 \wedge x^{\eta_g})}\, \mathrm{d}x = \int_s^t x^{-\eta_g}\, \mathrm{d}x = \frac{t^{1 - \eta_g} - s^{1-\eta_g}}{1-\eta_g} \leq \frac{(t-s)^{1-\eta_g}}{1 - \eta_g}.
    \]
    When $0 \le s < 1 \leq t$, we find
    \begin{align*}
        \int_s^t \frac{1}{(1 \wedge x^{\eta_g})}\, \mathrm{d}x
        &= \int_s^1 x^{-\eta_g}\, \mathrm{d}x + t - 1
        = \frac{1 - s^{1-\eta_g}}{1 - \eta_g} + t - 1
        \\ &\leq \frac{t^{1-\eta_g} - s^{1-\eta_g}}{1-\eta_g} + t-s
        \leq \frac{(t-s)^{1 - \eta_g}}{1-\eta_g} + t -s 
        \lesssim_T (t-s)^{1-\eta_g}.
    \end{align*}
    The assertion now follows in combination with condition~(TS2).
\end{proof}

The following is our first main result on the failure of the Markov property for \eqref{eq:generalSVIE} under assumptions (TS1) and (TS2) with deterministic $g$ and $K:= K^b = K^{\sigma}$, where we set $H_i =H_i^b= H_i^{\sigma}$ for $i \in \{1,\dots, d\}$. It covers, in particular, Theorem \ref{thm: intro 1}.

\begin{theorem}\label{thm: time shift deterministic}
    Suppose that conditions (TS1), (TS2) are satisfied with $k_i := k_i^b = k_i^{\sigma}$ for all $i \in \{1,\dots, d\}$, $b \in B(\R_+; C^{\chi_b}(\R^d; \R^d))$, and $\sigma \in B(\R_+; C^{\chi_{\sigma}}(\R^d; \R^{d\times m}))$ with $\chi_b \in [0,1]$ and $\chi_{\sigma} \in (0,1]$. Assume that there exist $i_0 \in \{1,\dots, d\}$ and $m_{i_0} \in \mathbb{N}_0$ such that $k_{i_0} \in C^{\infty}(\R_+^*;\R)$, and $k_{i_0}^{(m_{i_0})} \in L_{\mathrm{loc}}^2(\R_+; \R)$ is completely monotone with its Bernstein measure $\mu_{i_0}$ having nondegenerate support. Finally, suppose that $g$ is deterministic and $\sigma$ is either diagonal or 
    \begin{align}\label{eq:parametercondExamples}
        H_{\max} < H_{\min} + \min \left\{ \frac{1}{2} + \chi_b \hspace{0.02cm}\overline{\gamma}_K, \ \chi_{\sigma}\hspace{0.02cm}\overline{\gamma}_K \right\},
    \end{align}
    where $\overline{\gamma}_K=(1-\eta_g)/2$ which can be improved to $\overline{\gamma}_K=H_{\min}$ when $g$ is constant. Then any continuous weak solution $X$ of \eqref{eq:generalSVIE} that satisfies $\P[X_{t_0} \in \Gamma_{\sigma,t_0}] > 0$ for some $t_0 > 0$ is not a Markov process with respect to $(\mathcal{F}_t^{X,B})_{t \in \R_+}$.
\end{theorem}
\begin{proof}
    First, we embed the SVE into the general framework for Markovian lifts established in Subsection \ref{section:abstractLifts}. When $\eta_g> 2- 2H_{\min}$, we may select $\eta_K=\eta_g$. If $\eta_g\le 2- 2H_{\min}$, we take $\eta_K > 2 - 2H_{\min}$ sufficiently close to the lower bound such that, in view of condition~(TS2), we obtain
    \[
        \frac{1}{p} + \frac{\eta_K - \eta_g}{2} < \frac{1}{2}.
    \]
    In particular, $\eta_g \in [0,1)\cap (1-2H_{\min},2-2H_{\min}]$ implies $\eta_g \leq \eta_K < \eta_g + 1$. By Lemma~\ref{lemma: AC A condition}, condition (B) is satisfied in both cases for $\eta = \eta_g$, $\mathcal{H} = \mathcal{H}_{\eta_K}$, $\mathcal{V} = \mathcal{H}_{\eta}$, $\rho = (\eta_K - \eta)/2$, $\xi_K := \xi_b = \xi_{\sigma}=K$, we have $|g(t) - g(s)| \lesssim (t-s)^{\frac{1-\eta_g}{2}}$, and condition (TS2) ensures $\xi=g\in\mathcal{V}$ and $g\in\mathcal{G}_p$ with $\frac{1}{p} + \rho < \frac{1}{2}$. 

    It follows from Lemma \ref{lemma:compmonadmissiblekernel} that $\widetilde{k} e_{i_0}^{\intercal}$ is admissible for any choice of $\widetilde{k} \in L^2_{\mathrm{loc}}(\R_+; \R)$, where $e_{i_0}$ denotes the $i_0$-th canonical basis vector in $\R^d$. Let $\widetilde{k}(t) = t^{H_{i_0}-1/2 + \varepsilon}$ where $\varepsilon > 0$. As $g$ is deterministic, the assumptions of Theorem \ref{thm: special case nonmarkov} are satisfied whenever $(X_{t_0}, Z_{t_0})$ is absolutely continuous with respect to the Lebesgue measure on $\Gamma_{\sigma, t_0} \times \R$, where $Z$ is defined according to \eqref{eq:Zperturbationdefinition} for the kernel $t^{H_{i_0}-1/2 + \varepsilon}\hspace{0.02cm}e_{i_0}^{\intercal}$. Below, we prove the desired absolute continuity.

    By $K = \mathrm{diag}(k_1,\dots, k_d)$, Lemma \ref{lemma: AC B condition} shows that the kernel bounds in condition (A) are satisfied for any choice $0 < \gamma_K < H_{\min}$. Since $\eta_g > 1-2H_{\min}$ by assumption, we may take $\gamma_K=(1-\eta_g)/2 $, whence we obtain also $|g(t) - g(s)| \leq (t-s)^{\gamma_K}$. Moreover, when $g$ is constant, we can select $\gamma_K=H_{\min}-\zeta$ where $\zeta>0$ will be chosen sufficiently small. Concerning the nondegeneracy, note that $\widetilde{K} = (K, \widetilde{k} e_{i_0}^{\intercal})^{\intercal}$ is diagonal-like with $S_j = \{j\}$ for $j \neq i_0$, and $S_{i_0} = \{i_0, d+1\}$ defined in \eqref{eq: S decomposition}. Then we obtain for $\xi \in \R^{d+1}$:
    \begin{align}\label{eq:detexampleslowerbound}
        \int_0^h \big|\widetilde{K}(r)^{\intercal}\xi\big|^2\, \mathrm{d}r = \sum_{j \in\{1,\dots,d\}\setminus \{i_0\}}\xi_j^2 \int_0^h k_j(r)^2\, \mathrm{d}r + \int_0^h \left( k_{i_0}(r) \xi_{i_0} + \widetilde{k}(r)\xi_{d+1}\right)^2\, \mathrm{d}r.
    \end{align}
    By \eqref{eq: slowly varying bound}, we find constants $c_*,c>0$ so that for every $\delta>0$ there exists $h_0>0$ such that for every $j\in\{1,\dots,d\}$: 
    \begin{equation}\label{eq:slowlyvarboundsexamples}
        c_*\hspace{0.02cm} h^{2(H_j+\delta)}\le\int_0^h k_j(r)^2\, \mathrm{d}r \le c\hspace{0.02cm} h^{2(H_j-\delta)},\quad h\in (0,h_0).
    \end{equation}
    For the terms in \eqref{eq:detexampleslowerbound} with $j \neq i_0$, this immediately yields a lower bound, while for the remaining term, we obtain from Lemma \ref{lemma: fractional like}: 
    \[
        \int_0^h \left( k_{i_0}(r) \xi_{i_0} + \widetilde{k}(r)\xi_{d+1}\right)^2\, \mathrm{d}r \geq C_* h^{2(H_{i_0} + \varepsilon+\delta)}\hspace{0.02cm} \left( \xi_{i_0}^2 + \xi_{d+1}^2 \right)
    \]
    for some $C_* > 0$ and $h$ small enough. In particular, the above shows that the diagonal-like kernel $\widetilde{K} = (K, \widetilde{k})^{\intercal}$ is $\gamma_*$-nondegenerate in the sense of Definition~\ref{def: nondegeneracy anisotrop}, where
    \[
        \gamma_*^j = \begin{cases}H_j+\delta, & j \neq i_0
        \\ H_{i_0}  + \varepsilon+ \delta, & j = i_0.
        \end{cases}
    \]
    
    Likewise, again by \eqref{eq:slowlyvarboundsexamples}, it follows that condition \eqref{eq: ani upper bound 1} is satisfied with $\gamma^j = H_j-\delta$ for all $j \in \{1,\dots, d\}$. In case $\sigma$ is diagonal, the desired absolute continuity follows from Theorem \ref{thm: density diagonal} since $\varepsilon, \delta > 0$ can be selected arbitrarily small, $|\gamma_*^j-\gamma^j|=2\delta+\varepsilon\hspace{0.02cm}\mathbbm{1}_{\{i_0\}}(j)$, and $0<\chi_{\sigma}\hspace{0.01cm}\gamma_K\le\chi^j_{\sigma}\hspace{0.02cm}\gamma_K$. For general $\sigma$, Remark \ref{remark:nondegeneracyconnection} guarantees that $\widetilde{K}$ is also $\gamma_*:=\max\{\gamma_*^1, \dots, \gamma_*^d\}$-nondegenerate in the sense of Definition \ref{def: nondegeneracy general}. In particular, $\gamma_*=\max_j(H_{j}+\delta+\varepsilon\hspace{0.02cm}\mathbbm{1}_{\{i_0\}}(j))$. Moreover, $\widetilde{K}$ fulfills the estimate \eqref{eq: upper bound} with $\gamma:=\min\{\gamma^1, \dots, \gamma^d\}=H_{\min}-\delta$. By selecting again $\varepsilon, \delta, \zeta > 0$ sufficiently small and recalling $|\gamma_K-\overline{\gamma}_K|\le\zeta$, \eqref{eq:parametercondExamples} shows that Theorem \ref{thm: density} is applicable. By Theorem \ref{thm: special case nonmarkov}, this completes the proof.
\end{proof}

In particular, by Remark \ref{remark:filtration}, $X$ is not a Markov process w.r.t.\ to \emph{any} filtration to which $(X,B)$ is adapted. While the asymptotics from (TS1) implies in case $H_{i_0} \neq \frac{1}{2}$ that $\mathrm{supp}(\mu_{i_0})$ necessarily has infinitely many elements, its support not necessarily contains a sequence $(\lambda_n)_{n \geq 1}$ such that $\sum_{n=1}^{\infty}\lambda_n^{-1} = \infty$. Indeed, $\mu_{i_0}(\mathrm{d}x) = \sum_{n=1}^{\infty} n^{-2H_{i_0}} \delta_{n^2}(\mathrm{d}x)$ does \textit{not} have nondegenerate support, while $k_{i_0}(t) = \int_{\R_+}e^{-tx}\, \mu_{i_0}(\mathrm{d}x) \sim c\hspace{0.02cm} t^{H_{i_0} - 1/2}$ for some $c > 0$, as can be seen for $H_{i_0}\in (0,1/2)$ from Karamata's Tauberian theorem \cite[Theorem~1.7.1]{BiGoTe87}. Hence, the additional assumption on the support of $\mu_{i_0}$ cannot be omitted.

Below, we illustrate the above result for affine Volterra processes with square-root diffusion coefficients.

\begin{example}[multivariate Volterra square-root process]\label{example:VolterraCIR}
    Let $b, V_0, (\sigma_1,\dots, \sigma_d)^{\intercal} \in \R_+^d$, $\beta \in \R^{d \times d}$ such that $\beta_{ij} \geq 0$ for $i \neq j$, and $0 \neq k_1,\dots, k_d \in L_{\mathrm{loc}}^2(\R_+; \R)$ be completely monotone such that for all $j \in \{1,\dots, d\}$:
    \[
        k_j(t) \sim t^{H_j - 1/2}\ell_j(1/t), \quad t \to 0, \quad \text{ and } \quad \int_{1}^{\infty} | (k_j)'(t)|^2\, \mathrm{d}t < \infty,
    \]
    where $H_j\in (0,1/2)$ and $\ell_j: \R_+^* \longrightarrow \R_+^*$ is a slowly varying function. The multivariate Volterra square-root process is defined as the unique $\R_+^d$-valued weak solution of 
    \begin{align*}
        V_t &= V_0 + \int_0^t K(t-s)\hspace{0.02cm}( b + \beta V_s)\, \mathrm{d}s + \int_0^t K(t-s)\hspace{0.02cm}\sigma(V_s)\, \mathrm{d}B_s,\quad t\in\R_+,
    \end{align*}
    where $\sigma(v) = \mathrm{diag}(\sigma_1 \sqrt{v_1}, \dots, \sigma_d \sqrt{v_d})$ and $K = \mathrm{diag}(k_1,\dots, k_d)$, see \cite[Theorem 6.1]{AbiJaLaPu19}.
    If there exists $i \in \{1,\dots, d\}$ such that the Bernstein measure of $k_i$ has nondegenerate support, then $V$ is not a Markov process provided that there exists $t_0 > 0$ such that $\P[V_{t_0} \in \R_+^d \,\backslash\, \partial \R_+^d] > 0$. In a yet unpublished working paper, we will show that the latter is satisfied even for all $t_0 > 0$. When $d=1$, the existence of $t_0>0$ with $\P[V_{t_0} \in \R_+^*] > 0$ is guaranteed as soon as $V$ is not the zero-process.
\end{example}

Note that in the above example, it suffices that only one $k_{i_0}$ has a Bernstein measure with nondegenerate support. For all $j \neq i_0$, $k_j \equiv 1$ is allowed. In particular, it covers the kernels $k_{i_0}(t) = (t+\varepsilon)^{H_{i_0}-\frac{1}{2}}e^{-\lambda t}$ with $H_{i_0} \in (0,1/2)$, and $\varepsilon, \lambda \geq 0$. Another admissible choice is $k_{i_0}(t) = \log(1 + (\varepsilon + t)^{-\alpha})\hspace{0.02cm}e^{-\lambda t}$ with $\varepsilon, \lambda \geq 0$, and $\alpha \in (0,1]$. Indeed, this function has Bernstein measure $\mathbbm{1}_{(\lambda, \infty)}(x)\hspace{0.02cm}\mu(x-\lambda)\hspace{0.02cm}e^{-\varepsilon x}\,\mathrm{d}x$ with $\mu$ given by the subordination formula $\mu(x) = \int_0^{\infty} \frac{1-e^{-u}}{u} \hspace{0.02cm}p_u(x)\,\mathrm{d}u$, where $p_u(x)$ denotes the density of the $\alpha$-stable Levy subordinator on~$\R_+$ evaluated at time $u$. 

We next turn to the Volterra Heston model, a standard example in the rough volatility literature where the instantaneous variance is modeled by a Volterra square-root process. For simplicity, we focus on a single asset with an underlying stochastic variance factor, noting that analogous results can be obtained for multiple assets and factors.

\begin{example}[Volterra Heston model]\label{example:VolterraHeston}
    The (rough) Volterra Heston model consists of a risky asset $S$ with $S_0>0$ and a variance process $V$ that follow on $\R \times \R_+$ the dynamics 
    \begin{equation}\label{eq:Volterra Heston}
    \begin{aligned}
        \frac{\mathrm{d}S_t}{S_t} &= \sqrt{V_t}\left( \rho \,\mathrm{d}B_t + \sqrt{1 - \rho^2}\,\mathrm{d}B^{\perp}_t \right)
        \\ V_t &= V_0 + \int_0^t k(t-s)\hspace{0.02cm}\kappa\hspace{0.02cm}( \theta- V_s)\, \mathrm{d}s + \sigma \int_0^t k(t-s)\sqrt{V_s}\, \mathrm{d}B_s, \quad t\in\R_+,
    \end{aligned}
    \end{equation}
    where $B, B^{\perp}$ are two independent standard Brownian motions, $V_0, \kappa,\theta \in\R_+$, $\rho \in [-1,1]$, $\sigma > 0$, and $0 \neq k \in L_{\mathrm{loc}}^2(\R_+; \R)$ is completely monotone satisfying for some $H\in(0, 1/2)$: 
    \[
        k(t) \sim t^{H-1/2}\ell(1/t), \quad t \to 0, \quad \text{ and } \quad \int_{1}^{\infty} | k'(t)|^2\, \mathrm{d}t < \infty,
    \]
    where $\ell:\R_+^*\longrightarrow\R_+^*$ is slowly varying. Weak existence and uniqueness for $(S,V)$ is established in \cite[Theorem 7.1]{AbiJaLaPu19}. If the Bernstein measure of $k$ has nondegenerate support and $H\in (1/3,1/2)$ when $\rho\neq 0$, then $(S,V)$ is not a Markov process.
\end{example}
\begin{proof} 
    Define $Y_t = \log(S_t)$, then 
    \[
        \mathrm{d}Y_t = - \frac{V_t}{2}\,\mathrm{d}t + \sqrt{V_t}\left( \rho \,\mathrm{d}B_t + \sqrt{1 - \rho^2}\,\mathrm{d}B^{\perp}_t \right). 
    \]
    In particular, since $\log$ is bijective, $(S,V)$ is a Markov process if and only if $X = (Y, V)$ is Markov. Below, we prove that the latter is not the case. Firstly, note that $X$ is an affine Volterra process and satisfies
    \[
        X_t = X_0 + \int_0^t K(t-s)\hspace{0.02cm}b(X_s)\, \mathrm{d}s + \int_0^t K(t-s)\hspace{0.02cm}\sigma(X_s)\, \mathrm{d}W_s,\quad t\in\R_+,
    \]
    where $W = (B^{\perp}, B)$ is a two-dimensional standard Brownian motion, $X_0=(Y_0, V_0)^{\intercal}$, $K = \mathrm{diag}(1, k)$, $b(x) = (-\frac{v}{2},\, \kappa(\theta- v))^{\intercal}$, and 
    \[
        \sigma(x) = \begin{pmatrix}
            \sqrt{1 - \rho^2}\sqrt{v} & \rho\sqrt{v} \\ 0 & \sigma \sqrt{v} \end{pmatrix}, \ \text{ where } \ x = (y, v) \in \R \times \R_+.
    \]
    Then we obtain from direct computation $\mathrm{det}(\sigma(x) \hspace{0.02cm}\sigma^{\intercal}(x)) = v^2  \sigma^2 \hspace{0.02cm}( 1 - \rho^2 )$. Consequently, if $\rho \in (-1,1)$, we have $\Gamma_{\sigma}\equiv\Gamma_{\sigma,t} = \R\times\R_+^*$. Moreover, note that there exists $t_0>0$ with $\P[ V_{t_0} > 0] > 0$ by $\E[ V_{t_0}^2 ] \neq 0$ as $V$ is not the zero-process. When $\rho=0$, $\sigma$ has a diagonal form, while for $\rho\neq 0$ and $H\in(1/3,1/2)$ we can deduce condition \eqref{eq:parametercondExamples} from $H_{\min}=H$, $H_{\max}=1/2$, $\chi_b=1$, $\chi_{\sigma}=1/2$, $\overline{\gamma}_K=H_{\min}$, and $1/2<\frac{3}{2}H$. Therefore, the assertion follows in both cases from Theorem \ref{thm: time shift deterministic}.

    For the case $\rho \in \{-1,1\}$, Theorem \ref{thm: time shift deterministic} is not directly applicable since $\Gamma_{\sigma} = \emptyset$. However, as $\Gamma_{\sigma}$ was used only to derive the bound \eqref{eq:2dimdensityB3estimation}, and hence to establish the existence of a density for $(X_t,Z_t)$ on $\Gamma_{\sigma, t} \times \R$, below we directly verify \eqref{eq:2dimdensityB3estimation} for some admissible perturbation. Take 
    \[
        \widetilde{K}(t) = \begin{pmatrix}
            1 & 0
            \\ 0 & k(t)
            \\ 0 & \widetilde{k}(t)
        \end{pmatrix}
    \]
    with $\widetilde{k}(t) = t^{\widetilde{H} - 1/2}$ such that $H \neq \widetilde{H}\in (0,1/2)$. For this choice, we obtain from the explicit forms of $\sigma(x)$, $\widetilde{K}$, $1-\rho^2=0$, Lemma~\ref{lemma: fractional like} and \eqref{eq: slowly varying bound} for an arbitrarily small $\delta>0$:
    \begin{align}\label{eq: nondegeneracy TS}
    \notag \int_s^t \left| \sigma(X_s)^{\intercal} \widetilde{K}(t-r)^{\intercal}\xi\right|^2\, \mathrm{d}r
    &= V_s \int_s^t \left( \rho\hspace{0.02cm} \xi_1 + \sigma \hspace{0.02cm}k(t-r)\hspace{0.02cm}\xi_2 + \sigma\hspace{0.02cm} \widetilde{k}(t-r)\hspace{0.02cm}\xi_3\right)^2\, \mathrm{d}r
    \\ &\geq \big(1\wedge \sqrt{V_s}\big)^2\hspace{0.02cm} C_*\hspace{0.02cm} |\xi|^2\hspace{0.02cm}(t-s)^{1+2\delta},
    \end{align}
    where $\xi \in \R^3$ and, therefore, again \eqref{eq:2dimdensityB3estimation} with $\rho=1\wedge \sqrt{V}$. Hence, the law of the random variable $\widetilde{Z}_t = (X_t, Z_t)$ defined as in \eqref{eq:perturbedVolterraprocessDef} with the kernel given by $\widetilde{K}$ is, restricted to $\{V_t>0\}$, absolutely continuous. The assertion now follows by a minor modification of the final step in the proof of Theorem \ref{thm: time shift deterministic}.  
\end{proof}

Below, we complement Theorem \ref{thm: time shift deterministic} by allowing $K^b\neq K^{\sigma}$ and considering not necessarily deterministic $g$. In addition, we will prove that the corresponding Markovian lift has infinite-dimensional range in the sense of Corollary \ref{cor: dimension}. Moreover, the additional parameters $\beta_b, \beta_{\sigma}$ below will allow for different power-law asymptotics for small and large $t$, covering, e.g., kernels of the form $t^{H - \frac{1}{2}} + t^{\beta}$ with $\beta\ge H - \frac{1}{2}$.

\begin{theorem}\label{theorem:NonMarkovSec5MostAbstract}
    Assume conditions (TS1), (TS2), let $b \in B(\R_+; C^{\chi_b}(\R^d; \R^d))$, and $\sigma \in B(\R_+; C^{\chi_{\sigma}}(\R^d; \R^{d\times m}))$ where $\chi_b \in [0,1]$ and $\chi_{\sigma} \in (0,1]$. Suppose that there exists $i_0 \in \{1,\dots, d\}$ with the properties:
    \begin{enumerate}
        \item[(i)] $|H^{b}_{i_0}-H^{\sigma}_{i_0}|<1/2$, and $H_{i_0}^{\sigma} \neq \frac{1}{2}$;
        \item[(ii)] There exist $0 \leq \beta_b,\beta_{\sigma} < \min\{ H_{i_0}^b, H_{i_0}^{\sigma}, \frac{1 - \eta_g}{2}\}$ and $h_0 \in (0,1)$ such that for $t \in (0,h_0)$ and $u \geq 1$:
        \[
            \left| \frac{k_{i_0}^a(t\hspace{0.02cm} u)}{k_{i_0}^a(t)}\right| \lesssim u^{\beta_a}, \qquad a \in \{b,\sigma\}.
        \]
    \end{enumerate}
    Finally, let us define the constants $H_{\min,i_0}^{a}:=\min\big\{H_{\min}^{a},H_{i_0}^{a}- \max\{\beta_b, \beta_{\sigma}\}\big\}>0$, for $a\in\{b,\sigma\}$, and $\overline{\gamma}_K=(1-\eta_g)/2$, which can be improved to $\overline{\gamma}_K=H_{\min}$ when $g$ is constant in $t$. Suppose that the parameters satisfy 
    \begin{align}\label{eq: H condition Section 5}
        H_{\max}^{\sigma} <  \min\left\{H_{\min,i_0}^{b}+ \tfrac{1}{2} + \chi_b \overline{\gamma}_K,\ H_{\min,i_0}^{\sigma}+\chi_{\sigma} \overline{\gamma}_K\right\},
    \end{align} 
    and if $\sigma$ is diagonal, this condition can be replaced to
    \begin{align}\label{eq: H condition Section 5 diagonal}
        H_{i_0}^{\sigma} <  \min\left\{H_{i_0}^{b} - \max\{\beta_b, \beta_{\sigma}\} + \tfrac{1}{2} + \chi_b^{i_0}\hspace{0.02cm} \overline{\gamma}_K,\ H_{i_0}^{\sigma} - \max\{\beta_b, \beta_{\sigma}\} + \chi_{\sigma}^{i_0}\hspace{0.02cm} \overline{\gamma}_K\right\},
    \end{align}
    where $\chi_b^{i_0}\ge\chi_b$ and $\chi_{\sigma}^{i_0}\ge \chi_{\sigma}$ are determined by the Hölder regularity of the component functions $b_{i_0}$ and $\sigma_{i_0}$, respectively.
    
    Then any continuous weak solution $X$ of \eqref{eq:generalSVIE} that satisfies $\P[X_{t_0} \in \Gamma_{\sigma,t_0}] > 0$ for some $t_0 > 0$ is not a Markov process with respect to $(\mathcal{F}_t^{X,B})_{t \in \R_+}$. Furthermore, the corresponding Markovian lift $\mathcal{X}$ has infinite-dimensional range.
\end{theorem}
\begin{proof}
    First, observe that the SVE can be embedded into the framework for Markovian lifts from Subsection \ref{section:abstractLifts} exactly as in the proof of Theorem \ref{thm: time shift deterministic}. In particular, condition~(B) is satisfied for $\mathcal{H} = \mathcal{H}_{\eta_K}$ with a suitable $\eta_K > 2 - 2H_{\min}$, $\mathcal{V} = \mathcal{H}_{\eta}$ with $\eta=\eta_g$, $\rho = (\eta_K - \eta)/2$, $\xi_b=K^b$, $\xi_{\sigma}=K^{\sigma}$, $g\in\mathcal{G}_p$ with $\frac{1}{p} + \rho < \frac{1}{2}$, and condition (A) is satisfied for $\gamma_K=(1-\eta_g)/2 $, noting that when $g$ is constant, the latter can be improved to $\gamma_K=H_{\min}-\zeta$ with an arbitrarily small $\zeta>0$.
    
    Let us construct for each $N \geq 1$ admissible perturbations $\widetilde{\Xi}_1(t),\dots, \widetilde{\Xi}_N(t)$, $t>0$, that can be used to generate $\R^{N\times d}$-valued nondegenerate Volterra kernels. Fix $i_0\in \{1,\dots, d\}$ as in the assumptions. Define for $j \in\{1,\dots, N\}$ and $t>0$:
    \begin{equation}\label{eq:operatordefforwarddiff}
        \widetilde{\Xi}_j(t)y = \int_{0}^{\infty}\left( y_{i_0}(t+z) - y_{i_0}(t)\right) \,  \frac{e^{-z}\,\mathrm{d}z}{z^{1 + \alpha_j}},\quad y\in\mathcal{D}_j\subseteq\mathcal{V},
    \end{equation}
    where $\alpha_1,\dots, \alpha_N \in \big( \max\{\beta_b, \beta_{\sigma}\},\, \min\{H^{b}_{i_0},H^{\sigma}_{i_0}\}\big)$ are pairwise distinct. Note that the interval is nondegenerate by $\max\{\beta_b, \beta_a\} < \min\{H_{i_0}^b, H_{i_0}^{\sigma}\}$ according to condition (ii). Let us further define 
    \[
        \widetilde{k}_j^{a}(t) = \int_0^{\infty}\left( k_{i_0}^{a}(t+z) - k_{i_0}^{a}(t)\right)\frac{e^{-z}\,\mathrm{d}z}{z^{1 + \alpha_j}}.
    \]
    By Lemma \ref{lemma: appendix B} combined with condition (ii) and assumption (TS1), it follows that $\widetilde{k}_j^{a} \in L_{\mathrm{loc}}^2(\R_+; \R)$ and it satisfies \eqref{eq:Kintegrabilityadmissibledifferentiation} with $\nu_{i_0}^{(j)}(\mathrm{d}z) = e^{-z}z^{-1 - \alpha_{j}}\, \mathrm{d}z$. Since $\|S(z)\|_{L(\mathcal{V})}\lesssim 1+\sqrt{z}$, $z\in\R_+$ by Lemma \ref{lemma: AC A condition}, Example~\ref{admissible operator fractional differentiation} shows that the operators $\widetilde{\Xi}_j(t)$ are admissible on the (joint) domain
    \[
        \mathcal{D} = \left\{ y \in \mathcal{H}_{\eta} \ : \ \int_0^{\infty} |y_{i_0}(t+z) - y_{i_0}(t)|\, e^{-z}z^{-1 - \alpha_{j}}\, \mathrm{d}z < \infty, \ \ \forall t > 0, \ \ \forall j \in\{ 1,\dots, N\} \right\}.
    \] 
    Repeating the final part of the proof of Lemma \ref{lemma: AC A condition} and noting $\eta=\eta_g$, we observe for each $y \in \mathcal{H}_{\eta}$ that $|y_{i_0}(t+z) - y_{i_0}(t)| \lesssim_T \| y\|_{\eta}\hspace{0.02cm} z^{\frac{1-\eta}{2}}$ for $t \in [0,T]$, arbitrary $T>0$, and $z \in (0,1]$, and $|y_{i_0}(t+z) - y_{i_0}(t)| \lesssim \|y\|_{\eta}\hspace{0.02cm} z^{1/2}$ for $z \geq 1$. Hence, it follows that $\mathcal{D} = \mathcal{H}_{\eta}$ whenever $\alpha_j < \frac{1 - \eta}{2}$. As condition (ii) ensures $\max\{\beta_b, \beta_{\sigma}\} < \frac{1 - \eta}{2}$, by taking $\alpha_j > \max\{\beta_b, \beta_{\sigma}\}$ close enough to its lower bound, we may always select for every $N\ge 1$ a family $(\alpha_j)_{j\in\{1,\dots,N\}}$ that satisfies $\alpha_j < \frac{1 - \eta}{2}$. Moreover, note that the increment bound on $y_{i_0}$ remains valid for $t=0$, whence~\eqref{eq:operatordefforwarddiff} defines also for $t=0$ a bounded linear operator $\widetilde{\Xi}_j$ on $\mathcal{H}_{\eta}$. In particular, we obtain the consistency property $\widetilde{\Xi}_j(t)=\widetilde{\Xi}_j S(t)$ for every $t>0$ and $j\in\{1,\dots,N\}$.

    Let $e_{i_0} \in \R^d$ denote the $i_0$-th canonical basis vector. Then $\widetilde{K}^{a} = (K^{a}, \widetilde{k}_1^{a} e_{i_0}^{\intercal})^{\intercal}$ are diagonal-like with $S_j = \{j\}$ for $j \neq i_0$ and $S_{i_0} = \{i_0, d+1\}$ defined as in \eqref{eq: S decomposition}. The assertion follows from Theorem \ref{thm: abstract nonmarkov}, once we have shown that the law of $\widetilde{Z}_t$ given by
    \[
        \widetilde{Z}_t = \widetilde{\Xi}(t)\xi + \int_0^t \widetilde{K}^b(t-s)\hspace{0.02cm}b(s,X_s)\, \mathrm{d}s + \int_0^t \widetilde{K}^{\sigma}(t-s)\hspace{0.02cm}\sigma(s,X_s)\, \mathrm{d}B_s
    \]
    is absolutely continuous to the Lebesgue measure on $\Gamma_{\sigma, t} \times \R$ for $t=t_0$. Since $H_{i_0}^{\sigma} \neq \frac{1}{2}$, the constant $C(\alpha_1)$ in Lemma \ref{lemma: appendix B} does not vanish, whence $\widetilde{k}_1^{\sigma}$ is regularly varying with index $H_{i_0}^{\sigma} - \frac{1}{2} - \alpha_1$. In particular, it follows from Lemma \ref{lemma: fractional like} and the specific structure of $S_1,\dots, S_d$ that $\widetilde{K}^{\sigma}$ is $\gamma_*$-nondegenerate with $\gamma_*^j = H^{\sigma}_j+\delta$ for $j \in \{1,\dots, d\}$ and an arbitrarily small $\delta>0$. Likewise, due to \eqref{eq: B appendix regular varying}, also condition \eqref{eq: ani upper bound 1} is satisfied for $\gamma_{a}^j = H_{j}^{a}-\delta$ when $j \neq i_0$ and $\gamma_{a}^{i_0} = H^{a}_{i_0}-\delta - \alpha_{1}$. Hence, under \eqref{eq: H condition Section 5}, we may apply Theorem~\ref{thm: density}, or, in view of \eqref{eq: H condition Section 5 diagonal} for diagonal $\sigma$, its anisotropic analogue given in Theorem~\ref{thm: density diagonal}, provided that $\alpha_{1}$ is sufficiently close to its lower bound $\max\{\beta_b, \beta_{\sigma}\}$ and $\delta,\zeta>0$ are chosen small enough.

    For the second assertion, note that also $\widetilde{K}^{a} = (\widetilde{\Xi}_1S(\cdot)\xi_{a}, \dots, \widetilde{\Xi}_N S(\cdot)\xi_{a})^{\intercal}$, $a\in\{b,\sigma\}$, are diagonal-like with $S_{i_0} = \{1, \dots, N\}$. Moreover, $\widetilde{K}^{\sigma}$ is $\gamma_*$-nondegenerate for any choice of $N \geq 1$ with $\gamma_* = H^{\sigma}_{i_0}-\min_{j\in\{1,\dots,N\}}\alpha_{j}+\delta$. Likewise, also condition \eqref{eq: upper bound} is satisfied for $\gamma_a = H^a_{i_0} - \max_{j\in\{1,\dots,N\}}\alpha_{j}-\delta$. As $\delta>0$ can again be chosen arbitrarily small, \eqref{eq: H condition} is implied by
    \begin{align*}
        H^{\sigma}_{i_0}-\min_{j\in\{1,\dots,N\}}\alpha_{j} &< H^b_{i_0} - \max_{j\in\{1,\dots,N\}}\alpha_{j} + \frac{1}{2} + \chi_b \overline{\gamma}_K,
        \\ H^{\sigma}_{i_0}-\min_{j\in\{1,\dots,N\}}\alpha_{j} &< \ H^{\sigma}_{i_0} - \max_{j\in\{1,\dots,N\}}\alpha_{j} + \chi_{\sigma}\overline{\gamma}_K.
    \end{align*}
    Choosing $\alpha_1,\dots, \alpha_N$ distinct, but sufficiently close to each other such that  
    \begin{displaymath}
        \max_{j\in\{1,\dots,N\}}\alpha_{j} - \min_{j\in\{1,\dots,N\}}\alpha_{j} < \min\big\{1/2-|H^{b}_{i_0}-H^{\sigma}_{i_0}|, \chi_{\sigma} \overline{\gamma}_K\big\}
    \end{displaymath}
    shows that both conditions are automatically satisfied due to $|H^{b}_{i_0}-H^{\sigma}_{i_0}|<1/2$. The assertion follows from Corollary \ref{cor: dimension}. 
\end{proof}

Next, we provide a proof for Theorem \ref{thm: intro} from the introduction.

\begin{proof}[Proof of Theorem \ref{thm: intro}]
    This theorem is a particular case of Theorem \ref{theorem:NonMarkovSec5MostAbstract} where $k_i^a = k^a$ is independent of $i \in \{1,\dots, d\}$. Thus, $H^b = H_{i}^b$, $H^{\sigma} = H_i^{\sigma}$, $\beta_b = (H^b - \frac{1}{2})_+$, $\beta_{\sigma} = (H^{\sigma} - \frac{1}{2})_+$, and, hence, condition \eqref{eq: H condition Section 5} reduces to 
    \begin{align*}
        H^{\sigma} + \max \left\{ (H^b - \tfrac{1}{2})_+, (H^{\sigma} - \tfrac{1}{2})_+ \right\} &< H^b + \tfrac{1}{2} + \chi_b H_{\mathrm{min}}
        \\ \max \left\{ (H^b - \tfrac{1}{2})_+, (H^{\sigma} - \tfrac{1}{2})_+ \right\} &< \chi_{\sigma}H_{\min}.
    \end{align*}
    If $H_{\max} \leq 1/2$, both conditions are satisfied by $\chi_{\sigma}H_{\min}>0$ and $|H^b-H^{\sigma}|<\frac{1}{2}$. When $H_{\max} \in (\frac{1}{2},1)$, we obtain the following cases and corresponding conditions:
    \begin{enumerate}
        \item[(i)] $H^{\sigma} < H^b$ with $H^b < \frac{1}{2} + \chi_{\sigma}H^{\sigma}$;
        \item[(ii)] $H^b \leq H^{\sigma}$ with $H^{\sigma} < \frac{1}{2} + \min\{\chi_{\sigma}, \frac{1+\chi_b}{2}\}H^b$. 
    \end{enumerate}
    Both conditions are satisfied under the assumptions of Theorem \ref{thm: intro}.
\end{proof}

\begin{remark}\label{remark:abstractnonMarkovsimpli}
    Conditions \eqref{eq: H condition Section 5} and \eqref{eq: H condition Section 5 diagonal} simplify significantly once there exists $i_0\in\{1,\dots,d\}$ such that $H^{b}_{i_0}, H^{\sigma}_{i_0} \in (0, 1/2)$, i.e.\ both $k_{i_0}^b$, $k_{i_0}^{\sigma}$ are singular. Here, condition~(i) holds trivially and, if $\beta_b=\beta_{\sigma}=0$, \eqref{eq: H condition Section 5} simplifies to
    \begin{align*}
        H_{\max}^{\sigma} <  \min\left\{H_{\min}^{b}+ \tfrac{1}{2} + \chi_b \overline{\gamma}_K,\ H_{\min}^{\sigma}+\chi_{\sigma} \overline{\gamma}_K\right\},
    \end{align*}
    while for diagonal $\sigma$, \eqref{eq: H condition Section 5 diagonal} is always satisfied due to $\chi_{\sigma}^{i_0}\hspace{0.03cm} \overline{\gamma}_K>0$. 
\end{remark}

As before, it suffices that the Volterra kernel $K^{\sigma}$ is regularly varying with a non-zero index only in one component $k_{i_0}^{\sigma}$, while the indices in all other components are arbitrary. Furthermore, the Volterra square-root process and the rough Heston model discussed in Examples \ref{example:VolterraCIR} and \ref{example:VolterraHeston} are covered by the above theorem, whence their Markovian lift is infinite-dimensional. Likewise, our results contain equations driven by additive fractional noise. We illustrate this for the simple case of additive Riemann-Liouville noise:
\begin{example}
    Let $b \in B(\R_+; C^{\chi_b}(\R; \R))$ with $\chi_b \in [0,1]$ and 
    \[
        H\in \big(0,\, \tfrac{3+\chi_b}{4}\big)\setminus\left\{\tfrac{1}{2}\right\}.
    \]
    Then any continuous weak solution $X$ of the stochastic equation 
    \[
        X_t = x_0 + \int_0^t b(s,X_s)\, \mathrm{d}s + \sigma\int_0^t \frac{(t-s)^{H - \frac{1}{2}}}{\Gamma(H+\frac{1}{2})} \, \mathrm{d}B_s,\quad t\in\R_+,
    \]
    where $x_0\in\R$, $\sigma\neq 0$, does not possess the Markov property. Moreover, its Markovian lift is infinite-dimensional.
\end{example}

\begin{appendices}

\section{Anisotropic regularity of densities for Volterra Ito-processes}

Let $N \geq 1$. An $N$-tuple of positive numbers $a:=(a_1,\dots,a_N)\in(\R_+^*)^N$ is called anisotropy when it satisfies $a_1+\dots+a_N=N$. Consider $\lambda>0$ such that $\lambda/a_j\in (0,1)$ for every $j\in\{1,\dots,N\}$. Then the anisotropic Besov space $B_{1,\infty}^{\lambda,a}(\R^N)$ with parameters $\lambda,a$ is defined as the Banach space of integrable scalar functions $f\in L^1(\R^N)$ equipped with the norm
\begin{displaymath} 
\|f\|_{B_{1,\infty}^{\lambda,a}(\R^N)}:=\|f\|_{L^1(\R^N)}+\sum_{j=1}^N \sup_{h\in [-1,1]}|h|^{-\lambda/a_j}\hspace{0.02cm}\|f(\cdot + h \hspace{0.02cm}e_j)-f\|_{L^1(\R^N)},
\end{displaymath}
where $e_j$ denotes the $j$-th canonical basis vector. The anisotropic Hölder-Zygmund space $C_{b}^{\lambda,a}(\R^N)$ is defined as the Banach space of bounded scalar functions $\phi\in L^{\infty}(\R^N)$ endowed with the norm
\begin{displaymath} 
\|\phi\|_{C_{b}^{\lambda,a}(\R^N)}:=\|\phi\|_{L^{\infty}(\R^N)}+\sum_{j=1}^N \sup_{h\in [-1,1]}|h|^{-\lambda/a_j}\hspace{0.02cm}\|\phi(\cdot + h \hspace{0.02cm}e_j)-\phi\|_{L^{\infty}(\R^N)}.
\end{displaymath}
Remark that $\phi$ is globally H\"older continuous since for $x,y \in \R^N$ we obtain 
\begin{align}
    |\phi(x) - \phi(y)| 
    \leq 2\| \phi\|_{C_b^{\lambda,a}(\R^N)} \sum_{j=1}^N  |x_j - y_j|^{\lambda/ a_j}. \label{eq:anisotropictestfctHölder}
\end{align} 
For more details and further useful properties of these spaces, we refer to \cite{T06}. 
In the following, we study the absolute continuity of the law for the Volterra Ito-process 
\begin{align}\label{eq:anisotropicVolterraIto}
    X_t = \widetilde{g}(t) + \int_0^t \widetilde{K}^b(t-s)\hspace{0.02cm}b_s\, \mathrm{d}s
    + \int_0^t \widetilde{K}^{\sigma}(t-s)\hspace{0.03cm} \mathrm{diag}(\sigma_s^1, \dots, \sigma_s^d)\, \mathrm{d}B_s,\quad t\in\R_+.
\end{align}
Here, $B$ denotes a $d$-dimensional Brownian motion on some filtered probability space $(\Omega,\mathcal{F},(\mathcal{F}_t)_{t \in\R_+},\mathbb{P})$, the initial curve $\widetilde{g}:\Omega\times\R_+\longrightarrow\R^N$ is $\mathcal{F}_0\otimes \mathcal{B}(\R_+)$-measurable, $\widetilde{K}^b, \widetilde{K}^{\sigma} \in L_{\mathrm{loc}}^2(\R_+; \R^{N \times d})$, and $b: \Omega\times \R_+ \longrightarrow \R^d$ and $\sigma^1,\dots, \sigma^d: \Omega\times \R_+ \longrightarrow \R$ are progressively measurable processes.

\begin{proposition}\label{proposition:VolterraItoanisotropdensityperturbations}
    Let $X$ be the Volterra Ito-process \eqref{eq:anisotropicVolterraIto}. Suppose that the following conditions are satisfied:
    \begin{enumerate}
        \item[(D1)] For each $j \in \{1,\dots, N\}$ and $a\in\{b,\sigma\}$, there exist $\widetilde{k}^{a}_j\in L_{\mathrm{loc}}^2(\R_+;\R)$ and a unique $a(j) \in \{1,\dots, d\}$ such that
        \[
            \widetilde{K}^{a}_{j\ell}(t) = \delta_{\ell a(j)}\hspace{0.02cm}\widetilde{k}^{a}_j(t), \qquad t > 0, \ \ \ell \in \{1,\dots, d\}.
        \]
        \item[(D2)] Define $S_i = \{ j \in \{1,\dots, N\} \, : \, a(j) = i \}$, $i\in\{1,\dots, d\}$. There exist $C_*, C>0$, $\gamma_{b} = (\gamma_{b}^1,\dots, \gamma_{b}^d), \gamma_{\sigma} = (\gamma_{\sigma}^1,\dots, \gamma_{\sigma}^d), \gamma_* = (\gamma_*^1,\dots, \gamma_*^d)\in (\R_+^*)^d$, and $h_0 \in (0,1)$ such that for each $i \in \{1,\dots, d\}$ and $a\in\{b,\sigma\}$:
        \[
            \int_0^h \big|\widetilde{k}^{a}_j(r)\big|^2\, \mathrm{d}r \leq C h^{2\gamma_{a}^i}, \qquad j \in S_i, \ h \in (0,h_0),
        \]
        holds, and the Gram matrix $\widetilde{G}_{\ell \ell'}^{(i)}(h) = \int_0^h \widetilde{k}^{\sigma}_{\ell}(r)\widetilde{k}^{\sigma}_{\ell'}(r)\, \mathrm{d}r$, $\ell, \ell' \in S_i$, satisfies
        \[
            \lambda_{\min}\big(\widetilde{G}^{(i)}(h)\big) \geq C_* h^{2\gamma_*^i}, \qquad h \in (0,h_0),
        \]
        and for each $i \in \{1,\dots, d\}$ with $S_i=\emptyset$ we may w.l.o.g.\ select $\gamma_{a}^{i}=\infty$ and $\gamma_*^{i}=0$.
        \item[(D3)] There exist $p\ge 2$ and $\alpha_b^i \geq 0$ and $\alpha_{\sigma}^i>0$ with $i \in\{ 1,\dots, d\}$ such that for each $T>0$ we find $C_{p,T}>0$ satisfying for all $0<s<t\le T$:
        \begin{displaymath}
            \big\|b_t^i-b_s^i\big\|_{L^p(\Omega)}\le C_{p,T}\hspace{0.02cm}(t-s)^{\alpha_b^i}\quad\mbox{and}\quad \big\|\sigma_t^i-\sigma_s^i\big\|_{L^p(\Omega)}\le C_{p,T}\hspace{0.02cm}(t-s)^{\alpha_{\sigma}^i}.
        \end{displaymath}
        \item[(D4)] The parameters satisfy the relation 
        \[
            \gamma_*^i < \min  \left\{ \gamma_b^i + \frac{1}{2} + \alpha_b^i, \ \gamma_{\sigma}^i + \alpha_{\sigma}^i \right\}, \qquad i \in \{1,\dots, d\}.
        \]
    \end{enumerate}
    Then, for every $t>0$, the measure 
    \begin{equation}\label{eq:anisotropicmeasureperturbations}
        \mathcal{B}\big(\R^{N}\big)\ni A\longmapsto \mathbb{E}\left[ \left(1\wedge\min_{i\in\{1,\dots,d\}}\big|\sigma_t^i\big|\right) \mathbbm{1}_A(X_t)\right],
\end{equation}
   is absolutely continuous with respect to the Lebesgue measure on $\R^{N}$.
\end{proposition}
\begin{proof}
   Fix $t>0$, select $T=t+1$ and an arbitrary $\varepsilon\in (0, h_0\wedge t)$ and let us introduce the difference operator $\Delta_yf:=f(\cdot+y)-f$ for $f:\R^N\longrightarrow\R$ and $y\in\R^N$. Define $X_t^{\varepsilon} = U_t^{\varepsilon} + V_t^{\varepsilon}$ on $\R^N$, where
   \begin{align*}
       U_t^{\varepsilon}&= \widetilde{g}(t) + \int_0^{t-\varepsilon} \widetilde{K}^b(t-s)\hspace{0.02cm}b_s\, \mathrm{d}s + \int_{t-\varepsilon}^{t} \widetilde{K}^b(t-s)\hspace{0.02cm} b_{t-\varepsilon}\, \mathrm{d}s 
       \\ &\qquad  \quad\hspace{-0.02cm} + \int_0^{t-\varepsilon}\widetilde{K}^{\sigma}(t-s)\hspace{0.03cm} \mathrm{diag}(\sigma^1_s, \dots, \sigma_s^d)\, \mathrm{d}B_s,
       \\ V_t^{\varepsilon } &= \int_{t-\varepsilon}^t\widetilde{K}^{\sigma}(t-s)\hspace{0.02cm} \mathrm{diag}(\sigma^1_{t-\varepsilon}, \dots, \sigma_{t - \varepsilon}^d)\, \mathrm{d}B_s.
   \end{align*}   
   Then we immediately obtain
    \begin{equation*}
   \begin{aligned}
       X_t-X_t^{\varepsilon}
       = \int_{t-\varepsilon}^{t} \widetilde{K}^b(t-s)\hspace{0.02cm}(b_{s} - b_{t-\varepsilon})\, \mathrm{d}s 
       + \int_{t-\varepsilon}^t \widetilde{K}^{\sigma}(t-s)\hspace{0.02cm} \mathrm{diag}(\sigma_s^1 - \sigma_{t-\varepsilon}^1, \dots, \sigma_s^d - \sigma_{t-\varepsilon}^d)\, \mathrm{d}B_s,
   \end{aligned}
   \end{equation*}
   and for fixed $j \in \{1,\dots, N\}$, by assumption (D1), its $j$-th coordinate process has the form
   \begin{align*}
       X_t^j - X_t^{\varepsilon, j}
       &= \int_{t-\varepsilon}^t \sum_{\ell = 1}^d \widetilde{K}^b_{j\ell}(t-s)\big(b_s^{\ell} - b_{t-\varepsilon}^{\ell}\big)\, \mathrm{d}s
        + \int_{t-\varepsilon}^t \sum_{\ell=1}^d \widetilde{K}^{\sigma}_{j \ell}(t-s) \big(\sigma_s^{\ell} - \sigma_{t- \varepsilon}^{\ell}\big)\, \mathrm{d}B_s^{\ell}
        \\ &= \int_{t-\varepsilon}^t  \widetilde{k}^b_{j}(t-s)\big(b_s^{a(j)} - b_{t-\varepsilon}^{a(j)}\big)\, \mathrm{d}s
        + \int_{t-\varepsilon}^t \widetilde{k}^{\sigma}_{j}(t-s) \big(\sigma_s^{a(j)} - \sigma_{t- \varepsilon}^{a(j)}\big)\, \mathrm{d}B_s^{a(j)}.
   \end{align*}
   Thus, by an application of the BDG and Jensen's inequality, we obtain from the upper bound in (D2) combined with (D3):
     \begin{align}\label{eq:anisotropmomentestimateperturb}
      \notag\mathbb{E}\Big[\big|X_t^j-X_t^{\varepsilon, j}\big|^p\Big]
      &\lesssim \bigg(\int_{t-\varepsilon}^t \big|\widetilde{k}_j^{b}(t-s)\big|\, \mathrm{d}s\bigg)^{p-1}\int_{t-\varepsilon}^t \big|\widetilde{k}_j^{b}(t-s)\big|\hspace{0.02cm}\mathbb{E}\left[\big|b_s^{a(j)}-b_{t-\varepsilon}^{a(j)}\big|^p\right]\, \mathrm{d}s
      \\  \notag&\quad+ \bigg(\int_{t-\varepsilon}^t \big|\widetilde{k}_j^{\sigma}(t-s)\big|^2 \, \mathrm{d}s\bigg)^{\frac{p}{2}-1}\int_{t-\varepsilon}^t \big|\widetilde{k}_j^{\sigma}(t-s)\big|^2 \hspace{0.02cm}\mathbb{E}\left[\big|\sigma_s^{a(j)}-\sigma_{t-\varepsilon}^{a(j)}\big|^p\right]\, \mathrm{d}s
      \\ &\lesssim \varepsilon^{\big(\gamma_b^{a(j)} + \frac{1}{2}+\alpha_b^{a(j)}\big)p}+\varepsilon^{\big(\gamma_{\sigma}^{a(j)}+\alpha_{\sigma}^{a(j)}\big)p}.
   \end{align}
   For an arbitrary $\phi\in C_{b}^{\lambda,a}(\R^N)$, where $\lambda>0$ and $a=(a_j)_{j\in\{1,\dots,N\}}\in (\R_+^*)^N$ shall be chosen below, we estimate for each $j\in\{1,\dots,N\}$ and $h\in [-1,1]$:
   \begin{align*}
       \left| \mathbb{E}\big[\rho_t\hspace{0.02cm}\Delta_{he_j}\phi(X_t)\big]\right|
       &\le \left|\mathbb{E}\big[\rho_t\hspace{0.02cm}\Delta_{he_j}\phi(X_t)-\rho_{t-\varepsilon}\hspace{0.02cm}\Delta_{he_j}\phi(X_t^{\varepsilon})\big]\right| + \left|\mathbb{E}\big[\rho_{t-\varepsilon}\hspace{0.02cm}\Delta_{he_j}\phi(X_t^{\varepsilon})\big]\right|
       \\ &=:R_1+R_2,
   \end{align*}
   where $\rho_t = 1 \wedge \min_{i\in\{1,\dots,d\}}|\sigma_t^i|$. For $R_1$, we use $\rho_t \leq 1$ combined with the Lipschitz continuity of the minimum, the Hölder regularity of $\phi\in C_{b}^{\lambda,a}(\R^N)$ according to \eqref{eq:anisotropictestfctHölder}, and finally~\eqref{eq:anisotropmomentestimateperturb} to find:
   \begin{align*}
       R_1 &\le \mathbb{E}\big[|\rho_t-\rho_{t-\varepsilon}|\hspace{0.02cm}|\Delta_{he_j}\phi(X_t)|\big] + \mathbb{E}\big[\rho_{t-\varepsilon}\hspace{0.02cm}|\Delta_{he_j}\phi(X_t)-\Delta_{he_j}\phi(X_t^{\varepsilon})|\big]
       \\ &\le \|\phi\|_{C_{b}^{\lambda,a}(\R^N)}|h|^{\lambda/a_j}\sum_{n=1}^d\mathbb{E}\big[|\sigma_t^n-\sigma_{t-\varepsilon}^n|\big]+ 4\|\phi\|_{C_{b}^{\lambda,a}(\R^N)}\hspace{0.02cm}\sum_{k=1}^N \mathbb{E}\Big[\big|X_t^{k} - X_t^{\varepsilon,k}\big|^{\lambda/a_{k}}\Big]
       \\ &\lesssim \|\phi\|_{C_{b}^{\lambda,a}(\R^N)}\left(|h|^{\lambda/a_j}\varepsilon^{\alpha_{\sigma}^*} + \max_{k \in \{1,\dots, N\}}\varepsilon^{\big(\gamma_b^{a(k)}+\frac{1}{2}+\alpha_b^{a(k)}\big)\frac{\lambda}{a_k}} + \max_{k \in \{1,\dots, N\}} \varepsilon^{\big(\gamma_{\sigma}^{a(k)}+\alpha_{\sigma}^{a(k)}\big)\frac{\lambda}{a_k}}\right),
   \end{align*}
    where $\alpha_{\sigma}^{*}:=\min_{n\in\{1,\dots,d\}}\alpha_{\sigma}^n$. For the second term, observe that $U_t^{\varepsilon}$ is $\mathcal{F}_{t-\varepsilon}$-measurable, while $V_t^{\varepsilon}$ is, conditionally on $\mathcal{F}_{t-\varepsilon}$, centered Gaussian with covariance matrix 
   \begin{align*}
       \mathrm{cov}(V_t^{\varepsilon}\,|\,\mathcal{F}_{t-\varepsilon})_{\ell \ell'}
       = \delta_{a(\ell)a(\ell')}\big(\sigma_{t-\varepsilon}^{a(\ell)}\big)^2\int_{0}^\varepsilon \widetilde{k}_{\ell}^{\sigma}(s)\hspace{0.02cm}\widetilde{k}_{\ell'}^{\sigma}(s)\,\mathrm{d}s
       = \sum_{i=1}^d \1_{\{ \ell, \ell' \in S_i\}} (\sigma_{t-\varepsilon}^i)^2 \hspace{0.02cm}\widetilde{G}_{\ell \ell'}^{(i)}(\varepsilon),
   \end{align*}
   where $\ell, \ell' \in \{1,\dots, N\}$.

    For $i \in \{1,\dots, d\}$, let $P_i x = (x_j)_{j \in S_i}$ be the projection onto the $S_i$-coordinates of $x \in \R^N$. Then, up to permutation of coordinates, $V_t^{\varepsilon} = ( P_1V_t^{\varepsilon}, \dots, P_dV_t^{\varepsilon})$, where each block is given by $(P_iV_t^{\varepsilon})_{j} = \int_{t-\varepsilon}^t \widetilde{k}_j^{\sigma}(t-s)\hspace{0.02cm}\sigma_{t-\varepsilon}^{i}\, \mathrm{d}B_s^i$ with $j \in S_i$. In particular, conditionally on $\mathcal{F}_{t-\varepsilon}$, $P_1V_t^{\varepsilon}, \dots, P_dV_t^{\varepsilon}$ are mutually independent centered Gaussian random vectors with covariance matrix (up to a change of variables $\{1,\dots, |S_i|\} \simeq S_i$) given by
    \begin{align*}
        \mathrm{cov}\left( P_i V_t^{\varepsilon} \, | \, \mathcal{F}_{t-\varepsilon}\right)_{\ell \ell'} = (\sigma_{t-\varepsilon}^i)^2\int_{t-\varepsilon}^t \widetilde{k}_{\ell}^{\sigma}(t-s)\hspace{0.02cm}\widetilde{k}_{\ell'}^{\sigma}(t-s)\, \mathrm{d}s = (\sigma_{t-\varepsilon}^i)^2 \widetilde{G}_{\ell \ell'}^{(i)}(\varepsilon),
    \end{align*}
    where $\ell, \ell' \in S_i$. In particular, by $\{\rho_{t-\varepsilon}>0\}\subseteq\{|\sigma_{t-\varepsilon}^i|>0\}$ and the lower bound in assumption (D2), we obtain
   \begin{align}\label{eq:anisotroppermutedGaussianblock}
   \xi^{\intercal} \mathrm{cov}\left( P_i V_t^{\varepsilon} \, | \, \mathcal{F}_{t-\varepsilon}\right) \xi \geq C_* \hspace{0.02cm}\rho_{t-\varepsilon}^2\hspace{0.02cm}\varepsilon^{2\gamma_*^i} |\xi |^2
   \end{align}
   for all $\xi = (\xi_j)_{j \in S_i}$, where we again implicitly use $\{1,\dots, |S_i|\} \simeq S_i$. Hence, $\mathcal{L}(P_i V_t^{\varepsilon}\, | \, \mathcal{F}_{t - \varepsilon})$ has a density $f_t^{\varepsilon, i}$ with respect to the Lebesgue measure. By the conditional independence of $P_1V_t^{\varepsilon}, \dots, P_dV_t^{\varepsilon}$ on $\mathcal{F}_{t-\varepsilon}$, also $\mathcal{L}(V_t^{\varepsilon} \, | \, \mathcal{F}_{t-\varepsilon})$ has a density given by $f_t^{\varepsilon} = \prod_{i=1}^d f_t^{\varepsilon, i}$, where we use the convention $f_t^{\varepsilon, i} := 1$ if $S_i = \emptyset$, and this relation holds up to permutation of the variables.
   
   For every $i\in\{1,\dots,d\}$ and $h\in\R^{|S_i|}$, by Pinsker's inequality for the total variation norm and a direct computation for the Kullback-Leibler divergence of Gaussian densities, we can estimate on the event $\{\rho_{t-\varepsilon}>0\} \subseteq \{ |\sigma_{t-\varepsilon}^i| > 0\}$:
   \begin{align}\label{eq:anisotropicGaussiandensityestimateperturb}
       \notag \int_{\R^{|S_i|}} \left|f_t^{\varepsilon,i}(x)-f_t^{\varepsilon,i}(x-h)\right|\,\mathrm{d}x 
       &\leq \sqrt{ 2 D_{\mathrm{KL}}\big(f_t^{\varepsilon, i} \, \| \, f_t^{\varepsilon, i}(\cdot - h)\big)}
       \\ \notag &\le \hspace{0.02cm} |h| \hspace{0.02cm}\left(\lambda_{\min}\left(\mathrm{cov}\left( P_iV_t^{\varepsilon}\,|\,\mathcal{F}_{t-\varepsilon}\right)\right)\right)^{-1/2}
       \\ &\lesssim |h| \hspace{0.02cm} \rho_{t-\varepsilon}^{-1}\hspace{0.02cm}\varepsilon^{-\gamma_*^i}, 
   \end{align} 
   where the last estimate is justified by \eqref{eq:anisotroppermutedGaussianblock}. Hence, for $j \in \{1,\dots, N\}$ let $i \in \{1,\dots, d\}$ be the unique index such that $j \in S_i$, i.e.\ $i=a(j)$. Then, using the $\mathcal{F}_{t-\varepsilon}$-measurability of $U_t^{\varepsilon}$, the density $f_t^{\varepsilon}$ for $\mathcal{L}(V_t^{\varepsilon}\,|\,\mathcal{F}_{t-\varepsilon})$ on $\{\rho_{t-\varepsilon}>0\}$, a shift of the integration variable, the boundedness of $\phi\in C_{b}^{\lambda,a}(\R^N)$, and Fubini's theorem, we can estimate $R_2$ by
   \begin{align*}
       R_2 &= \left|\mathbb{E}\big[\mathbb{E}\big[\rho_{t-\varepsilon}\hspace{0.02cm}\Delta_{he_j}\phi(U_t^{\varepsilon}+V_t^{\varepsilon})\,\big|\,\mathcal{F}_{t-\varepsilon}\big]\big]\right|
       \\ &= \bigg|\mathbb{E}\bigg[\rho_{t-\varepsilon}\mathbbm{1}_{\{\rho_{t-\varepsilon}>0\}}\int_{\R^N}\big(\phi(U_t^{\varepsilon}+z+he_j)-\phi(U_t^{\varepsilon}+z)\big) \hspace{0.02cm}f_t^{\varepsilon}(z)\,\mathrm{d}z\bigg]\bigg|
       \\  &\le \mathbb{E}\bigg[\rho_{t-\varepsilon}\mathbbm{1}_{\{\rho_{t-\varepsilon}>0\}}\int_{\R^N}\big|\phi(U_t^{\varepsilon}+z)\big| \hspace{0.02cm}\big|f_t^{\varepsilon}(z)-f_t^{\varepsilon}(z-he_j)\big|\,\mathrm{d}z\bigg]
       \\ &\le \|\phi\|_{C_{b}^{\lambda,a}(\R^N)} \hspace{0.02cm}\mathbb{E}\bigg[\rho_{t-\varepsilon}\mathbbm{1}_{\{\rho_{t-\varepsilon}>0\}}
       \\ &\hskip29mm \cdot\int_{\R^N}\prod_{k\in\{1,\dots,d\}\setminus\{a(j)\}}f_t^{\varepsilon,k}( P_kz )\hspace{0.02cm}\big|f_t^{\varepsilon,a(j)}(P_{a(j)} z) - f_t^{\varepsilon,a(j)}\big(P_{a(j)} (z - he_{j})\big)\big|\,\mathrm{d}z\bigg]
       \\ &= \|\phi\|_{C_{b}^{\lambda,a}(\R^N)} \hspace{0.02cm}\mathbb{E}\bigg[\rho_{t-\varepsilon}\mathbbm{1}_{\{\rho_{t-\varepsilon}>0\}}\int_{\R^{|S_i|}}\big|f_t^{\varepsilon,a(j)}(w)-f_t^{\varepsilon,a(j)}\big(w - h\hspace{0.02cm}P_{a(j)}e_j \big)\big|\,\mathrm{d}w\bigg]\\
       &\lesssim \|\phi\|_{C_{b}^{\lambda,a}(\R^N)} \hspace{0.02cm} |h| \hspace{0.02cm} \varepsilon^{-\gamma_*^{a(j)}},
   \end{align*}
   where the last step is justified by~\eqref{eq:anisotropicGaussiandensityestimateperturb}. Therefore, combining the bounds for $R_1$ and $R_2$ and selecting $\varepsilon = (1\wedge t)\frac{h_0}{2} |h|^{\delta}$, with a parameter $\delta>0$ to be specified below, gives
   \begin{align*}
       \left| \mathbb{E}\big[\rho_t\hspace{0.02cm}\Delta_{he_j}\phi(X_t)\big]\right| 
       &\lesssim \|\phi\|_{C_{b}^{\lambda,a}(\R^N)} \bigg(|h|^{\lambda/a_j+\delta\alpha_{\sigma}^*} + \max_{k \in \{1,\dots, N\}}|h|^{\delta\big(\gamma_b^{a(k)}+\frac{1}{2}+\alpha_b^{a(k)}\big)\frac{\lambda}{a_k}} 
       \\ &\qquad \qquad  \qquad \qquad + \max_{k \in \{1,\dots, N\}} |h|^{\delta\big(\gamma_{\sigma}^{a(k)}+\alpha_{\sigma}^{a(k)}\big)\frac{\lambda}{a_k}} + |h|^{1-\delta\gamma_*^{a(j)}} \bigg).
   \end{align*}    
   For an application of \cite[Lemma 3.1]{FJR18}, it is sufficient to prove the existence of an anisotropy $a=(a_1,\dots,a_N)\in (\R_+^*)^N$, $\lambda>0$ with $(\lambda/a_1,\dots,\lambda/a_N)\in (0,1)^N$ and $\delta>0$ such that each of the four exponents above is strictly larger than $\lambda/a_j$ for each $j \in \{1,\dots, N\}$. Notice that for the first term, this is trivial as $\alpha_{\sigma}^*>0$. Moreover, the remaining three exponents yield for each $j,k\in\{1,\dots,N\}$:
   \begin{align}\label{eq:anisotropycoeffcond1perturb}
      \delta\hspace{0.02cm}\min\big\{\gamma_b^{a(k)}+\tfrac{1}{2}+\alpha_b^{a(k)},\gamma_{\sigma}^{a(k)}+\alpha_{\sigma}^{a(k)}\big\}\hspace{0.02cm}\frac{\lambda}{a_k}>\frac{\lambda}{a_j} \quad\text{and}\quad 1 - \delta \gamma_*^{a(j)} > \frac{\lambda}{a_j}.
   \end{align}
   Let us now consider the anisotropy defined by
   \[
   a_{\ell} = \gamma_*^{a(\ell)} \left(\frac{1}{N}\sum_{k=1}^N \gamma_*^{a(k)}\right)^{-1}, \qquad \ell \in\{1,\dots, N\}.
   \]
   Then condition \eqref{eq:anisotropycoeffcond1perturb} reduces to 
   \[ 
    \delta > \frac{\gamma_*^{a(k)}}{\min\big\{\gamma_b^{a(k)}+\tfrac{1}{2}+\alpha_b^{a(k)},\gamma_{\sigma}^{a(k)}+\alpha_{\sigma}^{a(k)}\big\}} \frac{1}{\gamma_*^{a(j)}}\quad\text{and}\quad \delta < \frac{1}{\gamma_*^{a(j)}}-\frac{\lambda}{\gamma_*^{a(j)} a_j},
   \]
   for all $j,k \in \{1,\dots,N\}$. As we may choose $\lambda>0$ arbitrarily small, this can be achieved as soon as
   \[
    \forall j,k\in\{1,\dots,N\}:\quad\frac{\gamma_*^{a(k)}}{\min\big\{\gamma_b^{a(k)}+\tfrac{1}{2}+\alpha_b^{a(k)},\gamma_{\sigma}^{a(k)}+\alpha_{\sigma}^{a(k)}\big\}} \frac{1}{\gamma_*^{a(j)}} < \frac{1}{\gamma_*^{a(j)}},
   \]
   which is an immediate consequence of assumption (D4). We may, therefore, select $\lambda\in(0,\min_{k\in\{1,\dots,N\}}a_k)$ sufficiently close to $0$, which also gives $(\lambda/a_1,\dots,\lambda/a_N)\in (0,1)^N$. Hence, \cite[Lemma 3.1]{FJR18} proves that the measure defined in \eqref{eq:anisotropicmeasureperturbations} is absolutely continuous with respect to the $N$-dimensional Lebesgue measure.
\end{proof}

Using \cite[Lemma 3.1]{FJR18}, one can in fact even show that the density belongs to an appropriate anisotropic Besov space $B_{1,\infty}^{\lambda,a}\big(\R^{N}\big)$. Since this refinement is not needed for the present work, we omit the details and refer the interested reader to \cite{FJR18}.

\section{Fractional forward differentiation}\label{section:fracdifferentiation}

\subsection{Admissibility}\label{subsection:fracdifferentiationadmiss}

In this section, we prove the assertions of Example \ref{admissible operator fractional differentiation}.

\begin{proof}[Proof of Example \ref{admissible operator fractional differentiation}]

 \textit{Step 1.} Let us first suppose that $\int_{\R_+}(1 + z)\,\nu_i(\mathrm{d}z) < \infty$ for every $i\in\{1,\dots,d\}$. Then $\widetilde{\Xi}$ defined via $\widetilde{\Xi} = \sum_{i=1}^d c_i \int_{\R_+} \Xi^{i} (S(z) - \mathrm{id}_{\mathcal{V}})\, \nu_i(\mathrm{d}z)$ is a bounded linear operator on the entire space~$\mathcal{V}$ since
 \begin{displaymath}
     |\widetilde{\Xi}y| \lesssim \hspace{0.03cm}\sum_{i=1}^d |c_i|\hspace{0.02cm} \| \Xi^{i} \|_{L(\mathcal{V}, \R)} \int_{\R_+} \big( 1 + \sqrt{z}\big)\, \nu_i(\mathrm{d}z) \hspace{0.02cm}\|y\|_{\mathcal{V}},
 \end{displaymath}
 where we have used that $\|S(z)\|_{L(\mathcal{V})} \lesssim 1 + \sqrt{z}$ by assumption. In particular, this implies $\widetilde{\Xi}(t)=\widetilde{\Xi}S(t) \in L(\mathcal{V}, \R)$ for all $t \geq 0$ and hence $\mathcal{D} = \mathcal{V}$. For the Volterra kernels, we obtain for $a \in \{b,\sigma\}$:
 \begin{displaymath}
     |\widetilde{k}^a(t)| \leq \sum_{i=1}^d |c_i| \| \Xi^{i} \|_{L(\mathcal{V}, \R)} \int_{\R_+} \big( 1 + \sqrt{z}\big)\, \nu_i(\mathrm{d}z)\hspace{0.02cm} \|S(t)\|_{L(\mathcal{H},\mathcal{V})} \|\xi_a\|_{L(\R^d,\mathcal{H})},
 \end{displaymath}
 and hence by \eqref{eq:abstractoperatornormestimate} from condition (B) it follows that $\widetilde{k}^b, \widetilde{k}^{\sigma} \in L_{\mathrm{loc}}^{2}(\R_+; \R^{1 \times d})$. Setting $\widetilde{\Xi}_{\lambda} := \widetilde{\Xi}$ for $\lambda\in (0,1)$, it remains to verify that $\widetilde{\Xi}$ can be approximated by linear combinations of $\Xi^{i}S(z)$ with $i \in \{1,\dots, d\}$ and $z\in\R_+$. Since by $\int_{\R_+}(1 + z)\,\nu_i(\mathrm{d}z) < \infty$, $\nu_i$ is a finite Borel measure on~$\R_+$, after normalization, we may apply \cite[Theorem 6.18]{MR2459454} to find a sequence of measures $(\nu_{i}^n)_{n\in\mathbb{N}}$ such that $\nu_{i}^n \longrightarrow \nu_{i}$ in the Wasserstein-$1$ distance, where
 \begin{displaymath}
     \nu_{i}^n = \sum_{j=1}^{N_n^{i}}a_{ij}^{(n)} \delta_{z_{ij}^{(n)}}
 \end{displaymath}
 is a finite linear combination of Dirac measures with $N_n^{i} \geq 1$, $(a_{ij}^{(n)})_{i \in \{1,\dots, d\}, j \in\{1,\dots, N_n^{i}\}} \subseteq \R$ and $(z_{ij}^{(n)})_{i \in \{1,\dots, d\}, j \in\{ 1,\dots, N_n^{i}\}} \subseteq \R_+$. Define $N_n = \max\{N_n^1, \dots, N_n^d\}$ and $a_{ij}^{(n)} = z_{ij}^{(n)} = 0$ when $j > N_n^i$ and $i \in \{1,\dots, d\}$. Then $\nu_i^n = \sum_{j=1}^{N_n}a_{ij}^{(n)}\delta_{z_{ij}^{(n)}}$.

 Since for $y \in \mathcal{V}$ and $i \in \{1,\dots, d\}$, $\Xi^{i} (S(\cdot)y - y) \in C(\R_+; \R)$ has linear growth by $\|S(z)\|_{L(\mathcal{V})} \lesssim 1 + \sqrt{z}$ and $\Xi^i \in L(\mathcal{V},\R)$, we obtain from the convergence in the Wasserstein distance:
 \begin{align*}
     \widetilde{\Xi}y = \sum_{i=1}^d c_i \int_{\R_+} \Xi^{i} (S(z)y - y)\, \nu_i(\mathrm{d}z) 
      = \lim_{n \to \infty}\sum_{i=1}^d c_i \int_{\R_+} \Xi^i(S(z)y - y)\, \nu_{i}^n(\mathrm{d}z).
 \end{align*}
 When replacing $y$ with $S(t)y \in \mathcal{V}$, we obtain the property~\eqref{eq:admissibleNlimy}, since $-c_i\hspace{0.02cm}\nu_{i}^n(\R_+)\hspace{0.02cm}\Xi^iS(t)y$ also has the structure required for admissible summands therein. For \eqref{eq:admissibleNlimKerneloperatorcase}, let us first observe that $\widetilde{\Xi}(t)\xi_b = \widetilde{k}^b(t)$ and $\widetilde{\Xi}(t)\xi_{\sigma} = \widetilde{k}^{\sigma}(t)$ with the right-hand sides given by \eqref{eq: K admissible differentiation}. Hence, we obtain for our choice of approximation:
 \begin{align} \label{eq: 11}
    & \int_0^T \left| \sum_{i=1}^d c_i \left(\sum_{j=1}^{N_n} a_{ij}^{(n)}K^a_i\big(z_{ij}^{(n)}+t\big) - K_i^{a}(t)\hspace{0.02cm}\nu_i(\R_{+})\right) - \widetilde{k}^a(t)\right|^2\, \mathrm{d}t 
    \\ &= \int_0^T \left| \sum_{i=1}^d c_i \int_{\R_+} \Xi^{i} S(z + t)\hspace{0.02cm}\xi_a\, \big(\nu_{i}^n(\mathrm{d}z) - \nu_i(\mathrm{d}z)\big)\right|^2\, \mathrm{d}t,  \notag
 \end{align}
 for $a \in \{b, \sigma\}$. By $\nu_{i}^n \longrightarrow \nu_{i}$ in the Wasserstein distance, combined with the continuity of $\Xi^{i} S(\cdot+t)\xi_a$ for each $t \in (0,T]$ and its linear growth in $z$ due to
 \begin{align}\label{eq: 10}
    |\Xi^{i} S(z+t)\xi_a| \lesssim \|\Xi^{i} \|_{L(\mathcal{V}, \R)} \left( 1 + \sqrt{z}\right) (1 + t^{-\rho}) \|\xi_a\|_{L(\R^d,\mathcal{H})},
 \end{align}
 where we have used \eqref{eq:abstractoperatornormestimate}, the inner integrals converge pointwise in $t$ to zero for every $i\in\{1,\dots,d\}$. Moreover, the entire integrand of the outer integral has an integrable majorant due to~\eqref{eq: 10}, $\rho\in[0,1/2)$, and $\sup_{n \geq 1}\int_{\R_+}(1+\sqrt{z})\, \nu_i^n(\mathrm{d}z)< \infty$, which is implied by
 \[
    \lim_{n \to \infty}\int_{\R_+}(1+\sqrt{z})\, \nu_i^n(\mathrm{d}z) = \int_{\R_+}(1+\sqrt{z})\, \nu_i(\mathrm{d}z) < \infty,
 \]
 following from the convergence in the Wasserstein distance. Thus, by dominated convergence, \eqref{eq: 11} converges to zero as $n\to\infty$ for fixed $T > 0$. This shows the admissibility for the case with $\int_{\R_+}(1 + z)\, \nu_i(\mathrm{d}z) < \infty$ for every $i\in\{1,\dots,d\}$.

 \textit{Step 2.} Let us now prove the general case where $\int_{\R_+}(1\wedge z)\hspace{0.03cm}e^{-\lambda z}\, \nu_i(\mathrm{d}z) < \infty$ for all $i \in \{1,\dots, d\}$ and each $\lambda > 0$. Define for fixed $\lambda > 0$ and $i\in\{1,\dots,d\}$ the operator
 \[
    \widetilde{\Xi}_{\lambda}^{i} = \int_{\R_+} \Xi^{i} (S(z) - \mathrm{id}_{\mathcal{V}})\, e^{-\lambda z}\1_{(\lambda, \infty)}(z)\,\nu_i(\mathrm{d}z).
 \]
 Since $\nu^{\lambda}_{i}(\mathrm{d}z):=e^{-\lambda z}\1_{(\lambda, \infty)}(z)\,\nu_i(\mathrm{d}z)$ satisfies $\int_{\R_+}(1+z)\,\nu^{\lambda}_i(\mathrm{d}z) < \infty$ by construction, it follows from step 1 that $\widetilde{\Xi}_{\lambda}^{i} \in L(\mathcal{V},\R)$. In particular, step 1 proves that the family $(\widetilde{\Xi}_{\lambda})_{\lambda\in (0,1)}$ defined by $\sum_{i=1}^dc_i\hspace{0.02cm}\widetilde{\Xi}_{\lambda}^{i}$ fulfills both \eqref{eq:admissibleNlimy} and \eqref{eq:admissibleNlimKerneloperatorcase} in Definition \ref{def:admissibleoperator}, where \eqref{eq:admissibleNlimy} holds even for all $y\in\mathcal{V}$. Moreover, when passing to the limit $\lambda\to 0$, we obtain for each $y \in \mathcal{D}$ and $t>0$ from the dominated convergence theorem and $\nu_i(\{0\}) = 0$:
 \begin{align*}
    \widetilde{\Xi}_{\lambda}S(t)y = \sum_{i=1}^d c_i \int_{\R_+} \Xi^{i} S(t)(S(z)y - y)\, e^{-\lambda z} \1_{(\lambda, \infty)}(z)\, \nu_i(\mathrm{d}z) \longrightarrow \widetilde{\Xi}(t)y, \qquad \mbox{as}\ \ \lambda \rightarrow 0.
 \end{align*}
 Likewise, an iterated application of the dominated convergence theorem in combination with \eqref{eq:Kintegrabilityadmissibledifferentiation} gives for $a \in \{b,\sigma\}$:
 \begin{align*}
     &\int_0^T \left| \widetilde{\Xi}_{\lambda}S(t)\xi_a - \widetilde{k}^a(t) \right|^2\, \mathrm{d}t
     \\ &\qquad  \leq \int_0^T \left( \sum_{i=1}^d |c_i| \int_{\R_+} \big|K^a_i(z+t) - K^a_i(t)\big| \big(1 - e^{-\lambda z} \1_{(\lambda, \infty)}(z) \big)\, \nu_i(\mathrm{d}z)\right)^2\, \mathrm{d}t \longrightarrow 0,
 \end{align*}
 as $\lambda \rightarrow 0$, which shows $\widetilde{\Xi}_{\lambda}S(\cdot)\xi_a \longrightarrow \widetilde{k}^a$ in $L^2([0,T];\R^{1\times d})$ for every $T>0$, i.e.\ the second condition in \eqref{eq:admissiblelamdalim}. This proves all assertions. 
\end{proof}

\subsection{Stability for regularly varying functions}

Below, we prove that the forward differentiation rule preserves the class of regularly varying functions.

\begin{lemma}\label{lemma: appendix B}
    Let $k \in C^1(\R_+^*) \cap L_{\mathrm{loc}}^2(\R_+; \R)$ be regularly varying in $t=0$ with index $\rho < 1$, and~$k'$ being monotone in a right-neighborhood of $t = 0$. Suppose that there exists $\beta \geq 0$, $h_0 \in (0,1)$, and $C > 0$ such that
    \begin{align}\label{eq: bounded away}
        \left| \frac{k(tu)}{k(t)}\right| \leq Cu^{\beta},\qquad t \in (0,h_0], \ u \geq 1.
    \end{align} 
    Define for $\alpha\in (\beta, 1)$ the function
        \[ 
        \widetilde{k}(t) = \int_{0}^{\infty}\big( k(t+z) - k(t)\big)\,\frac{e^{-z}\mathrm{d}z}{z^{1 +  \alpha}}, \qquad t > 0.
        \]
    Then there exists an explicit constant $C(\alpha) \in \R$ such that
    \begin{align}\label{eq: B appendix regular varying}
        \lim_{t \searrow 0} \frac{\widetilde{k}(t)}{t^{-\alpha}k(t)} = C(\alpha).
    \end{align} 
    In particular, if $\rho > - \frac{1}{2}$ and $\alpha\in(\beta, \min\{\rho+\frac{1}{2},1\})$, then $\widetilde{k} \in L_{\mathrm{loc}}^2(\R_+; \R)$ and for every $T>0$: 
    \begin{align}\label{eq: B last}
        \int_0^T \left( \int_{0}^{\infty} |k(t+z) - k(t)|\, \frac{e^{-z}\,\mathrm{d}z}{z^{1+\alpha}} \right)^2\, \mathrm{d}t < \infty.
    \end{align} 
\end{lemma}
\begin{proof}
    Remark that by assumption \eqref{eq: bounded away}, for each $t_0 > 0$ there exists some $C_{t_0} > 0$ such that $|k(t)| \leq C_{t_0}t^{\beta}$ holds for $t \geq t_0$. Indeed, without loss of generality we may suppose that $t_0 \in (0,h_0)$. Then
    \[
        |k(t)| = \left| \frac{k\big( \frac{t}{t_0} t_0\big)}{k(t_0)}\right| \cdot |k(t_0)| \leq C |k(t_0)| \left( \frac{t}{t_0}\right)^{\beta}.
    \]
    
    Let us first show that the integral in the definition of $\widetilde{k}$ converges for each $t>0$. For $z \in (0,1]$ we use the mean-value theorem to find $\xi_z \in [t,t+1]$ such that $k(t+z) - k(t) = k'(\xi_z)z$. Hence, using the conclusion after \eqref{eq: ki derivative 2}, we find $|k'(\xi_z)| \lesssim_{\varepsilon} \xi_z^{\rho - 1 - \varepsilon} \lesssim t^{\rho - 1 - \varepsilon}$ for each $\varepsilon > 0$ by $\rho <1$. Further, by assumption we have $|k(r)| \leq C_t\hspace{0.02cm} r^{\beta}$ for $r \geq t$, and hence 
    \begin{align}\label{eq:ktildeforwardfracbound}
        \notag\big|\widetilde{k}(t)\big| &\leq \int_{0}^{1}\left| k(t+z) - k(t)\right| \, \frac{e^{-z}\,\mathrm{d}z}{z^{1 + \alpha}}
        + \int_1^{\infty} \left| k(t+z) - k(t)\right|\, \frac{e^{-z}\,\mathrm{d}z}{z^{1 + \alpha}}
        \\ &\lesssim_{\varepsilon,t}  \int_0^1 \frac{e^{-z}\mathrm{d}z}{z^{\alpha}} + \int_1^{\infty} \left(  (t+z)^{\beta} + t^{\beta} \right)\, \frac{\mathrm{d}z}{z^{1 + \alpha}} < \infty
    \end{align}
    for all $\alpha\in(\beta, 1)$ and $t>0$.

    Next, we prove \eqref{eq: B appendix regular varying}. By assumption, we can write $k(t) = t^{\rho}L(t)$, where $\ell(t) = L(1/t)$ is slowly varying. Using the substitution $z = t u$, we obtain
    \begin{align*}
        \frac{\widetilde{k}(t)}{t^{\rho - \alpha}L(t)} &= \frac{t^{-\alpha}}{t^{\rho - \alpha}L(t)}\int_{0}^{\infty} \big( k(t(1+u)) - k(t) \big)\, \frac{e^{-tu}\, \mathrm{d}u}{u^{1 + \alpha}}
        \\ &= \int_0^{\infty} \left( (1+u)^{\rho}\hspace{0.02cm}\frac{L(t(1+u))}{L(t)} - 1 \right)\, \frac{e^{-tu}\, \mathrm{d}u}{u^{1 + \alpha}}.
    \end{align*}
    For each fixed $u \in (0,\infty)$, the integrand converges as $t \searrow 0$. 

    To apply dominated convergence, note that for $u \geq 1$ and $t \in (0,h_0)$, we obtain by assumption \eqref{eq: bounded away}:
    \begin{align*}
        \left|(1+u)^{\rho}\hspace{0.02cm}\frac{L(t(1+u))}{L(t)}\right|
        = \left|\frac{k(t(1+u))}{k(t)}\right|
        \leq C (1+u)^{\beta}
    \end{align*}
    and hence
    \[
       \int_1^{\infty} \left| (1+u)^{\rho}\hspace{0.02cm}\frac{L(t(1+u))}{L(t)} - 1 \right|\, e^{-tu} \frac{\mathrm{d}u}{u^{1+\alpha}} 
        \lesssim \int_1^{\infty} \left((1+u)^{\beta} + 1 \right) \, \frac{\mathrm{d}u}{u^{1+\alpha}}
    \]
    which is finite due to $\alpha >\beta$. For $u \in [0,1]$, we use again the mean-value theorem to find $\xi \in [1,2]$ such that $k(t(1+u)) - k(t) = tu \hspace{0.02cm}k'(t\xi)$. By an application of the monotone density theorem similarly to \eqref{eq: ki derivative 1} and \eqref{eq: ki derivative 2}, we obtain 
    \begin{align}\label{eq: appendix 2}
        \sup_{\xi \in [1,2]}\sup_{t \in (0,h_0/2)} \left| \frac{t\xi\hspace{0.02cm} k'(t\xi)}{k(t\xi)} \right|=\sup_{r\in (0,h_0]} \left| \frac{r k'(r)}{k(r)} \right| < \infty.
    \end{align} 
    Hence, using \eqref{eq: bounded away}, we obtain for $t\in (0,h_0/2)$: 
    \begin{align*}
        \left|(1+u)^{\rho}\hspace{0.02cm}\frac{L(t(1+u))}{L(t)} - 1 \right| 
        &= \left|\frac{k(t(1+u))}{k(t)} - 1 \right| 
        \\ &= u \left| \frac{t k'(t\xi)}{k(t\xi)} \right| \cdot\left| \frac{k(t\xi)}{k(t)} \right|  
        \lesssim u\cdot \sup_{r \in (0, h_0]}\left| \frac{r k'(r)}{k(r)} \right| \cdot \xi^{\beta-1} 
        \lesssim u.
    \end{align*}
    Due to $\alpha<1$, this yields
    \begin{align*}
        \int_0^{1} \left| (1+u)^{\rho}\hspace{0.02cm}\frac{L(t(1+u))}{L(t)} - 1 \right|\, e^{-tu} \frac{\mathrm{d}u}{u^{1+\alpha}} 
        &\lesssim \int_0^1 \, \frac{\mathrm{d}u}{u^{\alpha}} < \infty.
    \end{align*}
    By dominated convergence, we obtain
    \begin{align*}
        \lim_{t \searrow 0}\int_0^{\infty} \left( (1+u)^{\rho}\hspace{0.02cm}\frac{L(t(1+u))}{L(t)} - 1 \right)\, \frac{e^{-tu}\, \mathrm{d}u}{u^{1 + \alpha}} 
        = \int_0^{\infty} \big( (1+u)^{\rho} - 1 \big)\, \frac{\mathrm{d}u}{u^{1 + \alpha}}, 
    \end{align*}
     where the latter is finite by the mean value theorem for small $u$ and $\max\{0,\rho\} < \alpha$ for large $u$, which follows from $\beta\in[0,\alpha)$ and $\rho\le\beta$ implied by \eqref{eq: bounded away} and $k$ being regularly varying in $t=0$. 
     
     Finally, repeating the above argument for 
     \[ 
        \overline{k}(t) := \int_{0}^{\infty}\big| k(t+z) - k(t)\big|\hspace{0.02cm}\frac{e^{-z}\mathrm{d}z}{z^{1 +  \alpha}}, \qquad t > 0,
        \]
    proves that also $\overline{k}$ satisfies \eqref{eq: B appendix regular varying} for a constant $\overline{C}(\alpha)\in\R$. According to \eqref{eq: slowly varying bound} and assuming $\alpha<\rho+1/2$, there exists $\overline{h}_0>0$ such that $\overline{k}\in L^2([0,\overline{h}_0];\R)$. Moreover, for all $T\ge \overline{h}_0$ and each $t\in [\overline{h}_0,T]$ the same estimate as in \eqref{eq:ktildeforwardfracbound} yields  
    \begin{align*}
        \big|\overline{k}(t)\big|  \le \sup_{r\in [\overline{h}_0,T+1]} |k'(r)|\int_0^1 \frac{e^{-z}\mathrm{d}z}{z^{\alpha}} + \int_1^{\infty} C_{\overline{h}_0}\left( (t+z)^{\beta} + t^{\beta} \right)\, \frac{\mathrm{d}z}{z^{1 + \alpha}}=:C(\overline{h}_0,T) < \infty,
    \end{align*}
    where we utilized $\beta\in[0,\alpha)$. Consequently, we have $\overline{k} \in L_{\mathrm{loc}}^2(\R_+; \R)$, which implies condition \eqref{eq: B last}. As $\big|\widetilde{k}\big|\le \overline{k}$, this proves also $\widetilde{k} \in L_{\mathrm{loc}}^2(\R_+; \R)$.
\end{proof}

\end{appendices}

\begin{footnotesize}
\bibliographystyle{siam}
\bibliography{literature}
\end{footnotesize}

\end{document}